\documentclass[10pt,a4paper,twoside,reqno]{amsart}

\usepackage{amssymb}
\usepackage{mathtools}
\usepackage[shortlabels]{enumitem}
\usepackage[utf8]{inputenc}
\usepackage{color}
\usepackage{microtype}
\usepackage{mathabx}
\usepackage{tikz}
\usetikzlibrary{positioning,decorations.pathreplacing,decorations.markings,arrows.meta,calc,fit,quotes,cd,math,arrows,backgrounds,shapes.geometric}
\usepackage[all,2cell]{xy}\UseAllTwocells\SilentMatrices
\usepackage[a4paper,top=3cm,bottom=3cm,inner=2.5cm,outer=2.5cm]{geometry}
\definecolor{darkgreen}{rgb}{0,0.45,0}
\usepackage[colorlinks,citecolor=darkgreen,urlcolor=gray,final,hyperindex,linktoc=page,hyperfootnotes=true]{hyperref}
\usepackage[capitalize]{cleveref}
\usepackage{mathrsfs}
\usepackage{comment}

\crefname{equation}{}{}
\crefname{thm}{Theorem}{Theorems}
\crefname{defi}{Definition}{Definitions}
\crefname{rmk}{Remark}{Remarks}
\crefname{prop}{Proposition}{Propositions}
\crefname{ex}{Example}{Examples}
\crefformat{footnote}{#2\footnotemark[#1]#3}

\theoremstyle{plain}
\newtheorem{thm}{Theorem}[subsection]
\newtheorem{cor}[thm]{Corollary}
\newtheorem{lem}[thm]{Lemma}
\newtheorem{prop}[thm]{Proposition}

\theoremstyle{remark}
\newtheorem{rmk}[thm]{Remark}
\newtheorem{ex}[thm]{Example}

\theoremstyle{definition}
\newtheorem{defi}[thm]{Definition}

\definecolor{mypurple}{rgb}{0.5, 0.0, 0.5}

\tikzstyle{start}=[to path={(\tikztostart.#1) -- (\tikztotarget)}]
\tikzstyle{end}=[to path={(\tikztostart) -- (\tikztotarget.#1)}]

\newcommand{\ca}{\mathcal}
\newcommand{\dc}{\mathbb}
\newcommand{\nc}{\mathsf}
\newcommand{\Gr}{\mathfrak}

\newcommand{\B}{\mathbf}
\newcommand{\wc}{\widecheck}
\newcommand{\wh}{\widehat}

\newcommand{\mi}{\textrm{-}}
\newcommand{\ot}{\otimes}
\newcommand{\ob}{\ensuremath{\mathrm{ob}}}
\newcommand{\Set}{\nc{Set}}
\newcommand{\Cat}{\nc{Cat}}
\newcommand{\VCat}{\ca{V}\textrm{-}\nc{Cat}}
\newcommand{\VCocat}{\ca{V}\textrm{-}\nc{Cocat}}
\newcommand{\VGrph}{\ca{V}\textrm{-}\nc{Grph}}
\newcommand{\Coalg}{\nc{Coalg}}
\newcommand{\Alg}{\nc{Alg}}

\newcommand{\Mon}{\nc{Mon}}
\newcommand{\Comon}{\nc{Comon}}
\newcommand{\Mnd}{\nc{Mnd}}
\newcommand{\Cmd}{\nc{Cmd}}
\newcommand{\Mod}{\nc{Mod}}
\newcommand{\Comod}{\nc{Comod}}
\newcommand{\VMod}{\ca{V}\mi\nc{Mod}}

\newcommand{\VComod}{\ca{V}\mi\nc{Comod}}

\newcommand{\op}{\mathrm{op}}
\newcommand{\ev}{\mathrm{ev}}
\newcommand{\id}{\mathrm{id}}
\newcommand{\Hom}{\mathrm{Hom}}
\newcommand{\HOM}{\textsc{Hom}}
\newcommand{\simrightarrow}{\xrightarrow{\raisebox{-3pt}[0pt][0pt]{\ensuremath{\sim}}}}
\tikzset{tick/.style={postaction={decorate,decoration={markings,mark=at position 0.5 with {\draw[-] (0,.5ex) -- (0,-.5ex);}}}}}
\tikzset{bul/.style={postaction={decoration={markings,mark=at position 0.5 with {\node{$\sbul$};}},decorate}}}

\tikzset{tick/.style={postaction={decorate,decoration={markings,mark=at position 0.5 with {\draw[-] (0,.4ex) -- (0,-.4ex);}}}}}
\newcommand{\tickar}{\begin{tikzcd}[baseline=-0.5ex,cramped,sep=small,ampersand replacement=\&]{}\ar[r,tick]\&{}\end{tikzcd}}
\newcommand{\bular}{\begin{tikzcd}[baseline=-0.5ex,cramped,sep=small,ampersand replacement=\&]{}\ar[r,bul]\&{}\end{tikzcd}}
\newcommand{\VMMat}{\ca{V}\mi\dc{M}\nc{at}}
\newcommand{\sbul}{\scriptstyle\bullet}
\newcommand{\tick}{\object@{|}}
\newcommand{\VMat}{\ca{V}\mi\nc{Mat}}
\newcommand{\PProf}{\dc{P}\nc{rof}}
\newcommand{\VPProf}{\ca{V}\textrm{-}\PProf}
\newcommand{\matr}[3]{\SelectTips{eu}{10}\xymatrix@C=.2in{#1\colon #2\ar[r]|-{\object@{|}} & #3}}
\newcommand{\Two}{\scriptstyle\Downarrow}
\newcommand{\Cart}{\mathrm{Cart}}
\newcommand{\Lim}{\mathrm{Lim}}
\newcommand{\Cocart}{\mathrm{Cocart}}

\newcommand{\K}{K}
\newcommand{\LL}{L}
\newcommand{\ullimit}{\ar[dr,phantom,very near start,"\lrcorner" description]}

\newcommand{\tboxed}[2][inner sep=1pt]{\ifmmode
\tikz[baseline=(X.base),outer
sep=0pt]{\node[draw,#1](X){\ensuremath{#2}};}
\else
\tikz[baseline=(X.base),outer
sep=0pt]{\node[draw,#1](X){#2};}
\fi
}
\newcommand{\circled}[2][inner sep=1pt]{\ifmmode
\tikz[baseline=(X.base),outer sep=0pt]{\node[circle,draw,#1](X){\ensuremath{#2}};}
\else
\tikz[baseline=(X.base),outer sep=0pt]{\node[circle,draw,#1](X){#2};}
\fi
}

\newcommand{\bultwocell}[9]{\begin{tikzcd}[ampersand replacement=\&,sep=.3in]
#1\ar[r,bul,"{#2}"]\ar[d,"{#8}"']\ar[dr,phantom,"\Two{#9}"] \& #3\ar[d,"{#4}"] \\
#7\ar[r,bul,"{#6}"'] \& #5
\end{tikzcd}}
\newcommand{\ticktwocell}[9]{\begin{tikzcd}[ampersand replacement=\&,sep=.3in]
#1\ar[r,tick,"{#2}"]\ar[d,"{#8}"']\ar[dr,phantom,"\Two{#9}"] \& #3\ar[d,"{#4}"] \\
#7\ar[r,tick,"{#6}"'] \& #5
\end{tikzcd}}
\newcommand{\globtwocell}[7]{\begin{tikzcd}[ampersand replacement=\&]
#1\ar[r,bul,"{#2}"]\ar[d,equal]\ar[dr,phantom,"\Two{#7}"] \& #3\ar[d,equal] \\
#6\ar[r,bul,"{#5}"'] \& #4
\end{tikzcd}}

\begin{document}
\title{Sweedler theory for double categories}

\author{Vasileios Aravantinos-Sotiropoulos}
\address{School of Applied Mathematical and Physical Sciences, National Technical University of Athens, Greece}
\email{vassilarav@yahoo.gr}

\author{Christina Vasilakopoulou}
\address{School of Applied Mathematical and Physical Sciences, National Technical University of Athens, Greece}
\email{cvasilak@math.ntua.gr}

\begin{abstract}
In this work, we establish certain enrichments of dual algebraic structures in the setting of monoidal double categories. In more detail, we obtain a tensored and cotensored enrichment of monads in comonads, as well as a tensored and cotensored enrichment of modules in comodules, under very general conditions on the surrounding double category. These include ``monoidal closedness'' and ``local presentability'', leading classical notions which are here introduced in the double categorical context. Furthermore, we show that in this setting, the actual fibration of modules over monads is itself enriched in the opfibration of comodules over comonads.  
Applying this abstract double categorical framework to the setting of $\ca{V}$-matrices, one directly obtains a many-object generalization of the known enrichment of modules over monoids in comodules over comonoids in a monoidal category $\ca{V}$, which was originally induced by Sweedler's measuring $k$-coalgebras in vector spaces. In the present setting the result involves $\ca{V}$-enriched modules of $\ca{V}$-categories and $\ca{V}$-enriched comodules of $\ca{V}$-cocategories.
\end{abstract}

\maketitle

\setcounter{tocdepth}{2}
\tableofcontents

\section{Introduction}

Sweedler's measuring coalgebras \cite{Sweedler} were the starting point for establishing a framework in which dual algebraic structures like algebras and coalgebras, modules and comodules 
in monoidal categories are enriched in one another. In more detail, for two $k$-algebras $A$ and $B$, their universal measuring $k$-coalgebra $P(A,B)$ is defined via the property that any $k$-coalgebra map $C\to P(A,B)$ bijectively corresponds, in a natural way, to a $k$-algebra map $A\to\Hom_k(C,B)$ into the convolution $k$-algebra. Already part of folklore for the category of vector spaces, and formally obtained in \cite{mine,Measuringcomonoid} in any symmetric monoidal closed and locally presentable category $\ca{V}$, universal measuring comonoids are the hom-objects of an enrichment of monoids in comonoids which is furthermore tensored and cotensored -- see also \cite{Foxuniversal}. Notably, \cite{AnelJoyal} studies the special case of differential graded vector spaces in an effort to approach Koszul duality for (co)algebras \cite{AlgebraicOperads}, and exhibits how various functors involved in the enrichment (collectively called \emph{Sweedler theory}) relate to the bar-cobar adjunctions.

On the other hand, measuring comodules \cite{Batchelor} are defined in an analogous way to measuring coalgebras: for two $k$-modules $M,N$ and a $k$-comodule $X$, we have a natural bijection between $k$-module morphisms $M\to \Hom_k(X,N)$ and $k$-comodule morphisms $X\to Q(M,N)$ into the universal measuring $k$-comodule of $M$ and $N$. These are also the hom-objects of a tensored and cotensored enrichment of a (global) category of modules over monoids in a category of comodules over comonoids \cite{Measuringcomodule}, under the same assumptions on the ambient monoidal category as before. Moreover, the natural structures of fibration of modules over monoids and opfibration of comodules over comonoids are compatible with said enrichments, providing an illustative example of an enriched fibration structure \cite{EnrichedFibration}. 
Pictorially, 
\begin{equation}\label{pic:1}
	\begin{tikzcd}[column sep=.5in]
		\Mod\ar[r,dashed,"\mathrm{enriched}"] \ar[d,"\mathrm{fibred}"'] & \Comod \ar[d,"\mathrm{opfibred}"] \\
		\Alg\ar[r,dashed,"\mathrm{enriched}"'] & \Coalg
	\end{tikzcd}
\end{equation}

In \cite{VCocats}, a double categorical framework is proposed as the appropriate one for generalizing Sweedler theory for (co)algebras to its many-object setting. 
More specifically, it is shown therein that $\ca{V}$-categories are tensored and cotensored enriched in a category of $\ca{V}$-\emph{cocategories} -- structures equipped with cocomposition
arrows of the form $C(x,z)\to\sum_y C(x,y)\otimes C(y,z)$, which are also of independent interest \cite{Ainfinityalgebras}. These results manifest themselves in a double category of enriched matrices $\dc{D}=\VMMat$ (e.g. \cite{Varthrenr} for the corresponding bicategory), using double categorical monads (\cite{Monadsindoublecats}) which in this case are precisely $\ca{V}$-enriched categories. In the general setting, an enrichment of monads in an appropriate category of \emph{comonads} is therein established when $\dc{D}$ is a monoidal, fibrant (\cite{Framedbicats}) and moreover \emph{locally closed monoidal} double category -- asking that $\dc{D}_0$ and $\dc{D}_1$ are monoidal closed in a coherent way. 
This story ultimately results in an enriched fibration
\begin{displaymath}
	\begin{tikzcd}[column sep=.5in]
		\VCat\ar[r,dashed,"\mathrm{enriched}"] \ar[d,"\mathrm{fibred}"'] & \VCocat \ar[d,"\mathrm{opfibred}"] \\
		\Set\ar[r,dashed,"\mathrm{enriched}"'] & \Set
	\end{tikzcd}
\end{displaymath}
where the bottom arrow is simply due to the cartesian closure of $\Set$.

The present paper began as a natural continuation of the many-object generalization of Sweedler theory, now for modules and comodules. The end result concerns $\ca{V}$-modules for $\ca{V}$-categories in two flavors, both special cases of classical $\ca{V}$-bimodules \cite{Lawvereclosedcats}. One the one hand, one-sided single-indexed $\ca{V}$-modules $M$ of $\ca{V}$-categories $A$, equipped with action arrows of the form $A(x,y)\otimes M(y)\to M(x)$ for all $x,y\in\mathrm{obj}(A)$; these are ($A$,$\mathbf{1}$)-bimodules for $\mathbf{1}$ the terminal category.  On the other hand, one-sided double-indexed $\ca{V}$-modules of $\ca{V}$-categories, equipped with action arrows of the form $A(x,y)\otimes M(y,z)\to M(x,z)$ for all $x,y\in\mathrm{obj}(A)$ and $z\in Z$ for some set $Z$; these are ($A$,$\mathrm{disc}(Z)$)-bimodules for $\mathrm{disc}(Z)$ the discrete $\ca{V}$-category on the set $Z$.
Such structures naturally arise as particular instances, in $\dc{D}=\VMMat$, of the newly introduced \emph{modules} in double categories.
We note that, although \emph{bimodules} in fibrant double categories were considered e.g. in \cite[Sec.~11]{Framedbicats}, we here investigate one-sided modules as the analogue for modules over monoids in a monoidal category.

What started as developing the next many-object step, eventually led to a reexamination and reworking of the framework of the earlier \cite{VCocats}. Indeed, our first main result (\cref{thm:big1}) establishes a tensored and cotensored enrichment of monads in comonads in a double category under broader and more natural assumptions than the earlier version \cite[Thm.~3.24]{VCocats}, which had two technical conditions needed to be checked by hand in each specific double category of interest.  To that end, current \cref{prop:MonHcartesian} confirms that a certain internal hom-induced functor preserves cartesian liftings in any \emph{monoidal closed}\footnote{``Locally'' is henceforth dropped from our terminology, and the definition is also slightly altered --  see \cref{def:locclosed} and surrounding discussion.} double category, resolving the first technical condition. Furthermore, by introducing the notion of a \emph{locally presentable} (fibrant) double category in \cref{sec:lpdoublecats}, various ordinary categories involved in the theory end up being locally presentable which in paricular resolves the second condition. 

This results in a more efficient and theoretically elegant high-level approach, where general conditions on the double category replace more specialized ones which involve performing technical calculations in each example application. Moreover, and perhaps more importantly, the notions of monoidal closed and locally presentable categories both constitute central and well-studied properties in the context of ordinary category theory. It is therefore natural to expect that double-categorical counterparts, such as the ones introduced herein, will play a significant role in the development of double category theory and lead to applications in their own right, independently of our current storyline.

Back to Sweedler theory, the second main result of the paper (\cref{thm:big2}) establishes, in a very analogous spirit, a tensored and cotensored enrichment of modules in comodules in a double category which is fibrant, braided monoidal closed and locally presentable. Moreover, under the same assumptions on $\dc{D}$, the fibration of modules over monads $\Mod(\dc{D})\to\Mnd(\dc{D})$ is enriched in the monoidal opfibration of comodules over comonads $\Comod(\dc{D})\to\Cmd(\dc{D})$, settling the analogue of \cref{pic:1} from monoidal to double categories. 
As a special case of this theorem, for $\dc{D}=\VMMat$ we obtain an enrichment of the fibration of $\ca{V}$-modules over $\ca{V}$-categories in the opfibration of $\ca{V}$-comodules over $\ca{V}$-cocategories. We note that the desired enrichments in the general double categorical setting are induced via \emph{actions} of monoidal categories, using standard results originating from an equivalence between ``closed representations'' and tensored enriched categories \cite{EnrThrVar}. 
Naturally, before reaching the final enrichment result, the general theory of the newly introduced modules and comodules in double categories is developed: particular emphasis is given to (co)monadicity results (\cref{prop:Mod(D)monadicComod(D)comonadic}) and (op)fibration structures (\cref{prop:ModfibredoverMnd,prop:sourcetargetbifib}), as well as monoidality (\cref{prop:ModComodmonoidal,prop:IModmonoidal}) and local presentability of certain fibers (\cref{_AMod_CComodLocallyPresentable}). We mention in passing that local presentability of categories of (co)monads and (co)modules in double categories are part of ongoing work in progress.

Finally, it should be stressed that such a thorough development of the theory regarding enrichment of (co)monads and (co)modules in double categories is of course anticipated to be applied to different settings as well, by changing the surrounding double category. 
One example of interest is the fruitful context of operads and their modules, which is part of ongoing work in progress \cite{MonoidalKleisli}. Considering the double category $\dc{D}=\ca{V}\textrm{-}\mathbb{S}\mathsf{ym}$ of $\ca{V}$-\emph{symmetric sequences} whose monads are colored symmetric $\ca{V}$-operads \cite{GambinoJoyal}, one can obtain similar enrichment results that involve (co)operads and their (co)modules -- broadly in a direction that analogously approaches bar-cobar duality for operads using abstract machinery \cite{Aneltalk}. 
One particular point, however, that shows that this shall not be a mere application of \cref{thm:big1,thm:big2} is that the double category $\ca{V}\textrm{-}\mathbb{S}\mathsf{ym}$ does not satisfy all assumptions: for example, it is an \emph{oplax monoidal} rather than monoidal double category, by \cite[Thm.~10.9]{MonoidalKleisli}. Thus, the general theory needs to be appropriately adjusted to cover richer examples, which is planned as future work.

\subsubsection*{Outline} \cref{sec:preliminaries} covers preliminaries which are largely known results, with certain convenient reformulations. 
\cref{sec:moncomondouble} contains some background on double categories, and revisits the context of monads and comonads but in a broader setting than \cite{VCocats}. New material
includes parallel limits and colimits (\cref{def:parallelcolimits}), monoidal closed double categories (\cref{def:locclosed} and related theory), locally presentable double categories and the enrichment of monads in comonads (\cref{sec:lpdoublecats} and preceding results from \cref{sec:monadscomonadsdouble}). Finally, \cref{sec:modscomods} contains new material and develops the theory of modules and comodules in double categories, reaching the tensored and cotensored enrichment of modules in comodules.

\subsubsection*{Acknowledgements}
The authors would like to thank Ji\v{r}\'{\i} Ad\'{a}mek, Natahanael Arkor, John Bourke, Richard Garner, Claudio Hermida, 
Panagis Karazeris, Steve Lack, Ignacio L{\'o}pez Franco, Susan Niefield, Bob Par\'e, Mike Shulman and Ross Street for helpful discussions. 

\section{Preliminaries}\label{sec:preliminaries}

In this section, we recall some known concepts and constructions that will be needed in what follows, drawn from various relevant topics such as 
monads and modules in bicategories, enrichment, fibrations and the so-called Sweedler theory for monoidal categories. The purpose is to make the material as self-contained 
as possible, providing references for further reading wherever appropriate.

\subsection{Monads and modules}\label{sec:monmodbicats} \label{sec:monoidsmodules}

We assume that the reader is familiar with the basic theory of bicategories, see e.g. \cite{Benabou,Review}, as well as monoidal categories, see 
e.g. \cite{BraidedTensorCats}. We first recall the notions of (co)monads and their (co)modules in a bicategory.

\begin{defi}\label{monadbicat}
	A \emph{monad} in a bicategory $\ca{K}$ consists of an object $X$ together with an endomorphism
	$a:X\to X$ and 2-cells $\eta:1_X\Rightarrow a$, $\mu:a\circ a\Rightarrow a$ called the \emph{unit} and \emph{multiplication},
	such that the following diagrams commute:
	\begin{displaymath}
	\begin{tikzcd}
(a\circ a)\circ a\ar[rr,"\sim"] \ar[d,"\mu\circ1"'] && a\circ(a\circ a)\ar[d,"1\circ \mu"] \\
		a\circ a\ar[dr,"\mu"] && a\circ a\ar[ld,"\mu"] \\
		& a &	 
	\end{tikzcd}\qquad\mathrm{and}\qquad
\begin{tikzcd}[sep=.55in]
1_X\circ a\ar[r,"\eta\circ 1"]\ar[dr,sloped,"\sim"'] & a\circ a \ar[d,"\mu"] & a\circ1_X \ar[l,"1\circ\eta"']\ar[ld,sloped,"\sim"'] \\
& a & 
\end{tikzcd}
\end{displaymath}
	A \emph{(lax) monad functor} between two monads $a:X\to X$ and $b:Y\to Y$ in a bicategory $\ca{K}$
	consists of a 1-cell $f:X\to Y$ between the 0-cells of the monads together with a 2-cell $bf\Rightarrow fa$
	satisfying compatibility conditions with multiplications and units.
\end{defi}

Together with the appropriate transformations between lax monad functors, we have the bicategory $\Mnd(\ca{K})$ of all monads in $\ca{K}$. Any 
lax functor between bicategories $F\colon\ca{K}\to\ca{L}$ induces a lax functor $\Mnd(\ca{K})\to\Mnd(\ca{L})$, or in other words lax functors 
preserve monads. One way of explaining this is to observe that a monad in $\ca{K}$ can be equivalently described as a lax functor 
$\bf{1}\to\ca{K}$, where $\bf{1}$ is the terminal bicategory with a unique object -- and $\Mnd(\ca{K})$ is in fact $[\B{1},\ca{K}]_\mathrm{lax}$.

\begin{defi}\label{def:modulebicat}
	Let $(a\colon X\to X,\eta,\mu)$ be a monad in the bicategory $\ca{K}$. A \emph{(left) module} over $a$ consists of a 1-cell $m\colon U\rightarrow X$ and a 2-cell $\lambda\colon a\circ m\Rightarrow m$ such that the following diagrams commute:
	\begin{displaymath}
	\xymatrix @=3em
	{a\circ a\circ m\ar[r]^{\mu\circ 1}\ar[d]_{1\circ\lambda} & a\circ m\ar[d]^{\lambda}\\
		a\circ m\ar[r]_{\lambda} & m}\qquad\mathrm{and}\qquad
	\xymatrix @=3em
	{m\ar[r]^{\eta\circ 1}\ar[dr]_{1} & a\circ m\ar[d]^{\lambda} \\
	                  & m}
	\end{displaymath}
	
	A \emph{morphism} of left $a$-modules $(m,\lambda)\rightarrow(m',\lambda')$ with domain $U$ is a 2-cell $\alpha\colon m\Rightarrow m'$ such that the 
diagram below commutes:
	\begin{displaymath}
	\xymatrix @=3em
	{a\circ m\ar[r]^{1\circ\alpha}\ar[d]_{\lambda} & a\circ m'\ar[d]^{\lambda'}\\
		m\ar[r]_{\alpha} & m'}
	\end{displaymath}
\end{defi}

We thus have for any given object $U\in\ca{K}$ a \emph{category of (left) $a$-modules} with domain $U$, denoted by $_{a}^{U}\Mod$.
In the classical case $\ca{K}=\nc{Cat}$, modules over a monad 
$T:\ca{C}\rightarrow\ca{C}$ with domain the terminal category $\bf{1}$ are the usual `$T$-algebras' in the literature -- and the category 
$_{T}^{\bf{1}}\Mod$ is just the Eilenberg-Moore category $\ca{C}^{T}$. 

\begin{rmk}\label{rem:modulesaremonadic}
A more compact description of the category of left $a$-modules of given domain $U$, for a monad $a\colon X\rightarrow X$ in the bicategory $\ca{K}$, is as follows. For any object $U\in\ca{K}$ there is an induced monad (in the ordinary sense) on the hom-category $\ca{K}(U,X)$ given by post-composition
\begin{displaymath}
\ca{K}(U,a)\colon\ca{K}(U,X)\rightarrow\ca{K}(U,X)
\end{displaymath}
Then $_{a}^{U}\Mod$ is precisely the Eilenberg-Moore category $\ca{K}(U,X)^{\ca{K}(U,a)}$ of $\ca{K}(U,a)$-algebras. 
\end{rmk}

There are also the dual concepts of \emph{comonads} $(c\colon Z\to Z,\delta,\epsilon)$ and (left) $c$-\emph{comodules} $(k\colon V\to 
Z,\gamma\colon k\Rightarrow c\circ k)$ forming appropriate categories
${}_c^V\Comod$ which is the category of Eilenberg-Moore coalgebras $\ca{K}(V,Z)_{\ca{K}(V,c)}$ for the comonad given by post-composition with $c$.

Next, we recall some basic facts regarding monoids and comonoids in monoidal categories. 
These can be 
viewed as the one-object case of the above notions, since if $\ca{K}$ is a bicategory with a single object $*$, it is equivalent to the data of a 
monoidal category $\ca{V}=\ca{K}(*,*)$ with tensor corresponding to composition $\ca{K}(*,*)\times\ca{K}(*,*)\rightarrow \ca{K}(*,*)$. A monad is 
then just an object $A\in\ca{V}$ together with morphisms $\mu\colon A\otimes A\to A$, $\eta\colon I\to A$ satisfying associativity and unit laws, i.e. 
a \emph{monoid}. Dually, \emph{comonoids} $(C,\delta\colon C\to C\ot C,\epsilon\colon C\to I)$ can be viewed as comonads therein.

Along with the corresponding notions of morphism, 
we have the categories $\Mon(\ca{V})$ and $\Comon(\ca{V})$. Note that if $\ca{V}$ is braided 
monoidal, then both categories are themselves monoidal, although not necessarily braided. If however $\ca{V}$ is symmetric, then symmetry is 
inherited by both $\Mon(\ca{V})$ and $\Comon(\ca{V})$.

Analogously to how lax functors between bicategories preserve monads, any lax monoidal functor $F\colon\ca{V}\to\ca{W}$ with 
structure maps $\phi_{A,B}\colon FA\otimes FB\to F(A\otimes B)$	and $\phi_0\colon I\to F(I)$ induces a functor between their categories of monoids
\begin{equation}\label{eq:MonF}
\Mon F\colon\Mon(\ca{V})\to\Mon(\ca{W})
\end{equation}
explicitly given by $(A,\mu,\eta)\mapsto(FA,F\mu\circ\phi_{A,A},F\eta\circ\phi_0)$. Dually, oplax monoidal functors induce functors between the 
categories of comonoids. We usually denote $\Mon F$ and $\Comon F$ with the same name $F$ in order to not overburden with notation. 

Briefly recall that a functor $F\colon\ca{A}\times\ca{B}\to\ca{C}$ of two variables has a \emph{(right) parameterized} adjoint when for all 
$B\in\ca{B}$, $F(\mi,B)\dashv R(B,\mi)$; in particular this gives rise to a functor $R\colon\ca{B}^\op\times\ca{C}\to\ca{A}$, see 
\cite[\S IV.7,Thm.~3]{MacLane}. 
When $\ca{V}$ is braided monoidal closed, the internal hom functor $[\mi,\mi]\colon\ca{V}^\op\times\ca{V}\to\ca{V}$ obtains a lax monoidal structure, as 
the parameterized adjoint to the strong monoidal tensor functor, see e.g. \cite[Prop.~2.1]{Measuringcomonoid}. 
As a result, there is an induced functor between categories of monoids which is denoted 
\begin{equation}\label{defMon[]}
[\mi,\mi]\colon\Comon(\ca{V})^\op\times\Mon(\ca{V})\to\Mon(\ca{V}).
\end{equation}
This says that for a comonoid $C$ and a monoid $A$ in a braided monoidal closed category $\ca{V}$, the object $[C,A]$ has the standard \emph{convolution} monoid structure.

Specializing the notion of module over a monad in a bicategory, a (left) \emph{module} over a monoid $A$ in $\ca{V}$ consists of a pair $(M,\lambda)$, where $\lambda:A\otimes M\rightarrow M$ is an \emph{action} of $A$ on $M$, i.e. satisfies standard compatibilities with the multiplication and unit. Similarly, a module morphism $(M,\lambda)\rightarrow(M',\lambda')$ is now an arrow $f:M\rightarrow M'\in\ca{V}$ which commutes with the actions; thus we obtain the category $\Mod_{\ca{V}}(A)$ of $A$-modules in $\ca{V}$. 
Dually, we have the category $\Comod_{\ca{V}}(C)$ of (left) \emph{$C$-comodules} $(\K,\gamma\colon \K\to C\otimes \K)$ over a comonoid $C$ in 
$\ca{V}$.
By \cref{rem:modulesaremonadic}, we obtain the well-known fact that the categories $\Mod_{\ca{V}}(A)$ and $\Comod_{\ca{V}}(C)$ are respectively monadic and comonadic over $\ca{V}$, e.g. via the monad $(A\otimes\mi,\mu\otimes\mi,\eta\otimes\mi)$.

Any lax monoidal functor $F\colon\ca{V}\to\ca{W}$ also induces a functor between categories of modules $$\Mod F\colon 
\Mod_{\ca{V}}(A)\to\Mod_{\ca{W}}(FA)$$ for any monoid $A$, mapping $(M,\lambda)$ to $(FM,F\lambda\circ\phi_{A,M})$.
As a particular instance, the lax monoidal internal hom $[\mi,\mi]\colon\ca{V}^\op\times\ca{V}\to\ca{V}$ for a braided monoidal closed category induces, further to \cref{defMon[]}, a functor
\begin{equation}\label{defMod[]}
[\mi,\mi]_{CA}\colon\Comod_{\ca{V}}(C)^\op\times\Mod_{\ca{V}}(A)\to\Mod_{\ca{V}}([C,A])
\end{equation}
for any comonoid $C$ and monoid $A$.
To put the action of this functor in more concrete terms, given any $C$-comodule $(\K,\gamma)$ and any $A$-module $(M,\lambda)$, $[\K,M]$ acquires a 
$[C,A]$-module structure whose action $[C,A]\otimes[\K,M]\to[\K,M]$ is defined as the transpose of the following composite under the 
adjunction $-\otimes \K\dashv[\K,-]:$
\begin{equation}\label{eq:KMaction}
\begin{tikzcd}[row sep=.15in]
{[C,A]\otimes[\K,M]\otimes \K}\ar[r,"1\otimes\gamma"]\ar[ddrr,dashed] & {[C,A]\otimes[\K,M]\otimes C\otimes \K}\ar[r,"1\otimes s\otimes1"]
& {[C,A]\otimes C\otimes[\K,M]\otimes \K}\ar[d,"\mathrm{ev}\otimes\mathrm{ev}"] \\
& & A\otimes M\ar[d,"\lambda"] \\
& & M.
\end{tikzcd}
\end{equation}

Up to this point, we have recalled the definitions of categories of (co)modules over a \emph{fixed} (co)monoid. We now turn to `global' categories of (co)modules, fundamental in this work, by also summarizing the main points of \cite[\S~4.1]{Measuringcomodule}. 

\begin{defi}\label{def:globalcats}
The \emph{global category of modules} $\Mod_\ca{V}$ has objects modules $M$ over any monoid $A$ in the monoidal category $\ca{V}$, denoted by $M_{A}$. A morphism $p_{f}\colon M_{A}\rightarrow N_{B}$ is defined to be a monoid morphism $f\colon A\rightarrow B$ together with an arrow $p\colon M\rightarrow N\in\ca{V}$ making the following diagram commute:
\begin{displaymath}
\begin{tikzcd}[sep=.3in]
A\otimes M\ar[rr,"\lambda"]\ar[d,"1\otimes p"'] & & M\ar[d,"p"] \\
		             A\otimes N\ar[r,"f\otimes 1"'] & B\otimes N\ar[r,"\lambda"'] & N 
\end{tikzcd}
\end{displaymath}
Dually, the \emph{global category of comodules} $\Comod_\ca{V}$ has objects all comodules $\K$ over any comonoid $C$ in $\ca{V}$, and a morphism 
$k_{g}:\K_{C}\rightarrow \LL_{D}$ is a comonoid morphism $g:C\to D$ and an arrow $k\colon \K\to\LL$ making the dual diagram to the above commute.
\end{defi}

There are obvious forgetful functors $\Mod_\ca{V}\to\Mon(\ca{V})$ and $\Comod_\ca{V}\to\Comon(\ca{V})$ which map $M_{A}\in\Mod_\ca{V}$ to its monoid $A$ and 
$X_{C}\in\Comod_\ca{V}$ to its comonoid $C$ respectively. We will describe their (op)fibrational structure in the next section; for now, the following 
states that (co)limits in the respective global categories are created from those in $\ca{V}$ and (co)monoids.

\begin{prop}\label{prop:Modmonadic}
The functor $\Comod_\ca{V}\to\ca{V}\times\Comon(\ca{V})$ which maps $\K_{C}$ to the pair $(\K,C)$ is comonadic, with right adjoint 
$(Y,D)\mapsto(D\otimes Y)_D$.  Dually, the global category of modules $\Mod_\ca{V}$ is monadic over the category $\ca{V}\times\Mon(\ca{V})$.
\end{prop}

Let us also record here that, as was the case for the categories of monoids and comonoids in $\ca{V}$, the existence of a braiding on $\ca{V}$ implies 
that the global categories $\Mod_\ca{V}$ and $\Comod_\ca{V}$ have a monoidal structure which furthermore is symmetric whenever $\ca{V}$ is such. For 
example, $X_C\ot Y_D$ is a $C\ot D$-comodule with coaction
\begin{displaymath}
X\ot Y\xrightarrow{\delta\ot\delta}X\ot C\ot Y\ot D\xrightarrow{1\ot s\ot1}X\ot Y\ot C\ot D.
\end{displaymath}

Finally, in what follows we are in particular interested in local presentability of categories of (co)monoids and (co)modules. For the 
general theory of locally presentable and accessible categories, we refer the reader to \cite{MakkaiPare,LocallyPresentable}. 
Briefly recall that if $\lambda$ is a regular cardinal, a $\lambda$-\emph{accessible} category is a category with $\lambda$-filtered colimits 
and a small set of $\lambda$-presentable objects (namely $\ca{C}(c,\mi)$ preserves $\lambda$-filtered colimits) such that every other object is a 
$\lambda$-filtered colimit of those; and a category is \emph{locally} $\lambda$-\emph{presentable} if it is furthermore cocomplete. A \emph{locally 
presentable} category is locally $\lambda$-presentable for some $\lambda$.

The following theorem ensures that any cocontinuous functor on a locally presentable 
category has a right adjoint, since presentable objects form a small dense subcategory.

\begin{thm}\cite[Thm.~5.33]{Kelly}\label{thm:Kelly}
 If the cocomplete category $\ca{C}$ has a small dense subcategory, every cocontinuous $S\colon\ca{C}\to\ca{B}$ has a right adjoint.
\end{thm}

For example, the above can be used to deduce monoidal closedness of comonoids when $\ca{V}$ is closed (see \cite[\S~3.2]{MonComonBimon}),
because $\Comon(\ca{V})$ -- as well as $\Mon(\ca{V})$ -- is locally presentable under the conditions below.

\begin{prop}\label{prop:MonComonlp}
If $\ca{V}$ is a locally presentable monoidal category where $\ot$ preserves filtered colimits in both variables, the category of monoids 
$\Mon(\ca{V})$ is locally presentable and finitary monadic over $\ca{V}$ and the category of comonoids $\Comon(\ca{V})$ is locally presentable and 
comonadic over $\ca{V}$. Moreover, if $\ca{V}$ is braided monoidal closed, then $\Comon(\ca{V})$ is monoidal closed.
\end{prop}

Briefly, both categories can be written using 2-categorical limits (e.g. inserters and equifiers) of locally presentable categories and (accessible) functors, 
thus they end up locally presentable essentially due to the ``Limit Theorem'' for accessible categories \cite[Thm.~5.1.6]{MakkaiPare}; 
more explicit descriptions of these constructions for comonoids can be found in \cite{MonComonBimon}. 

Similarly, local presentability is also inherited by categories of modules and comodules as a result of their (co)monadic structure over $\ca{V}$. 
This follows from \cite[\S~2.78]{LocallyPresentable} for Eilenberg-Moore algebras and an analogous argument for coalgebras again as a result of the Limit Theorem for accessible categories, and the fact that a category is locally 
presentable if and only if it is cocomplete and accessible category or equivalent a complete and accessible \cite[Cor.~2.47]{LocallyPresentable}.

\begin{thm}\label{thm:monadalgebras}
For a locally presentable category $\ca{V}$, the category of algebras for a finitary monad as well as the category of coalgebras for a finitary comonad are locally presentable categories.
\end{thm}

For example, under the assumptions below, the respective (co)monads preserve all filtered colimits. 

\begin{prop}\label{prop:ModComodlp}
If $\ca{V}$ is a locally presentable monoidal category where $\ot$ preserves filtered colimits in both variables, the category of modules 
$\Mod_\ca{V}(A)$ for any monoid $A$ and the category of comodules $\Comod_\ca{V}(C)$ for any comonoid $C$ are locally presentable. 
Moreover, the global category of modules 
$\Mod_\ca{V}$ and the global category of comodules $\Comod_\ca{V}$ are locally presentable. 
\end{prop}

\subsection{Fibrations}\label{sec:fibredadjunctions}

In this section we record some definitions and results involving fibred adjunctions and fibred limits; this is done to fix notation and for 
future reference. For further details and related bibliography, we refer the reader to \cite{FibredAdjunctions,Measuringcomodule}. 

Briefly recall that $P\colon\ca{A}\to\ca{X}$ is a (cloven) \emph{fibration} if for any $f\colon X\to Y$ in $\ca{X}$ (called the \emph{base}) and $B$ above $Y$ in $\ca{A}$ (called the \emph{total} category), there is a cartesian lifting $\Cart(f,B)\colon f^*B\to B$ of $f$ to $B$, namely any other map $u\colon A\to B$ with $Pu=f\circ g$ uniquely factorizes through $\Cart(f,B)$ via some $v\colon A\to f^*B$ with $Pv=g$. 
The \emph{fibre} category $\ca{A}_Y$ consists of objects above $Y$ and morphisms above $1_Y$ (called \emph{vertical}). By the universal 
property of liftings, there are induced \emph{reindexing} functors $f^*\colon\ca{A}_Y\to\ca{A}_X$ for any $f$ in the base, satisfying 
$(g\circ f)^*\cong g^*\circ f^*$ and $1_{\ca{A}_X}\cong(1_X)^*$. If these isomorphisms are equalities, we have a \emph{split} fibration. 
Dually, we have the notion of an \emph{opfibration} with cocartesian liftings deloned $\Cocart(f,A)\colon A\to f_!A$. A \emph{bifibration} is both a fibration and an opfibration, 
equivalently we have adjunctions $f_!\dashv f^*$.

The fundamental \emph{Grothendieck construction} provides a passage between fibrations $\ca{A}\to\ca{X}$ and pseudofunctors 
$\ca{X}^\op\to\Cat$. Concretely, given a pseudofunctor $F\colon\ca{X}^\op\to\Cat$, its \emph{Grothendieck category} $\int F$ has as objects pairs $(X\in\ca{X},A\in FX)$ and as morphisms pairs $(f\colon X\to Y\in\ca{X},\phi\colon A\to f^*B\in FX)$. Essentially, the induced fibration $\int F\to\ca{X}$ that maps $(X,A)$ to $X$ has as fibres the categories $FX$ and and as reindexing functors $Ff\colon FY\to FX$.
This construction extends to a 2-equivalence between the corresponding 2-categories; 
for example, a fibred 1-cell between two fibrations $P,Q$ is a commutative
\begin{displaymath}
\begin{tikzcd}
\ca{A}\ar[r,"K"]\ar[d,"P"'] & \ca{B}\ar[d,"Q"] \\
\ca{X}\ar[r,"F"'] & \ca{Y}
\end{tikzcd}
\end{displaymath}
where $K$ preserves cartesian liftings. In the split case, the correspondence restricts to ordinary functors into $\Cat$. Similarly, there is an equivalence between opfibrations and pseudofunctors $\ca{X}\to\Cat$. 

As an example, consider the global categories of (co)modules of \cref{def:globalcats} in a monoidal category $\ca{V}$. These are split fibred and 
opfibred respectively over the categories of (co)monoids, since they arise as the Grothendieck categories of the functors
\begin{equation}\label{eq:globalfibs}
 \begin{tikzcd}[row sep=.02in,/tikz/column 1/.append style={anchor=base east},/tikz/column 2/.append style={anchor=base west},/tikz/column 5/.append style={anchor=base east},/tikz/column 6/.append style={anchor=base west}]
\Mon(\ca{V})^\op\ar[r] & \Cat &&& \Comon(\ca{V})\ar[r] & \Cat \\
A\ar[r,mapsto]\ar[dd,"f"'] & \Mod_\ca{V}(A) &&& C\ar[r,mapsto]\ar[dd,"f"'] & \Comod_\ca{V}(C)\ar[dd,"f_!"]\\
\hole \\
B\ar[r,mapsto] & \Mod_\ca{V}(B)\ar[uu,"f^*"'] &&& D\ar[r,mapsto] & \Comod_\ca{V}(C)
\end{tikzcd}
\end{equation}
where $f^*$ is the so-called `restriction of scalars', turning any $B$-module $N$ into an $A$-module via an action $A\ot N\xrightarrow{f\ot1}B\ot 
N\xrightarrow{\mu}N$. Dually, $f_!$ is the `corestriction of scalars' for comodules. Moreover, in a braided monoidal closed category, it can be 
verified that the `fibrewise' functors $[\mi,\mi]_{CA}$ 
induced by the internal hom as in \cref{defMod[]} gather together into a functor between the total categories
\begin{equation}\label{eq:HomMod0}
[\mi,\mi]\colon\Comod_\ca{V}^\op\times\Mod_\ca{V}\to\Mod_\ca{V} 
\end{equation}
which is part of a fibred 1-cell as follows (see \cite[Eq.~(20)]{Measuringcomodule} and explanation below)
\begin{equation}\label{eq:HomMod}
\begin{tikzcd}
\Comod_\ca{V}^\op\times\Mod_\ca{V}\ar[r,"{[\mi,\mi]}"]\ar[d] & \Mod_\ca{V}\ar[d] \\
\Comon(\ca{V})^\op\times\Mon(\ca{V})\ar[r,"{[\mi,\mi]}"',"\cref{defMon[]}"] & \Mon(\ca{V}).
\end{tikzcd}
\end{equation}

In what follows, we are interested in lifting adjunctions from bases to total categories of (op)fibrations.  
The following result is proved in detail in \cite[Prop.~3.4\&3.5]{Measuringcomodule}, and for the simpler case of fibrations over a fixed base see \cite[Prop.~8.4.2]{Handbook2}.

\begin{thm}\label{thm:totaladjointthm}
Suppose $(L,F)\colon U\to V$ is an opfibred 1-cell and $F\dashv G$ is an adjunction with counit $\varepsilon$ between the bases of the opfibrations, as in
\begin{displaymath}
\begin{tikzcd}[column sep=.6in, row sep=.4in]
\ca{C}\ar[d,"U"']\ar[r,"L"] & \ca{D}\ar[d,"V"]   \\
\ca{X}\ar[r,shift left=1.5,"F"]\ar[r,phantom,"\scriptstyle\bot"] & \ca{Y}\ar[l,shift left=1.5,"G"]
\end{tikzcd}
\end{displaymath}
If, for each $Y\in\ca{Y}$, the composite functor between the fibres
\begin{displaymath}
\ca{C}_{GY}\xrightarrow{L_{GY}}\ca{D}_{FGY}
\xrightarrow{(\varepsilon_Y)_!}\ca{D}_Y
\end{displaymath}
has a right adjoint, then $L$ has a right adjoint $R$ between the total categories, and $(U,V)$ is a map between the adjunctions $L\dashv 
R$ and $F\dashv G$.
\end{thm}

Next, a fibration $P\colon\ca{A}\to\ca{X}$ is said to have \emph{fibred} $\ca{I}$-limits for a small category $\ca{I}$, if 
all fibers $\ca{A}_X$ have $\ca{I}$-limits and these are preserved by all reindexing functors $f^*\colon\ca{A}_Y\to\ca{A}_X$, see 
e.g. \cite[\S2.3]{FibredAdjunctions}. 
Moreover, if the base of a fibration has $\ca{I}$-limits, then we have the following useful result.

\begin{prop}\label{prop:fiberwiselimits}
Let $\ca{I}$ be a small category and suppose that $P\colon\ca{A}\to\ca{X}$ is a fibration, where the base category $\ca{X}$ has all 
$\ca{I}$-limits. Then the following are equivalent:
	\begin{enumerate}
		\item $\ca{A}$ has all $\ca{I}$-limits and these are (strictly)\footnote{As is clear from the proof, this essentially means that limits in 
the total category are constructed depending on a choice of limits in the base category.} preserved by the functor $P$;
		\item the fibration $P\colon\ca{A}\to\ca{X}$ has fibred $\ca{I}$-limits.
	\end{enumerate}
\end{prop}
In provided references, this is proved using results for fibred adjunctions, see \cite[Cor.~3.7]{FibredAdjunctions}; we here provide a direct 
calculation which is of use later in \cref{prop:continuous fibred 1-cells}.
\begin{proof}
\underline{(1)$\implies$(2)}: Consider any functor $G\colon\ca{I}\to\ca{A}_X$, for some fixed $X\in\ca{X}$. Let $(\lambda_i\colon \Lim G\to 
Gi)_{i\in\ca{I}}$ be the limit of this diagram in the ambient category $\ca{A}$ as depicted below. Since $P$ is assumed to preserve $\ca{I}$-limits, 
the cone 
$(P\lambda_i\colon P(\Lim G)\to PGi)_{i\in\ca{I}}$ is a limit in $\ca{X}$. Since $G$ takes values in $\ca{A}_X$, we have $PGi=X$ for all 
$i\in\ca{I}$, 
and $PG\alpha=1_X$ for all $\alpha\colon i\to j\in\ca{I}$. Thus, the collection of identity morphisms $(1_X\colon X\to X)_{i\in\ca{I}}$ constitutes a 
cone over $PG$ in $\ca{X}$, and hence there exists a unique $\Delta\colon X\to P(\Lim G)\in\ca{X}$ such that $P\lambda_i\circ\Delta=1_X$ for all 
$i\in\ca{I}$.
	
If we consider the cartesian lift $\Delta^{*}(\Lim G)\to \Lim G\in\ca{A}$ of $\Delta$ at $\Lim G$, it can be verified 
that $(\lambda_{i}\circ\Cart(\Delta,\Lim G)\colon \Delta^{*}(\Lim G)\to Gi)_{i\in\ca{I}}$ is the desired limit cone in $\ca{A}_X$

\begin{equation}\label{eq:fiberlimit}
\begin{tikzcd}[column sep=3.5em]
\Delta^*(\Lim G)\ar[dr,bend left,"{\Cart}"]\ar[d,dashed] && \\
Gi & \Lim G\ar[l,"\lambda_i"] &\ca{A}\ar[dd,dotted,"P"] \\
X\ar[d,equal]\ar[dr,bend left,"\Delta"] & & \\
X & P(\Lim G)\ar[l,"P\lambda_i"] & \ca{X}
\end{tikzcd}
\end{equation}
Moreover, for any $f\colon Y\to X$ in $\ca{X}$, the cone $(f^{*}(\lambda_{i}\circ\Cart)\colon 
f^{*}(\Delta^{*}(\Lim G))\to f^{*}(Gi))_{i\in\ca{I}}$ can be verified to be a limit cone in $\ca{A}_Y$.
	
\underline{(2)$\implies$(1)}: Consider a diagram $G\colon\ca{I}\to\ca{A}$. Let $(\kappa_{i}\colon \Lim PG\to PGi)_{i\in\ca{I}}$ be the corresponding 
limit of the diagram $PG$ in $\ca{X}$ as depicted below. For each $i\in\ca{I}$, we can take a cartesian lift $(\kappa_i)^*Gi\to Gi\in\ca{A}$ 
of $\kappa_{i}$ at $Gi$, and call its domain $(\kappa_i)^*Gi=Di$. Now for every $\alpha\colon i\to j\in\ca{I}$, we have 
$PG\alpha\circ\kappa_i=\kappa_j$ since these form a limit cone. Thus, by the universal property of the cartesian lift of $\kappa_i$ at $G_i$,  there 
exists a unique 
map $Di\to Dj\in\ca{A}$ above $1_{\Lim PG}$, call it $D\alpha$, with the property that $\Cart(\kappa_j,Gj)\circ 
D\alpha=G\alpha\circ\Cart(\kappa_i,Gi)$:
\begin{displaymath}
\begin{tikzcd}[column sep=3.5em]
	Di\ar[r,"{\Cart}"]\ar[d,dashed,"D\alpha"'] & Gi\ar[dr,"G\alpha"] &&  \\
	Dj\ar[rr,"{\Cart}"'] && Gj & \ca{A}\ar[dd,dotted,"P"] \\
	\Lim PG\ar[d,equal]\ar[r,"\kappa_i"] & PGi\ar[dr,"PG\alpha"] &&  \\
	\Lim PG\ar[rr,"\kappa_j"'] && PGj & \ca{X}
\end{tikzcd} 
\end{displaymath}
The uniqueness clause immediately implies that this assignment actually defines a functor $D\colon\ca{I}\to\ca{A}_{\Lim PG}$. Now by assumption, 
this functor has a limit in that fiber, say $(p_i\colon A\to Di)_{i\in\ca{I}}$ as in:
\begin{equation}\label{eq:LimG}
 \begin{tikzcd}[column sep=4em]
\Lim D{=}A\ar[d,"p_i"']\ar[dr,dashed,bend left=20] && \\
Di{=}(\kappa_i)^*Gi\ar[r,"\Cart"'] & Gi & \ca{A}\ar[d,dotted,"P"] \\
\Lim PG\ar[r,"\kappa_i"'] & PGi & \ca{X}
 \end{tikzcd}
\end{equation}
It can be checked that 
$(\Cart(\kappa_i,Gi)\circ p_{i}\colon A\to Gi)_{i\in\ca{I}}$ gives the desired limit of $G$ in $\ca{A}$, and clearly $PA=\Lim PG$.
\end{proof}

As is evident from the above constructions, when both exist, limits inside the fibers do not need to coincide with limits in the total category in 
general; one case this is true is that of connected limits. 

\begin{cor}\label{connected}
Let $\ca{I}$ be a small connected category and $P\colon\ca{A}\to\ca{X}$ a fibration such that $\ca{A}$ and $\ca{X}$ have, $P$ preserves all 
$\ca{I}$-limits. For every $X\in\ca{X}$, the inclusion $\ca{A}_X\hookrightarrow\ca{A}$ preserves $\ca{I}$-limits.
\end{cor}

\begin{proof}
Consider a diagram $G\colon\ca{I}\to\ca{A}_X$. Since $\ca{I}$ is connected, the limit of 
$\ca{I}\xrightarrow{G}\ca{A}_X\hookrightarrow\ca{A}\xrightarrow{P}\ca{X}$ is now (isomorphic to) $X$, since $PG$ is the constant diagram on $X$. More 
precisely, in the picture of \cref{eq:fiberlimit} there exists an isomorphism $X\cong P(\Lim G)$ in $\ca{X}$, which is forced to be $\Delta$ by the 
uniqueness of the triangle commutativities. Since $\Delta$ is invertible, its cartesian lift 
$\Cart(\Delta,\Lim G)$ is also invertible in $\ca{A}$, which proves the result.
\end{proof}

Finally, the following result concerns the preservation of limits by fibred 1-cells whose functors between the bases 
preserves them, and appears to be novel.

\begin{prop}\label{prop:continuous fibred 1-cells}
Let $(K,F)$ be a fibred 1-cell between the fibrations $P,Q$ as indicated below:
\begin{displaymath}
\begin{tikzcd}
	\ca{A}\ar[r,"K"]\ar[d,"P"'] & \ca{B}\ar[d,"Q"] \\
	\ca{X}\ar[r,"F"'] & \ca{Y}
\end{tikzcd}
\end{displaymath}
Assume furthermore that $P$ and $Q$ have fibred $\ca{I}$-limits for some small category $\ca{I}$, and that the bases $\ca{X},\ca{Y}$ have 
$\ca{I}$-limits which are preserved by the functor $F$. Then the following are equivalent:
\begin{enumerate}
	\item $K$ preserves $\ca{I}$-limits;
	\item for every $X\in\ca{X}$, the restricted functor $K_X\colon\ca{A}_X\to\ca{B}_{FX}$ preserves $\ca{I}$-limits.
\end{enumerate}
\end{prop}
\begin{proof}
\underline{(1)$\implies$(2)}: Consider a diagram $G\colon\ca{I}\to\ca{A}_X$. As explained above \cref{eq:fiberlimit}, the limit in the fiber 
$\ca{A}_X$ is $(\lambda_{i}\circ\Cart\colon \Delta^{*}(\Lim G)\to Gi)_{i\in\ca{I}}$, 
where $(\lambda_{i}\colon \Lim G\to Gi)_{i\in\ca{I}}$ is the limit of $G$ in $\ca{A}$.
Since $K$ preserves limits, $(K\lambda_{i}\colon K\Lim G\to KGi)_{i\in\ca{I}}$ is a limit of $KG$ in $\ca{B}$, as depicted below.
Moreover, $K\Cart(\Delta,\Lim G)$ is $Q$-cartesian over $F\Delta$ since $K$ preserves cartesian lifts. Furthermore, since $F$ also preserve 
limits, $FP\lambda_i=QK\lambda_i$ is a limiting cone in $\ca{Y}$, thus $F\Delta$ is the unique morphism from $FX$ 
such that $QK\lambda_{i}\circ F\Delta=1_{FX}$ for all $i\in\ca{I}$:
\begin{displaymath}
\begin{tikzcd}[column sep=3.5em]
\Delta^*(\Lim G)\ar[dr,bend left,"{\Cart}"]\ar[d,dashed] && \\
Gi & \Lim G\ar[l,"\lambda_i"] & \mapsto\\
X\ar[d,equal]\ar[dr,bend left,"\Delta"] & & \\
X & P(\Lim G)\ar[l,"P\lambda_i"]
\end{tikzcd}
\begin{tikzcd}[column sep=3.5em]
K\Delta^*(\Lim G)\ar[dr,bend left,"{K\Cart}"]\ar[d,dashed] & \\
KGi & K(\Lim G)\ar[l,"K\lambda_i"] \\
FX\ar[d,equal]\ar[dr,bend left,"F\Delta"] &  \\
FX & FP(\Lim G)=QK(\Lim G)\ar[l,"FP\lambda_i=QK\lambda_i"]
\end{tikzcd}
\end{displaymath}
Thus, by the corresponding construction of limits in the fiber $\ca{B}_{FX}$, $(K(\lambda_{i}\circ\Cart)\colon K\Delta^{*}(\Lim G)\to 
KGi)_{i\in\ca{I}}$ is indeed a limit cone.

\underline{(2)$\implies$(1)}: Consider a diagram $G\colon\ca{I}\to\ca{A}$. As shown in \cref{eq:LimG}, its limit 
is $(\Cart(\kappa_i,Gi)\circ p_{i}\colon A\to Gi)_{i\in\ca{I}}$ where the cartesian lifts of $(\kappa_{i}\colon \Lim PG\to PGi)_{i\in\ca{I}}$ at 
$Gi$'s are used to form a diagram $D$ in the fiber of $\Lim PG$, whose limit is $A$. Since $F$ preserves $\ca{I}$-limits, $F(\Lim PG)$
is the limit of $FPG=QKG$ in $\ca{Y}$ as depicted below. In addition, for every $i\in\ca{I}$ we have that $K\Cart(\kappa_{i},Gi)$ is $Q$-cartesian 
over $F\kappa_{i}$. Finally, since by assumption $K_{\Lim PG}\colon\ca{A}_{\Lim PG}\to\ca{B}_{F(\Lim PG)}$ preserves $\ca{I}$-limits, we know that 
$(Kp_i\colon KA\to KDi)_{i\in\ca{I}}$ is the limit in $\ca{B}_{F(\Lim PG)}$ of $KD\colon\ca{I}\to\ca{B}_{F(\Lim PG)}$
\begin{displaymath}
\begin{tikzcd}[column sep=3.5em]
A\ar[dr,dashed,bend left]\ar[d,"p_i"'] && \\
Di\ar[r,"\Cart"'] & Gi & \mapsto\\
\Lim PG\ar[r,"\kappa_i"'] & PGi &
\end{tikzcd}
\begin{tikzcd}[column sep=3.5em]
KA\ar[dr,bend left,dashed]\ar[d,"Kp_i"'] & \\
KDi\ar[r,"K\Cart"'] & KGi \\
F(\Lim PG)\ar[r,"F\kappa_i"'] & FPGi=QKGi
\end{tikzcd}
\end{displaymath}
Thus, by the corresponding construction of $\ca{I}$-limits in $\ca{B}$, we see that $\left(K(\Cart\circ p_{i})\colon KA\to 
KGi\right)_{i\in\ca{I}}$ is indeed a limit cone therein.
\end{proof}

\subsection{Sweedler theory for monoidal categories}\label{sec:enrichedfibrations}

In this section, we summarize the main points regarding an enrichment of monoids in comonoids using Sweedler's \emph{universal 
measuring coalgebras} \cite{Sweedler}, an enrichment of the global category of modules in comodules using \emph{universal measuring comodules} 
\cite{Batchelor}, as well as an enriched fibration structure they form. The term `Sweedler 
theory' was originally used in \cite{AnelJoyal} for the theory surrounding the enrichment of dg-(co)algebras specifically, but here we use it 
as an umbrella term for general (co)monoids, (co)modules and their (op)fibrations --  
\cref{thm:MonenrichedComon,thm:ModenrichedComod,thm:ModenrichedComodfib}. 
More details on this material can be found in \cite{Measuringcomonoid,Measuringcomodule,EnrichedFibration}.

Recall that an \emph{action} \cite{Benabou} of a monoidal category $(\ca{V},\otimes,I)$ on an ordinary category $\ca{C}$ is a functor 
$*\colon\ca{V}\times\ca{C}\to\ca{C}$ along with isomorphisms 
\begin{equation}\label{eq:actions}
\chi_{X,Y,C}\colon X*(Y*C)\cong (X\ot Y)*C \textrm{ and } \nu_C\colon C\cong I*C 
\end{equation}
satisfying compatibility conditions that express, on higher level, that $\ca{C}$ is a pseudomodule for the pseudomonoid $\ca{V}$ in the 
cartesian monoidal 2-category $\Cat$. If $*$ is an action, then $*^\op$ is also an action. There is a connection between $\ca{V}$-actions and 
$\ca{V}$-enrichment, which e.g. in \cite{AnoteonActions} is stated as follows:
if a monoidal category $\ca{V}$ acts on a category $\ca{C}$ where the action has a parameterized adjoint, then we can enrich $\ca{C}$ in 
$\ca{V}$. The following version is \cite[Prop.~7]{Duoidalenrichment}, and starts from an action of the opposite 
of a braided monoidal category and is more convenient for our purposes.

\begin{thm}\label{thm:cotensorenrich}
 Suppose $\ca{V}$ is a braided monoidal category, and $\pitchfork\colon\ca{V}^\op\times\ca{C}\to\ca{C}$ is an action with an adjoint $\mi\pitchfork^\op D\dashv R(\mi,D)$ via $\ca{C}(C,X\pitchfork D)\cong\ca{V}(X,R(C,D))$. Then there is a $\ca{V}$-enriched category with hom-objects $R(C,D)$ and underlying category $\ca{C}$.
 
Moreover, if $\ca{V}$ is closed then the enriched category admits cotensors $X\pitchfork C$. Finally, if the action has another parameterized adjoint $X*\mi\dashv X\pitchfork\mi$ then it also admits tensors $X*C$.
\end{thm}

The above theory can be applied to the categories of monoids and comonoids in a symmetric\footnote{Symmetry is needed because braiding alone only makes the category of comonoids monoidal.}
 monoidal closed $\ca{V}$ which is furthermore locally presentable.
More precisely, under these assumptions, $\Comon(\ca{V})$ is a symmetric monoidal closed category (\cref{prop:MonComonlp}) and the restricted 
internal hom $[\mi,\mi]\colon\Comon(\ca{V})^\op\times\Mon(\ca{V})\to\Mon(\ca{V})$ of \cref{defMon[]} is an action with adjoints
\begin{align}
[\mi,B]^\op\dashv P(\mi,B)\colon&\Mon(\ca{V})^\op\to\Comon(\ca{V})\label{eq:meascoalg}\\
C\triangleright\mi\dashv[C,\mi]\colon&\Mon(\ca{V})\to\Mon(\ca{V})\label{eq:tenscoalg}
\end{align}
using local presentability and (co)continuity properties and \cref{thm:Kelly} appropriately, as detailed in \cite[\S~4\&5]{Measuringcomonoid}. Thus the above theorem applies,
where the hom-objects $R(A,B)$ are the universal measuring comonoids first introduced in \cite{Sweedler} in the context of vector 
spaces. 

\begin{thm}[Sweedler theory for (co)monoids]\label{thm:MonenrichedComon}
Suppose $\ca{V}$ is a locally presentable, symmetric monoidal closed category. The category of monoids $\Mon(\ca{V})$ is tensored and cotensored enriched in the symmetric monoidal closed category of comonoids $\Comon(\ca{V})$.
\end{thm}

The situation for (co)modules is similar. Although it is still possible (yet less informative, see \cite[Rem.~4.8]{Measuringcomodule}) to obtain required adjunctions using results
like \cref{thm:Kelly}, it is more efficient to use the theory of (op)fibred adjunctions.
First of all, under the running assumptions on $\ca{V}$, the 
category of comodules is monoidal closed as shown in \cite[Prop.~4.5]{Measuringcomodule}; below is a sketch of the argument, which is later 
generalized in \cref{prop:Comodclosed}.

\begin{prop}\label{prop:ComodVclosed}
 For a braided monoidal closed and locally presentable category $\ca{V}$, the monoidal category $\Comod_\ca{V}$ is closed.
\end{prop}

\begin{proof}[Sketch]
 The existence of the internal hom of $\Comod_\ca{V}$ follows from \cref{thm:totaladjointthm} as an adjoint of the opfibred 1-cell
 \begin{displaymath}
 \begin{tikzcd}[column sep=.7in,row sep=.5in]
\Comod_\ca{V}\ar[d]\ar[r,bend left=10,"{\mi\otimes X_C}"]\ar[r,phantom,"\bot"] & \Comod_\ca{V}\ar[d]\ar[l,dashed,bend 
left=10,"{\overline{\HOM}(X_C,\mi)}"] \\
\Comon(\ca{V})\ar[r,bend left=10,"{\mi\otimes C}"]\ar[r,phantom,"\bot"] & \Comon(\ca{V})\ar[l,bend left=10,"{\HOM(C,-)}"]
 \end{tikzcd}
 \end{displaymath}
 where the bottom adjunction is the monoidal closed structure of $\Comon(\ca{V})$ from \cref{prop:MonComonlp}. In particular, even though the 
underlying object of the internal hom of two comonoids $\HOM(C,D)$ is not $[C,D]_\ca{V}$, the underlying comonoid of the internal hom of two comodules 
$\overline{\HOM}(X_C,Y_D)$ is $\HOM(C,D)$.
\end{proof}

Moreover, the restricted internal hom $[\mi,\mi]\colon\Comod_\ca{V}^\op\times\Mod_\ca{V}\to\Mod_\ca{V}$ of \cref{eq:HomMod0} is an action 
\cite[Prop.~4.9]{Measuringcomodule}, and has parameterized adjoints by \cite[Prop.~4.6\&Thm.~4.10]{Measuringcomodule} as sketched below.

\begin{prop}\label{prop:HomModadjoints}
 Suppose $\ca{V}$ is a symmetric monoidal closed and locally presentable category. There are adjunctions
 \begin{align*}
  [\mi,N_B]^\op\dashv Q(\mi,N_B)\colon&\Mod_\ca{V}^\op\to\Comod_\ca{V} \\
  X_C\oslash\mi\dashv [X_C,\mi]\colon&\Mod_\ca{V}\to\Mod_\ca{V} 
 \end{align*}
for any module $N_B$ and any comodule $X_C$.
\end{prop}

\begin{proof}[Sketch]
The existence of $Q$ and $\oslash$ follow from \cref{thm:totaladjointthm} and its dual, for
\begin{displaymath}
 \begin{tikzcd}[column sep=.7in,row sep=.5in]
\Comod_\ca{V}\ar[d]\ar[r,bend left=10,"{[\mi,N_B]^\op}"]\ar[r,phantom,"\bot"] & \Mod_\ca{V}^\op\ar[d]\ar[l,dashed,bend left=10] \\
\Comon(\ca{V})\ar[r,bend left=10,"{[\mi,B]^\op}"]\ar[r,phantom,"{\phantom{\scriptstyle\cref{eq:meascoalg}}}\bot{\scriptstyle\cref{eq:meascoalg}}"] & 
\Mon(\ca{V})^\op\ar[l,bend left=10,"{P(\mi,B)}"]
 \end{tikzcd}\qquad
  \begin{tikzcd}[column sep=.7in,row sep=.5in]
\Mod_\ca{V}\ar[d]\ar[r,bend right=10,"{[X_C,\mi]}"']\ar[r,phantom,"\bot"] & \Mod_\ca{V}\ar[d]\ar[l,dashed,bend right=10] \\
\Mon(\ca{V})\ar[r,bend right=10,"{[C,\mi]}"']\ar[r,phantom,"{\phantom{\scriptstyle\cref{eq:meascoalg}}}\bot{\scriptstyle\cref{eq:tenscoalg}}"] & 
\Mon(\ca{V})\ar[l,bend right=10,"{C\triangleright\mi}"']
 \end{tikzcd}
\end{displaymath}
As a result, $Q(M_A,N_B)$ is a $P(A,B)$-comodule and $X_C\oslash N_B$ is a $C\triangleright B$-module. 
\end{proof}

All clauses of \cref{thm:cotensorenrich} are now satisfied, see \cite[Thm.~4.10]{Measuringcomodule}. The hom-objects $S(M,N)$ are the universal measuring comodules, first introduced in \cite{Batchelor} for vector spaces. 

\begin{thm}[Sweedler theory for (co)modules]\label{thm:ModenrichedComod}
 Suppose $\ca{V}$ is a locally presentable, symmetric monoidal closed category. The global category of modules $\Mod_\ca{V}$ is tensored and cotensored enriched in the symmetric monoidal closed global category of comodules $\Comod_\ca{V}$.
\end{thm}

We now turn to an enrichment of the fibration of modules over monoids in the opfibration of comodules over comonoids in a monoidal 
category, which will later be generalized for (co)monads and (co)modules in double categories. For that purpose, we recall some 
relevant concepts and results.

\begin{defi}(\cite{Framedbicats})\label{monoidalfibration}
Let $\ca{V}$, $\ca{W}$ be monoidal categories. A fibration $T\colon\ca{V}\to\ca{W}$ is \emph{monoidal} if
$T$ is a strict monoidal functor and the tensor product $\otimes_\ca{V}$ preserves cartesian liftings. 
It is \emph{braided} or \emph{symmetric} monoidal when it is a braided strict monoidal functor between braided or symmetric monoidal 
categories. Dually, there is the notion of a (braided or symmetric) \emph{monoidal opfibration}.
\end{defi}

\begin{ex}\label{ex:monoidalfibs}
The fibration $\Mod_\ca{V}\to\Mon(\ca{V})$ and opfibration $\Comod_\ca{V}\to\Comon(\ca{V})$ as in \cref{eq:globalfibs} are both monoidal when 
$\ca{V}$ is braided monoidal, since indeed the monoidal structure is strictly preserved, 
and it can be 
verified that the tensor product preserves (co)cartesian liftings (\cite[Prop.~4.5]{Measuringcomodule}). 

Notice that whereas the global categories $\Mod_\ca{V}$ and $\Comod_\ca{V}$ are monoidal, their fibres $\Mod_\ca{V}(A)$ and $\Comod_\ca{V}(C)$ for a 
fixed (co)monoid are not monoidal -- since for two $A$-modules $M$ and $N$, $M\ot N$ is naturally just an $A\ot A$-module. This phenomenon 
is further discussed in \cite[\S~5.7]{MonGroth}.
\end{ex}

We next recall the definition of a fibration enriched in a monoidal fibration, see \cite[Def.~3.8]{EnrichedFibration}. If $\ca{A}$ is a 
$\ca{V}$-enriched category, we denote by $\ca{A}(\mi,\mi)$ the enriching 
hom-object functor, see e.g. \cite[\S~1.6]{Kelly}.

\begin{defi}\label{enrichedfibration}
A fibration $P\colon\ca{A}\to\ca{X}$ is \emph{enriched} in a monoidal fibration $T\colon\ca{V}\to\ca{W}$ when
\begin{itemize}
\item $\ca{A}$ is enriched in $\ca{V}$, $\ca{X}$ is enriched in $\ca{W}$ and the following commutes
\begin{displaymath}
\begin{tikzcd}[column sep=.5in]
\ca{A}^\op\times\ca{A}\ar[r,"{\ca{A}(-,-)}"]\ar[d,"P^\op\times P"'] & \ca{V}\ar[d,"T"] \\
\ca{X}^\op\times\ca{X}\ar[r,"{\ca{X}(-,-)}"']& \ca{W}
\end{tikzcd}
\end{displaymath}
\item composition and identities of the enrichments are compatible, in that
\begin{displaymath}
TM^{\ca{A}}_{A,B,C}=M^{\ca{X}}_{PA,PB,PC}\;\textrm{ and }\; Tj^\ca{A}_A=j^{\ca{X}}_{PA}.
\end{displaymath}
\end{itemize}
\end{defi}

Dually, we have the notion of an enriched opfibration, as well as the following mixed version.
\begin{defi}\label{def:fibenropfib}
If $T\colon\ca{V}\to\ca{W}$ is a monoidal opfibration, a fibration $P\colon\ca{A}\to\ca{X}$ is enriched in the opfibration $T$ if the opfibration 
$P^\op$ is $T$-enriched.
\end{defi} 

\begin{rmk}\label{rmk:braidingnotneeded}
Notice that for the opfibration $P^\op$ to be enriched in $T$, the condition ``braided'' or ``symmetric'' for the monoidal opfibration $T$ (as was 
the case in the original reference \cite{EnrichedFibration}) is not really necessary and thus here dropped. What is gained under those 
extra assumptions is that the $\ca{V}$-enrichment of $\ca{A}^\op$ implies that $(\ca{A}^\op)^\op\cong\ca{A}$ (or equal in the case of symmetry) 
is also $\ca{V}$-enriched, and similarly the $\ca{W}$-enrichment $\ca{X}^\op$ implies that of $\ca{X}$. This is of course what happens in the 
examples of interest, and as a result motivated the original definition.
\end{rmk}

Leading to an action-induced enrichment result similar to \cref{thm:cotensorenrich}, below a pseudomonoid acts on a pseudomodule in the 
cartesian monoidal 2-category of fibrations, \cite[Def.~3.3]{EnrichedFibration}.

\begin{defi}\label{Trepresentation}
A monoidal fibration $T\colon\ca{V}\to\ca{W}$ \emph{acts} on a fibration $P\colon\ca{A}\to\ca{X}$ when there exists a fibred 1-cell
\begin{displaymath}
\begin{tikzcd}
\ca{V}\times\ca{A}\ar[r,"*"]\ar[d,"T\times P"'] & 
\ca{A}\ar[d,"P"] \\
\ca{W}\times\ca{X}\ar[r,"\diamond"'] & \ca{X}
\end{tikzcd}
\end{displaymath}
where $*,\diamond$ are ordinary actions, such that the action constraints are compatible in that
\begin{displaymath}
P\chi^\ca{A}_{XYA}=\chi^\ca{X}_{(TX)(TY)(PA)},\quad  P\nu^\ca{A}_A=\nu^\ca{X}_{PA}
\end{displaymath}
for all $X,Y\in\ca{V}$ and $A\in\ca{A}$. Dually, a monoidal opfibration \emph{acts} on an opfibration.
\end{defi}

\begin{ex}\label{ex:actions}
If $\ca{V}$ is a braided monoidal closed category, the monoidal fibration $\Comod_\ca{V}^\op\to\Comon(\ca{V})^\op$ acts on the fibration $\Mod_\ca{V}\to\Mon(\ca{V})$ via the fibred 1-cell \cref{eq:HomMod}, where both instances of the internal hom are actions and the required compatibilities hold. 
\end{ex}

In \cite[Thm.~3.11]{EnrichedFibration}, it was proved that if $T\colon\ca{V}\to\ca{W}$ is a monoidal fibration which acts on a fibration 
$P\colon\ca{A}\to\ca{X}$ with a parameterized adjoint $(R,S)\colon P^\op\times P\to T$ in $\Cat^\mathbf{2}$, then we can enrich the fibration $P$ in 
the monoidal fibration $T$. The parameterized adjunction condition, see \cite[Thm.~3.4]{EnrichedFibration}, means that there are adjunctions 
$\mi*A\dashv R(A,\mi)$ and $\mi\diamond X\dashv S(X,\mi)$ such that
\begin{equation}\label{eq:parameterizedCat2}
 \begin{tikzcd}[column sep=.6in]
\ca{V}\ar[r,shift left,"{\mi*A}"]\ar[d,"T"'] & \ca{A}\ar[l,shift left,"{R(A,\mi)}"]\ar[d,"P"] \\
\ca{W}\ar[r,shift left,"{\mi\diamond PA}"] & \ca{X}\ar[l,shift left,"{S(PA,\mi)}"]
 \end{tikzcd}
\end{equation}
is a map of adjunctions (what is called `transformation of adjoints' in \cite[\S IV.7]{MacLane}), namely (co)units are above one another. 
We will here use a version of its dual for a monoidal opfibration acting on the opposite of a fibration
due to our subsequent applications. Notice that concerning enriched fibrations, there are no tensored or cotensored versions investigated in the literature yet.

\begin{thm}\label{thm:enrichedfib}
 Suppose that $T\colon\ca{V}\to\ca{W}$ is a 
 monoidal opfibration that acts on the opposite of a fibration $P\colon\ca{A}\to\ca{X}$ via an opfibred 1-cell
\begin{displaymath}
\begin{tikzcd}
\ca{V}\times\ca{A}^\op\ar[r,"\pitchfork^\op"]\ar[d,"T\times P^\op"'] & 
\ca{A}^\op\ar[d,"P^\op"] \\
\ca{W}\times\ca{X}^\op\ar[r,"\square^\op"'] & \ca{X}^\op    
\end{tikzcd}
\end{displaymath}
If $(\pitchfork^\op,\square^\op)$ has a right parameterized adjoint in $\mathsf{Cat}^\mathbf{2}$, then the fibration $P$ is enriched in the monoidal 
opfibration $T$. 
\end{thm}

\begin{proof}[Sketch]
By \cref{def:fibenropfib} we need to show that the opfibration $P^\op$ is enriched in the monoidal opfibration $T$. 
The adjunctions $(\mi\pitchfork^\op A)\dashv R(A,\mi)$ and $(\mi\square^\op X)\dashv S(X,\mi)$ give rise to an enrichment of $\ca{A}^\op$ in $\ca{V}$ and 
of $\ca{X}^\op$ in $\ca{W}$ (as well as of $\ca{A}$ and $\ca{X}$ in case $T$ is braided, by \cref{thm:cotensorenrich})
in such a way that $T\circ\ca{A}^\op(\mi,\mi)=\ca{X}^\op(P\mi,P\mi)$ by the parameterized adjunction condition.
Finally, the corresponding composition and identities of enriched categories are compatible, because their adjuncts are. Therefore all clauses of \cref{enrichedfibration} are satisfied.
\end{proof}

As a result, we obtain the following result (see \cite[Prop.~4.1]{EnrichedFibration}\footnote{\label{note1}Symmetry 
assumptions from the references are adjusted according to \cref{rmk:braidingnotneeded}; braiding in $\ca{V}$ is merely required for $\Comon(\ca{V})$ 
and $\Comod_\ca{V}$ to be monoidal categories.}), formally describing \cref{pic:1}.

\begin{thm}\label{thm:ModenrichedComodfib}
For a braided monoidal closed and locally presentable $\ca{V}$, the fibration $\Mod_\ca{V}\to\Mon(\ca{V})$ is enriched in the monoidal 
opfibration $\Comod_\ca{V}\to\Comon(\ca{V})$.
\end{thm}

\section{Monads and comonads in double categories, revisited}\label{sec:moncomondouble}

This chapter is mixed, in that it provides some general background on double categories, fibrant and monoidal structures, and recalls 
some results about monads and comonads in double categories appearing in \cite{VCocats}. However, it goes much further: the notion of a 
monoidal closed double category is settled, the notion of a locally presentable double category is introduced, and certain parts of the 
development in loc. cit. are much simplified by abstracting some results from the double category $\VMMat$ to a general double category $\dc{D}$.
Its end result, \cref{thm:big1}, generalizes Sweedler theory for (co)algebras (\cref{thm:MonenrichedComon}) from monoidal categories to double categories.

\subsection{Double categories}\label{sec:doublecats}

We recall some basic material from the theory of double categories, following references like \cite{Limitsindoublecats,Framedbicats}.
Moreover, we introduce \emph{parallel limits} and \emph{colimits} in double categories (\cref{def:parallelcolimits}) and study some equivalent formulations in terms of 
horizontal composition in fibrant double categories.

\begin{defi}\label{def:doublecats}
	A \emph{(pseudo) double category} $\dc{D}$
	consists of a category of objects $\dc{D}_0$ and
	a category of arrows $\dc{D}_1$, with identity, source and target, composition structure
	functors 
	\begin{displaymath}
		\B{1}\colon\dc{D}_0\to\dc{D}_1,\quad 
		\Gr{s},\Gr{t}\colon\dc{D}_1\rightrightarrows\dc{D}_0,\quad
		\odot\colon\dc{D}_1{\times_{\dc{D}_0}}\dc{D}_1\to\dc{D}_1
	\end{displaymath}
	such that
	$\Gr{s}(1_X)$=$\Gr{t}(1_X)$=$X,\;\Gr{s}(M\odot N)$=$\Gr{s}(N),\;
	\Gr{t}(M\odot N)$=$\Gr{t}(M)$
	for all $X\in\ob\dc{D}_0$ and $M,N\in\ob\dc{D}_1$,
	equipped with natural isomorphisms
	$a\colon(M\odot N)\odot P\simrightarrow M\odot(N\odot P)$, $\ell\colon1_{\Gr{s}(M)}\odot M\simrightarrow M$, $r\colon M\odot1_{\Gr{t}(M)}\simrightarrow M$
	in $\dc{D}_1$ such that 
	$\Gr{t}(a),\Gr{s}(a),\Gr{t}(\ell),\Gr{s}(\ell),
	\Gr{t}(r),\Gr{s}(r)$ are all
	identities, and satisfying the usual coherence conditions. 
\end{defi}

To set some terminology and notation, the objects and morphisms of $\dc{D}_0$ are the \emph{0-cells}  and 
\emph{vertical 1-cells} $f\colon X\to Y$. The objects of $\dc{D}_1$ are the \emph{horizontal 1-cells}
$M\colon X\bular Y$, and the morphisms of $\dc{D}_1$ are the 
\emph{2-morphisms}
\begin{equation}\label{2morphism}
	\bultwocell{X}{M}{Y}{g}{W}{N}{Z}{f}{\alpha}
\end{equation}
or $^f\alpha^g:M\Rightarrow N$.
Strict (vertical) identities are $\mathrm{id}_X:X\to X$ and $\id_M:M\Rightarrow M$,
and horizontal units are $1_X\colon X\bular X$ and
$1_f:1_X\Rightarrow 1_Y$.
A 2-morphism with identity source and target vertical 1-cell, like $a,\ell,r$ above, is called  \emph{globular}.

The \emph{opposite} double category $\dc{D}^\op$ is the double category with category of objects $\dc{D}_0^\op$ and category of arrows $\dc{D}^\op_1$.
For every double category $\dc{D}$, there is a corresponding \emph{horizontal bicategory} $\ca{H}(\dc{D})$ which essentially comes from 
discarding its vertical structure: it consists of the 0-cells, horizontal 1-cells and globular 2-morphisms. In fact, many bicategories that 
are of interest can be seen as the horizontal bicategory of suitable double categories.

We denote by $^X\dc{D}_1$ and $\dc{D}_1^Y$ the subcategories of $\dc{D}_1$ of fixed-domain or fixed-codomain horizontal 1-cells, and 2-morphisms with identity domain or codomain vertical arrows respectively, namely
\begin{displaymath}
	\begin{tikzcd}
		X\ar[r,bul,"M"]\ar[d,equal]\ar[dr,phantom,"\Two"] & Y\ar[d,"g"] \\
		X\ar[r,bul,"N"'] & W
	\end{tikzcd}\qquad\mathrm{or}\qquad
	\begin{tikzcd}
		X\ar[r,bul,"M"]\ar[d,"f"']\ar[dr,phantom,"\Two"] & Y\ar[d,equal] \\
		Z\ar[r,bul,"N"'] & Y.
	\end{tikzcd}
\end{displaymath}
Furthermore, the subcategories $^X\dc{D}_1^Y$ with fixed-domain and fixed-codomain horizontal 1-cells together with globular 2-morphisms coincide with the
hom-categories of the horizontal bicategory $\ca{H}(\dc{D})(X,Y)$.

Also, we denote by $\dc{D}_1^\bullet$ the subcategory of $\dc{D}_1$ of horizontal endo-1-cells $M\colon X\bular X$ and 2-morphisms with the same 
source and target $^f\alpha^f$, namely the equalizer of $\Gr{s},\Gr{t}\colon\dc{D}_1\rightrightarrows\dc{D}_0$. As 
it will be useful to what follows, notice that as for all equalizers, $\dc{D}_1^\bullet$ can be equivalently written as the pullback
\begin{equation}\label{eq:D1bulpul}
	\begin{tikzcd}
		\dc{D}_{1}^{\bullet}\ar[r]\ar[d]\ullimit & \dc{D}_{1}\ar[d,"{\langle \Gr{s},\Gr{t}\rangle}"] \\
		\dc{D}_{0}\ar[r,"\Delta"'] & \dc{D}_{0}\times\dc{D}_{0}.
	\end{tikzcd}
\end{equation}

Well-known examples of double categories include $\dc{S}\nc{pan}(\ca{C})$ for a category with pullbacks $\ca{C}$ whose horizontal 1-cells are spans,
$\dc{R}\nc{el}(\ca{C})$ for a regular category $\ca{C}$ whose horizontal 1-cells are relations, $\dc{B}\nc{Mod}$ whose horizontal 1-cells are ring 
bimodules, and $\VPProf$ whose horizontal 1-cells are $\ca{V}$-profunctors between $\ca{V}$-categories. The latter's `discrete' version, the double 
category of enriched matrices, is the central example of interest to us later; we elaborate below.

\begin{ex}\label{ex:VMMat}
	There is a double category $\VMMat$, where $\ca{V}$ is a monoidal category with small coproducts preserved by $\mi\otimes\mi$ in both variables.
	Its horizontal bicategory was studied since \cite[\S~6.2]{Distributeurs} or 
	generalized in \cite{Varthrenr} for the theory of bicategory-enriched categories.
	
	In more detail, the vertical category $\VMMat_0$ is just $\nc{Set}$, whereas 
	the horizontal 1-cells are $\ca{V}$-matrices $S\colon X\tickar Y$, namely functors $S\colon Y\times X\to\ca{V}$ where $Y\times X$ is viewed 
	as a discrete category; equivalently, these are families of objects $\{S(y,x)\}_{(y,x)\in Y\times X}$ in $\ca{V}$,
	sometimes also denoted $\{S_{y,x}\}$. The 2-morphisms $^f\alpha^g\colon S\Rightarrow T$ are natural transformations
	\begin{displaymath}
		\begin{tikzcd}[row sep=.1in,baseline=2ex]
			Y\times X\ar[rr,bend left,"S"]\ar[rr,phantom,"\Two\alpha"]\ar[dr,bend right=5,"g\times f"'] && \ca{V} \\
			& Z\times W\ar[ur,bend right=5,"T"'] &
		\end{tikzcd}
	\end{displaymath}
	given by families of arrows $\alpha_{y,x}\colon S(y,x)\to T(g(y),f(x))$ in $\ca{V}$, for all $x\in X$ and $y\in Y$.
	The identity $\ca{V}$-matrix $1_X\colon X\tickar X$  is given by $1_X(x',x)=I$ if $x=x'$ or $0$ otherwise,
	and the horizontal composition functor maps two composable $\ca{V}$-matrices $T\colon Y\tickar Z$ and $S\colon X\tickar Y$ to
	$T\circ S\colon X\tickar Z$ given by
	\begin{displaymath}
		(T\circ S)(z,x)=\sum_{y\in Y} T(z,y)\otimes S(y,x).
	\end{displaymath}
	More details can be found in the provided references or \cite[\S~4.1]{VCocats}, although we here use the opposite convention when defining 
$\ca{V}$-matrices $X\tickar Y$ as functors from $Y\times X$ rather than $X\times Y$. 
The category $\VMMat_1^\bullet$ is the category $\VGrph$ of $\ca{V}$-graphs.
\end{ex}

\begin{defi}\label{defi:doublefunctor}
	For $\dc{D}$ and $\dc{E}$ double categories, a \emph{(pseudo) double functor} $F\colon\dc{D}\to\dc{E}$
	consists of functors $F_0:\dc{D}_0\to\dc{E}_0$ and $F_1:\dc{D}_1\to\dc{E}_1$ such that $\Gr{s}\circ F_1=F_0\circ\Gr{s}$ 
	and $\Gr{t}\circ F_1=F_0\circ\Gr{t}$, and natural transformations $\phi$, $\phi_0$ 
	with components globular isomorphisms
	$F_1M\odot F_1N\simrightarrow F_1(M\odot N)\textrm{ and }1_{F_0X}\simrightarrow F_1(1_X)$
	which satisfy coherence axioms. If these isomorphisms are non-invertible in 
	that or the reverse direction, we have the notion of a \emph{lax} or \emph{oplax} double functor.
\end{defi}
We usually drop indices $0,1$ and write $FM\colon FX\bular FY$ and $^{Ff}(F\alpha)^{Fg}\colon FM\Rightarrow FN$.

In most interesting examples of double categories, vertical 1-cells can be turned to horizontal 1-cells in a canonical way, leading to the theory of 
fibrant double categories or framed bicategories \cite{Framedbicats}.

\begin{defi}\label{def:compconj}
	Let $\dc{D}$ be a double category and $f:X\to Y$
	a vertical 1-cell. A \emph{companion}, resp. \emph{conjoint} of $f$
	is a horizontal 1-cell $\wh{f}\colon X\bular Y$, resp. $\wc{f}\colon Y\bular X$ together with
	2-morphisms
	\begin{displaymath}
		\begin{tikzcd}[ampersand replacement=\&,sep=.3in]
			X\ar[r,bul,"\wh{f}"]\ar[d,"f"']\ar[dr,phantom,"\Two{p_1}"] \& Y\ar[d,equal] \\
			Y\ar[r,bul,"{1_Y}"'] \& Y
		\end{tikzcd}\qquad
		\begin{tikzcd}[ampersand replacement=\&,sep=.3in]
			X\ar[r,bul,"1_X"]\ar[d,equal]\ar[dr,phantom,"\Two{p_2}"] \& X\ar[d,"f"] \\
			X\ar[r,bul,"{\wh{f}}"'] \& Y
		\end{tikzcd}\;\;\textrm{ resp. }\;\; \begin{tikzcd}[ampersand replacement=\&,sep=.3in]
			Y\ar[r,bul,"\wc{f}"]\ar[d,equal]\ar[dr,phantom,"\Two{q_1}"] \& X\ar[d,"f"] \\
			Y\ar[r,bul,"{1_Y}"'] \& Y
		\end{tikzcd}
		\qquad
		\begin{tikzcd}[ampersand replacement=\&,sep=.3in]
			X\ar[r,bul,"1_X"]\ar[d,"f"']\ar[dr,phantom,"\Two{q_2}"] \& X\ar[d,equal] \\
			Y\ar[r,bul,"\wc{f}"'] \& X
		\end{tikzcd}
	\end{displaymath}
	such that $p_1p_2=1_f$ and $p_1\odot p_2\cong1_{\wh{f}}$, resp. $q_1q_2=1_f$ and $q_2\odot q_1\cong 1_{\wc{f}}$.
\end{defi}
A \emph{fibrant double category} then is a double category
for which every vertical 1-morphism has a companion and a conjoint.
Equivalently, it is a double category $\dc{D}$ such that the functor
\begin{equation}\label{eq:stbifibration}
	\langle\Gr{s},\Gr{t}\rangle\colon\dc{D}_1\longrightarrow\dc{D}_0\times\dc{D}_0
\end{equation}
is a fibration, or equivalently an opfibration. For example, the cartesian liftings for any horizontal map $N\colon Z\bular W$ and vertical maps $f\colon X\to Z,g\colon Y\to W$
are of the form
\begin{equation}\label{eq:D1targetlifts}
\begin{tikzcd}
X\ar[r,bul,"\wh{f}"]\ar[d,"f"']\ar[dr,phantom,"\Two p_1"] & Z\ar[r,bul,"N"]\ar[d,equal] & W\ar[d,equal]\ar[r,bul,"\wc{g}"]\ar[dr,phantom,"\Two q_2"] 
& Y\ar[d,"g"] \\
Z\ar[r,bul,"1_X"'] & Z\ar[r,bul,"N"'] & W\ar[r,bul,"1_W"'] & W.
\end{tikzcd}
\end{equation} 
By appropriately pasting the structure 2-morphisms $p_i,q_j$ 
of \cref{def:compconj} at the sides of any 2-morphism $^f\alpha^g$ \cref{2morphism},
we can verify bijections with globular 2-morphisms
\begin{equation}\label{eq:glob2map}
	\begin{tikzcd}
		X\ar[r,bul,"M"]\ar[d,equal]\ar[drr,phantom,"\Two\wh{\alpha}"] & Y\ar[r,bul,"\wh{g}"] & W\ar[d,equal] \\
		X\ar[r,bul,"\wh{f}"'] & Z\ar[r,bul,"N"'] & W
	\end{tikzcd}\quad
	\begin{tikzcd}
		Z\ar[r,bul,"\wc{f}"]\ar[d,equal]\ar[drr,phantom,"\Two\wc{\alpha}"] & X\ar[r,bul,"M"] & Y\ar[d,equal] \\
		Z\ar[r,bul,"N"'] & W\ar[r,tick,"\wc{g}"'] & Y
	\end{tikzcd}
\end{equation}
\begin{displaymath}
	\begin{tikzcd}
		X\ar[d,equal]\ar[rrr,bul,"M"]\ar[drrr,phantom,"\Two"] &&& Y\ar[d,equal] \\
		X\ar[r,bul,"\wh{f}"'] & Z\ar[r,bul,"N"'] & W\ar[r,bul,"\wc{g}"'] & Y
	\end{tikzcd}
	\quad
	\begin{tikzcd}
		Z\ar[d,equal]\ar[r,bul,"\wc{f}"]\ar[drrr,phantom,"\Two"] & X\ar[r,bul,"M"] & Y\ar[r,bul,"\wh{g}"] & W\ar[d,equal] \\
		Z\ar[rrr,bul,"N"'] &&& W
	\end{tikzcd}
\end{displaymath}
The following lemma recalls some properties of companions and conjoints; proofs can be found in \cite{Adjointfordoublecats,ConstrSymMonBicatsFun}.

\begin{lem}\label{lem:fibrantproperties}
	Suppose $f\colon X\to Y$ is a vertical morphism in a fibrant double category $\dc{D}$.
	 \begin{enumerate}[(i),ref=\thedefi(\roman*)]
		\item Companions and conjoints of $f$ are unique up to (unique) globular isomorphism.
		\item If $f$ is an isomorphism, its structure 2-morphisms $p_1,p_2,q_1,q_2$ are invertible.\label{fibr-ii}
		\item There is an adjunction $\wh{f}\dashv\wc{f}$ in the horizontal bicategory $\ca{H}(\dc{D})$. It becomes an adjoint equivalence when $f$ 
is an isomorphism.
		\item For two composable vertical 1-cells, $\wh{g}\odot\wh{f}$ is a companion and $\wc{f}\odot\wc{g}$ is a conjoint of $gf$.
		\item If $F\colon\dc{D}\to\dc{E}$ is a double functor, $F(\wh{f})$ is a companion and $F(\wc{f})$ is a conjoint of $Ff$.
	\end{enumerate}
\end{lem}

\begin{ex}
	Given any regular category $\ca{C}$, the double category $\dc{R}\nc{el}(\ca{C})$ is fibrant. Indeed, any morphism $f\colon X\to Y\in\ca{C}$ 
defines a relation $f\colon X\bular Y$ via its graph and also a relation $f^{\circ}\colon Y\bular X$ which is the opposite of the latter. In this 
case the 2-cells $p_{1},p_{2}$ are identities, while $q_{1},q_{2}$ are respectively the inclusions $ff^{\circ}\subseteq 1_{Y}$ and $1_{X}\subseteq 
f^{\circ}f$ expressing the familiar adjunction $f\dashv f^{\circ}$.
\end{ex}

\begin{ex}\label{ex:VMMatfibrant}
	The double category $\VMMat$ from \cref{ex:VMMat} is fibrant. Indeed, any function $f:X\to Y$ canonically determines two 
	$\ca{V}$-matrices $f_*\colon X\tickar Y$ and $f^*\colon Y\tickar X$ given by
	\begin{displaymath}
		f_*(y,x)=f^*(x,y)=\begin{cases}
			I,\quad \mathrm{if  }\;f(x)=y\\
			0,\quad \mathrm{ otherwise}
		\end{cases}
	\end{displaymath}
	with appropriate structure 2-cells, see e.g. \cite[Eq.~37]{VCocats}.
\end{ex}

The following observations are useful in what comes next, but also later in \cref{sec:ModsComodsdoublecats}. 

\begin{prop}\label{globalvslocal}
Suppose $\dc{D}$ is a fibrant double category. Then $\Gr{s}$ and $\Gr{t}$ are bifibrations $\dc{D}_1\to\dc{D}_0$, that restrict to bifibrations 
$\Gr{s}\colon\dc{D}^Y_1\to\dc{D}_0$ and $\Gr{t}\colon^X\dc{D}_1\to\dc{D}_0$ for any object $X,Y$. 
\end{prop}

\begin{proof}
The source and target functors are separately both bifibrations, formed as
\begin{equation}\label{eq:stbifibrations}
	\begin{tikzcd}
		\dc{D}_1\ar[r,"{\langle\Gr{s},\Gr{t}\rangle}","\cref{eq:stbifibration}"'] & \dc{D}_0\times\dc{D}_0\ar[r,shift left,"\pi_1"]\ar[r,shift 
right,"\pi_2"'] & 
		\dc{D}_0 
	\end{tikzcd}
\end{equation}
since (op)fibrations compose and projections are (op)fibrations in a straightforward way, see e.g. \cite[Prop.~8.1.12\&13]{Handbook2}. 
In more detail, the reindexing functors of $\dc{D}_1\xrightarrow{\Gr{s}}\dc{D}_0$ for some $f\colon X\to Y$ are 
$^X\dc{D}_1\xleftarrow{\mi\odot\wh{f}}{}^Y\dc{D}_1$ for the fibration and $^X\dc{D}_1\xrightarrow{\mi\odot\wc{f}}{}^Y\dc{D}_1$ for the opfibration 
structure, where of course $(\mi\odot\wc{f})\dashv(\mi\odot\wh{f})$. Similarly, the reindexing functors of the bifibration 
$\dc{D}_1\xrightarrow{\Gr{t}}\dc{D}_0$ for $g\colon Z\to W$ are 
\begin{equation}\label{eq:reindexingt}
\begin{tikzcd}
	\dc{D}^Z_1\ar[r,shift left=2,"\wh{g}\odot\mi"]\ar[r,phantom,"\bot"] &
	\dc{D}^W_1.\ar[l,shift left=2,"\wc{g}\odot\mi"]
\end{tikzcd}
\end{equation}
Furthermore, both functors restrict to sub-bifibrations. 
For example, $^X\dc{D}_{1}$ is closed under $\Gr{t}$-cocartesian lifts inside $\dc{D}_1$, 
since $\wh{g}\odot M$ has the same domain as $M$ and the lift is identity on the left. Notice that the fiber above any $Z\in\dc{D}_0$ is 
$^X\dc{D}_1^Z=\ca{H}(\dc{D})(X,Z)$, and similarly for $\Gr{s}$.
\end{proof}

In the present work we will also need to consider (co)limits in various categories arising in the setting of double categories, and so it is 
pertinent to introduce the following definition, of essentially a graph internal in the category of $\ca{I}$-(co)complete categories and 
$\ca{I}$-(co)limit preserving functors. 

\begin{defi}\label{def:parallelcolimits}
Let $\dc{D}$ be a double category and $\ca{I}$ is any small category. We shall say that $\dc{D}$ has \emph{parallel 
$\ca{I}$-(co)limits} if the categories $\dc{D}_0$, $\dc{D}_1$ both have $\ca{I}$-(co)limits and the functors $\Gr{s},\Gr{t}\colon\dc{D}_1\to\dc{D}_0$ 
preserve them. If $\dc{D}$ has parallel $\ca{I}$-(co)limits for any small $\ca{I}$, then we say that it is \emph{parallel 
(co)complete}.
\end{defi}

Regarding terminology, `parallel' indicates that this is not a general notion of double-categorical (co)limit e.g. in the sense of 
\cite{Limitsindoublecats}, but a more specific one concerning the categories $\dc{D}_1$ and $\dc{D}_0$. In particular, the term was adopted
from the special case of parallel (co)products appearing in \cite{Parallelproducts} (mentioned in a talk by Bob Par\'e in 2009).

\begin{prop}\label{ex:parallelcolimitsinVMMat}
The double category $\VMMat$ has parallel coproducts.
If, furthermore, $\ca{V}$ has all colimits and they are preserved by $\otimes$ in each variable, then $\VMMat$ has all parallel colimits. 
\end{prop}

\begin{proof}
The category of objects $\VMMat_0=\Set$ has coproducts, and for a small family of $\ca{V}$-matrices $(M_{i}\colon X_{i}\tickar X'_{i})_{i\in 
I}$ we define their coproduct to be the matrix $\bigsqcup 
M_{i}\colon\bigsqcup X_{i}\tickar\bigsqcup X'_{i}$ given by	
\begin{gather*}
\bigsqcup M_{i}(x',x)\coloneqq
\begin{cases}
	M_{i}(x',x),\quad \mathrm{if }\;x\in X_i, x'\in X'_{i}\\
	0,\quad \mathrm{otherwise }.
\end{cases}
\end{gather*}

There are obvious inclusion 2-morphisms
\begin{displaymath}
\begin{tikzcd}
	X_i\ar[d,"s_i"']\ar[r,bul,"M_i"]\ar[dr,phantom,"\Two\sigma_i"] & X'_{i}\ar[d,"s'_{i}"] \\
	\bigsqcup X_i\ar[r,bul,"\bigsqcup M_{i}"'] & \bigsqcup X'_{i}
\end{tikzcd}
\end{displaymath}
which are given by identity morphisms $M_{i}(x',x)\to M_{i}(x',x)$ in $\ca{V}$, for all $i\in I$ and $x\in X_i$, $x'\in X'_{i}$. It can then be 
verified that $\bigsqcup M_i$ has the required universal property in $\VMMat_1$: for any $\ca{V}$-matrix $N\colon Y\bular Y'$ together 
with 2-morphisms $^{f_{i}}\phi_{i}^{f'_{i}}\colon M_i\Rightarrow N$, we consider the induced functions 
$f\coloneqq(f_i)\colon\bigsqcup X_i\to Y$ and $f'\coloneqq(f'_i)\colon\bigsqcup X'_i\to Y'$ 
and uniquely define a factorization $^{f}\phi^{f'}\colon\bigsqcup M_i\to N$ by setting
\begin{gather*}
	\phi_{x,x'}\coloneqq
	\begin{cases}
		(\phi_i)_{x',x},\quad \mathrm{if }\;x\in X_i, x'\in X'_{i}\\
		0\xrightarrow{!} N(f'_j(x'),f_i(x)),\quad \mathrm{if }\;x\in X_i, x'\in X'_{j}, i\neq j.
	\end{cases}
\end{gather*}

For the second part of the proposition, it suffices to show that under the further assumptions, $\VMMat$ also has parallel coequalizers. Consider a pair 
of morphisms in $\VMMat_1$
\begin{displaymath}
\begin{tikzcd}
	X\ar[r,bul,"A"]\ar[d,"f"']\ar[dr,phantom,"\Two\phi"] & X'\ar[d,"{f'}"] \\
	Y\ar[r,bul,"B"'] & Y'
\end{tikzcd}
\qquad
\begin{tikzcd}
	X\ar[r,bul,"A"]\ar[d,"g"']\ar[dr,phantom,"\Two\psi"] & X'\ar[d,"{g'}"] \\
	Y\ar[r,bul,"B"'] & Y'
\end{tikzcd}
\end{displaymath}
and first form the coequalizers $\begin{tikzcd}
	X\ar[r,shift left=1ex,"f"]\ar[r,shift right,"g"'] & Y\ar[r,two heads,"q"] & Q,
\end{tikzcd}$ $\begin{tikzcd}
	X'\ar[r,shift left=1ex,"{f'}"]\ar[r,shift right,"{g'}"'] & Y'\ar[r,two heads,"{q'}"] & Q'
\end{tikzcd}$ in $\nc{Set}$.
It can be verified that the following 2-morphism is the coequalizer of $\phi,\psi$ in $\VMMat_1$
\begin{displaymath}
\begin{tikzcd}
	Y\ar[r,bul,"B"]\ar[d,"q"']\ar[dr,phantom,"\Two\omega"] & Y'\ar[d,"{q'}"] \\
	Q\ar[r,bul,"C"'] & Q'
\end{tikzcd}
\end{displaymath}
where $C$ and $\omega$ are defined as follows: for each $u\in Q, u'\in Q'$, form the coequalizer in $\ca{V}$
\begin{displaymath}
\begin{tikzcd}
	\sum\limits_{\substack{qf(x)=u \\ q'f'(x')=u'}}A(x',x)\ar[r,shift left=1ex,"\sum\phi_{x,x'}"]\ar[r,shift right=1ex,"\sum\psi_{x,x'}"'] & 
	\sum\limits_{\substack{q(y)=u \\ q'(y')=u'}}B(y',y)\ar[r,two heads,"{\omega(u,u')}"] & C(u',u)
\end{tikzcd}
\end{displaymath}
and for any $y\in Y,y'\in Y'$, $\omega_{y,y'}\colon B(y',y)\to C(q'(y'),q(y))$ is the following composite in $\ca{V}$
\begin{displaymath}
\begin{tikzcd}
	B(y',y)\ar[r] & \sum\limits_{\substack{q(z)=q(y) \\ q'(z')=q'(y')}}B(z',z)\ar[rr,two heads,"{\omega(q(y),q'(y'))}"] && C(q'(y'),q(y)).
\end{tikzcd}
\end{displaymath}
\end{proof}

The level of generality of \cref{def:parallelcolimits} is enough for certain categories of interest to 
us, like (co)monads and (co)modules in double categories, to inherit (co)limits -- see \cref{prop:(co)limits in (co)monads,prop:Mod(D)complete}. 
Those in fact depend on the following result.

\begin{prop}\label{lem:(co)limits in D_1 bullet}
Let $\dc{D}$ be a fibrant double category which has parallel $\ca{I}$-(co)limits, for some small category $\ca{I}$. Then $\dc{D}_{1}^{\bullet}$ has 
$\ca{I}$-(co)limits and the inclusion $\dc{D}_{1}^{\bullet}\to\dc{D}_1$ preserves them.
\end{prop}

\begin{proof}
Consider a diagram $D\colon\ca{I}\to\dc{D}_{1}^{\bullet}$ whose colimit cocone in $\dc{D}_{1}$ is 
\begin{displaymath}
	\begin{tikzcd}
		X_{i}\ar[r,bul,"Di"]\ar[d,"f_{i}"']\ar[dr,phantom,"\Two\phi_{i}"] & X_{i}\ar[d,"g_{i}"] \\
		X\ar[r,bul,"M"'] & Y
	\end{tikzcd}
\end{displaymath}
Since $\Gr{s},\Gr{t}$ preserve colimits, both $f_i$ and $g_i$
are colimiting cocones in $\dc{D}_0$, however of the same diagram in $\dc{D}_0$ since the original $D$ lies in $\dc{D}_{1}^{\bullet}$. Hence there is 
an isomorphism $h\colon Y\cong X$ such that $hg_i=f_i$ for all $i\in\ca{I}$. This allows us to post-compose the the initial colimit $\phi_i$ by the 
following isomorphism in $\dc{D}_1$
\begin{equation}\label{eq:normalizing}
	\begin{tikzcd}
		X\ar[d,equal]\ar[rr,bul,"M"]\ar[drr,phantom,"\Two\cong"] && Y\ar[d,equal] \\
		X\ar[d,equal]\ar[r,bul,"M"] & Y\ar[d,equal]\ar[r,bul,"1_Y"]\ar[dr,phantom,"\Two p_2"] & Y\ar[d,"h"] \\
		X\ar[r,bul,"M"'] & Y\ar[r,bul,"\wh{h}"'] & X
	\end{tikzcd}
\end{equation}	
where $p_2$ is invertible since $h$ is, by Lemma \ref{fibr-ii}. Hence the new colimiting cocone lies wholly in 
$\dc{D}_{1}^{\bullet}$, and has the required universal 
property.
\end{proof}

\begin{rmk}\label{rem:D1bul}
We here provide a high-level argument for the above result, that will be used in analogous ways in later results as well.
As seen in \cref{eq:D1bulpul}, $\dc{D}_1^\bullet$ is a pullback along $\langle\Gr{s},\Gr{t}\rangle$ which in the setting of a fibrant 
double category is a fibration \cref{eq:stbifibration}. By a standard result (see e.g. \cite{Joyal1993}), since every Grothendieck fibration is an isofibration,
the (strict) pullback is equivalent to a \emph{pseudopullback}. Since (co)complete categories with 
(co)continuous functors can both be expressed as 2-categories of algebras and pseudo-morphisms $T$-$\mathsf{Alg}$ for a 2-monad on $\mathsf{Cat}$, 
they have all pseudolimits by the well-known \cite[Thm.~2.6]{2-dimmonadtheory}. 
In particular, since $\langle\Gr{s},\Gr{t}\rangle$ is (co)continuous under these assumptions, and so is the diagonal $\Delta$, their (pseudo)pullback  
$\dc{D}_1^\bullet$ is a (co)complete category.
\end{rmk}

The following result connects the existence of parallel (co)limits in a fibrant double category, to the existence of (co)limits to the various 
fixed domain/codomain subcategories of $\dc{D}_1$, as well as the hom-categories of the horizontal bicategory $\ca{H}(\dc{D})$, in a natural way.

\begin{prop}\label{prop:equivalentdefparcom}
Suppose $\dc{D}$ is a fibrant double category such that $\dc{D}_0$ is complete. The following are equivalent:
\begin{enumerate}[(i),ref=\thedefi(\roman*)]
 \item $\dc{D}$ is parallel complete;
 \item The fibrations $\Gr{s},\Gr{t}\colon\dc{D}_1\to\dc{D}_0$ have all fibred limits.
 
 Namely, $^{X}\dc{D}_{1}$ and $\dc{D}_{1}^{Z}$ are complete categories for any $X,Z\in\dc{D}_0$, and $\mi\odot\wh{f}\colon 
^Y\dc{D}_1\to{}^X\dc{D}_1$ and $\wc{g}\odot\mi\colon\dc{D}_1^W\to\dc{D}_1^Z$ are continuous functors for any $f\colon X\to Y$ and $g\colon Z\to W$;
 \item $\ca{H}(\dc{D})(X,Z)$ is a complete category for any $X,Z\in\dc{D}_0$, and 
$\mi\odot\wh{f}\colon\ca{H}(\dc{D})(Y,Z)\to\ca{H}(\dc{D})(X,Z)$ and $\wc{g}\odot\mi\colon\ca{H}(\dc{D})(X,W)\to\ca{H}(\dc{D})(X,Z)$ are continuous 
functors.
\end{enumerate}
Dually, if $\dc{D}_0$ is cocomplete, the following are equivalent:
\begin{enumerate}[(i),ref=\thedefi(\roman*)]
 \item $\dc{D}$ is parallel cocomplete;
 \item The opfibrations $\Gr{s},\Gr{t}\colon\dc{D}_1\to\dc{D}_0$ have all opfibred colimits.
 
 Namely, $^{X}\dc{D}_{1}$ and $\dc{D}_{1}^{Z}$ are cocomplete for any $X,Z\in\dc{D}_0$, and 
$\mi\odot\wc{f}\colon{}^X\dc{D}_1\to{}^Y\dc{D}_1$ and $\wh{g}\odot\mi\colon\dc{D}_1^Z\to\dc{D}_1^W$ are cocontinuous functors for any $f\colon X\to Y$ and 
$g\colon Z\to W$;
\item $\ca{H}(\dc{D})(X,Z)={}^X\dc{D}_1^Z$ is a cocomplete category for any $X,Z\in\dc{D}_0$, and 
$\mi\odot\wc{f}\colon{}^X\dc{D}_1^Z\to{}^Y\dc{D}_1^Z$ and $\wh{g}\odot\mi\colon{}^X\dc{D}_1^Z\to{}^X\dc{D}_1^W$ are cocontinuous 
functors.
\end{enumerate}
\end{prop}

\begin{proof}
We will verify the version for limits; the version for colimits is dual. 

$(i)\Leftrightarrow (ii)$ This follows essentially from \cref{prop:fiberwiselimits}. In more detail, if $\dc{D}_1$ is complete and $\Gr{s},\Gr{t}$ are continuous functors as per 
\cref{def:parallelcolimits}, the fibration $\Gr{s}$ of \cref{globalvslocal} has all fibred limits, which means that all fibers $^{X}\dc{D}_1$ have limits 
and the reindexing functors $\mi\odot\wh{f}\colon^{Y}\dc{D}_1\to{}^{X}\dc{D}_1$ for any $f\colon X\to Y$ are continuous. Similarly for the fibration $\Gr{t}$. 

Conversely, if all $^X\dc{D}_1$ have limits and all appropriate $\mi\odot\wh{f}$ preserve them, this means that the fibration $\Gr{s}$ is complete. Since its base category $\dc{D}_0$ is also complete, it follows from \cref{prop:fiberwiselimits} that $\dc{D}_1$ has all limits and $\Gr{s}$ preserves them. Similarly, $\Gr{t}$ is also continuous hence $\dc{D}$ is parallel complete. 

$(ii)\Rightarrow (iii)$
It can be verified that the restricted fibration 
$\Gr{s}\colon\dc{D}_1^Z\to\dc{D}_0$ 
from \cref{globalvslocal} is continuous. Indeed, $\dc{D}_1^Z$ is a fiber of the target fibration, which has all fibred limits, hence it is complete. Furthermore, connected limits
are preserved under the inclusion $\dc{D}_1^{Z}\hookrightarrow\dc{D}_1$ by \cref{connected}. Given the equivalence of (1) and (2), we know that $\Gr{s}$ preserves all limits, hence so does the composite $\dc{D}_1^{Z}\hookrightarrow\dc{D}_1\to\dc{D}_0$. For discrete limits, by their general construction 
in \cref{eq:fiberlimit} and the form of the reindexing functors of $\Gr{t}$ \cref{eq:reindexingt} namely $\Delta^*=\wc{\Delta}\odot\mi$, it is the case that the source of the limit horizontal 1-cell
remains unchanged.
 
Thus, by \cref{prop:fiberwiselimits} the fibers $\ca{H}(\dc{D})(X,Z)$ of the continuous $\Gr{s}\colon\dc{D}_1^Z\to\dc{D}_0$  are also complete
and the restricted reindexing functors $\mi\odot\wc{f}$ preserve all limits. In a similar fashion, the restricted fibration $\Gr{t}\colon{}^X\dc{D}_1\to\dc{D}_0$ can be seen to preserve all 
limits, so also the restricted $\wh{g}\odot\mi$ on the horizontal bicategory hom-categories also preserve all limits.

$(iii)\Rightarrow(ii)$ Since every $\mi\odot\wh{f}\colon\ca{H}(\dc{D})(Y,Z)\to\ca{H}(\dc{D})(X,Z)$ is continuous and $\dc{D}_0$ is complete, the 
restricted 
fibration $\Gr{s}\colon\dc{D}_1^Z\to\dc{D}_0$ from \cref{globalvslocal} is complete for all objects $Z$, which in particular means that all 
$\dc{D}_1^Z$ are complete. Now we can form the following commutative triangle
\begin{displaymath}
 \begin{tikzcd}
\dc{D}_1^W\ar[rr,"\wc{g}\odot\mi"]\ar[dr,"\Gr{s}"'] && \dc{D}_1^Z\ar[dl,"\Gr{s}"] \\
& \dc{D}_0.
 \end{tikzcd}
\end{displaymath}
Since $(\wh{g}\odot\mi)\dashv(\wc{g}\odot\mi)$ is such that the counit is $\Gr{s}$-vertical, they form an adjunction in $\mathsf{Cat}/\dc{D}_0$ and so
by a standard argument where right adjoints in the slice category are cartesian (see e.g. \cite[Lem.~4.5]{Winskel}), the top arrow preserves cartesian liftings. As a result, by \cref{prop:continuous fibred 1-cells} 
in the special case that the functor between the bases is the identity, the fact that the fiberwise 
$\wc{g}\odot\mi\colon\ca{H}(\dc{D})(X,W)\to\ca{H}(\dc{D})(X,Z)$ are continuous by assumption implies that the ``global'' functor 
$\wc{g}\odot\mi\colon\dc{D}_1^W\to\dc{D}_1^Z$ is continuous.

In an analogous way, now starting from the fact that $\wc{g}\odot\mi\colon\ca{H}(\dc{D})(X,W)\to\ca{H}(\dc{D})(X,Z)$ is continuous and $\dc{D}_0$ is 
complete, we can deduce that the restricted fibration $\Gr{t}\colon {}^X\dc{D}_1\to\dc{D}_0$ from \cref{globalvslocal} is complete thus ensuring that 
each ${}^X\dc{D}_1$ is complete, which can then be used to form a fibred functor
\begin{displaymath}
 \begin{tikzcd}
^Y\dc{D}_1\ar[rr,"\mi\odot\wh{f}"]\ar[dr,"\Gr{t}"'] && ^X\dc{D}_1\ar[dl,"\Gr{t}"] \\
& \dc{D}_0
 \end{tikzcd}
\end{displaymath}
since $\mi\odot\wh{f}$ is again a right adjoint over $\dc{D}_0$. The fiberwise continuous restrictions of $\mi\odot\wh{f}$ from the assumptions then 
make the functor between the total categories continuous. 
\end{proof}

\begin{rmk}\label{rmk:parcocomplete}
One extra advantage of the above equivalent formulations of \cref{prop:equivalentdefparcom} in the fibrant setting is that they connect the notion 
of a parallel cocomplete double category to the more standard one of a ``locally cocomplete bicategory'', see e.g. \cite{Varthrenr}. In more 
detail, if a fibrant 
double category $\dc{D}$ is such that $\ca{H}(\dc{D})$ is locally cocomplete, namely all hom-categories are cocomplete and all functors of the form 
$M\odot\mi\odot N$ preserve colimits, condition (3) above is clearly verified when $\dc{D}$ is also ``vertically'' cocomplete. This is 
important because checking that a horizontal bicategory of a double category is locally cocomplete in the usual sense could be much easier. For 
example, (co)limits in the hom-categories $\VMat(X,Y)=[Y\times X,\ca{V}]$ are computed pointwise, and horizontal composition preserves colimits by 
usual assumptions on the monoidal structure of $\ca{V}$, hence we recover that $\VMat$ is parallel cocomplete -- \cref{ex:parallelcolimitsinVMMat}. 
\end{rmk}

\begin{ex}\label{ex:colimitsinVMMatfibers}
Since they will be used later, let us explicitly describe the colimits in the fixed-domain subcategories of $\VMMat_1$. Consider a diagram 
$D\colon\ca{I}\to\VMMat_{1}^{Y}$ depicted as
\begin{displaymath}
\begin{tikzcd}
	X_i\ar[r,bul,"Di"]\ar[d,"f_{\alpha}"']\ar[dr,phantom,"\Two\delta_{\alpha}"] & Y\ar[d,equal] \\
	X_j\ar[r,bul,"Dj"'] & Y
\end{tikzcd}
\end{displaymath}
for $\alpha\colon i\to j\in\ca{I}$. First, we form the colimit for $\ca{I}\xrightarrow{D}\VMMat_1^Y\xrightarrow{\mathfrak{s}}\Set$ and call it 
$(q_{i}\colon X_{i}\to X)_{i\in\ca{I}}$. We then define a $\ca{V}$-matrix $C\colon X\tickar Y$ by
$$C(y,x)\coloneqq\mathrm{Colim}_i\left(\sum\limits_{q_{i}(x_i)=x}
Di(y , x_i)\right)\in\ca{V}$$
which is in fact the desired colimit of $D$, with colimit cocone
\begin{displaymath}
\begin{tikzcd}
	X_i\ar[r,bul,"Di"]\ar[d,"q_i"']\ar[dr,phantom,"\Two\gamma^{i}"] & Y \\
	X\ar[r,bul,"C"'] & Y\ar[u,equal]
\end{tikzcd}\textrm{ where }\gamma^{i}_{x_{i},y}\colon D_{i}(y,x_{i})\to\sum\limits_{\substack{q_{i}(x_{i}')\\=q_{i}(x_i)}}Di(y,
x_i)\hookrightarrow C(y,q_{i}(x_i)).
\end{displaymath}
\end{ex}


\subsection{Monoidal closed double categories}\label{sec:monoidalclosed}

We now move to the context of monoidal double categories, see e.g. \cite[Def.~2.10]{ConstrSymMonBicatsFun} for a detailed description. We then describe in detail a monoidal closed double structure pertinent to our purposes.

\begin{defi}\label{defi:monoidaldoublecategory}
	A \emph{monoidal double category} is a double category $\dc{D}$ such that $(\dc{D}_0,\otimes_0,I)$ and 
	$(\dc{D}_1,\otimes_1,1_I)$ are monoidal categories, $\Gr{s},\Gr{t}$ are strict monoidal functors, and there exist globular isomorphisms
	\begin{equation}\label{eq:monoidaldoubleiso}
		(M\otimes_1 N)\odot(M'\otimes_1 N')\cong
		(M\odot M')\otimes_1(N\odot N'), \quad
		1_{(X\otimes_0 Y)}\cong
		1_X\otimes_1 1_Y
	\end{equation}
	subject to coherence conditions.
\end{defi}
The natural isomorphisms come from the fact that there are pseudo double functors $\otimes\colon\dc{D}\times\dc{D}\to\dc{D}$ and 
$I\colon\B{1}\to\dc{D}$ expressing that a monoidal double category 
is a pseudomonoid in a suitable cartesian 2-category of double categories.
A \emph{braided} or \emph{symmetric} monoidal double category $\dc{D}$ is one for which $\dc{D}_0$, $\dc{D}_1$ are
braided or symmetric, and $\Gr{s},\Gr{t}$ are strict braided monoidal, subject to two more axioms expressing compatibility of the braiding with the 
structure isomorphisms.

\begin{ex}\label{ex:VMMatmonoidal}
	The double category $\VMMat$ from \cref{ex:VMMat} is monoidal when $\ca{V}$ is moreover braided. The vertical category is the cartesian 
$(\B{Set},\times,\{*\})$ and $\VMMat_1$ is monoidal via
	\begin{displaymath}
		\otimes_1\colon\VMMat_1\times\VMMat_1\xrightarrow{\phantom{AAAAAAA}}\VMMat_1\phantom{AAAAA}
	\end{displaymath}
	\begin{displaymath}
		\ticktwocell{X}{S}{Y}{g}{Y',}{S'}{X'}{f}{\alpha}
		\ticktwocell{Z}{T}{W}{k}{W'}{T'}{Z'}{h}{\beta}\mapsto
		\ticktwocell{X{\times}Z}{S{\ot}T}{Y{\times}W}{g{\times}k}{Y'{\times}W'}{S'{\ot}T'}{X'{\times}Z'}{f{\times}h}{\alpha{\ot}\beta}
	\end{displaymath}
	where $(S\otimes T)\left((y,w),(x,z)\right):=S(y,x)\otimes T(w,z)$,
	and with unit $\ca{I}\colon\{*\}\bular\{*\}$ given by	$\ca{I}(*,*)=I_\ca{V}$.
	The isomorphisms \cref{eq:monoidaldoubleiso} come down to the tensor product in $\ca{V}$ being braided and commuting with coproducts, see 
\cite[\S~4.1]{VCocats}. 

Notice that $\VMMat$ is not a braided monoidal double category under these assumptions, because the braiding fails to become a transformation of double categories. However,
if the braiding of $\ca{V}$ is a symmetry, that extra axiom is satisfied and $\VMMat$ becomes a symmetric monoidal double category.
\end{ex}

\begin{ex}
 The double category $\dc{R}\nc{el}(\ca{C})$ for any regular $\ca{C}$ is (cartesian) monoidal. The product of two relations 
$R\colon X\bular Y$ and $R'\colon X'\bular Y'$ represented by the jointly monic $X\xleftarrow{r_0}R\xrightarrow{r_1}Y$
and $X'\xleftarrow{r'_0}R'\xrightarrow{r'_1}Y'$ respectively is the 
relation $R\times R'\colon X\times X'\bular 
Y\times Y'$ represented by $X\times X'\xleftarrow{r_0\times r_{0}^{'}}R\times R'\xrightarrow{r_1\times r_{1}^{'}}Y\times Y'$. 
The globular isomorphisms required are the equalities $(S\times S')(R\times R')=(SR)\times (S'R')$ and $\Delta_{X\times 
Y}=\Delta_{X}\times\Delta_{Y}$, which boil down to the stability of regular epimorphisms under finite products in the regular category $\ca{C}$.
Notice that the bicategory $\nc{Rel}(\ca{C})$ does not have products without additional assumptions on $\ca{C}$.
\end{ex}

In \cite[\S~3.3]{VCocats}, a notion of a `locally closed' monoidal double category was introduced, that renders both ordinary categories $\dc{D}_0$ 
and $\dc{D}_1$ monoidal closed. We here revisit this notion, moreover dropping the term `locally', by first taking a closer look to general adjoints for double categories, see 
\cite[\S3.2]{Adjointfordoublecats}.

\begin{defi}\label{def:oplaxlaxajd}
An \emph{oplax/lax} adjunction $F\dashv R$ between double categories consists of
\begin{itemize}
 \item an oplax double functor $(F,\phi,\phi_0)\colon\dc{D}\to\dc{E}$
 \item a lax double functor $(R,\rho,\rho_0)\colon\dc{E}\to\dc{D}$
 \item two ordinary adjunctions $F_0\dashv R_0$ and $F_1\dashv R_1$ that respect the source and target -- equivalently $(\Gr{s},\Gr{s})$ and 
$(\Gr{t},\Gr{t})$ are maps of adjunctions
\item the counits are compatible with the structure maps\footnote{ As discussed in \cite{Adjointfordoublecats}, it cannot be asked that 
the unit and counit have the structure of a transformation, because the composites $FR$ and $RF$ are 
neither lax nor oplax double functors.}, namely
\begin{equation}\label{eq:rhobyphi}
\begin{tikzcd}[column sep=.5in]
FRY\ar[d,equal]\ar[r,bul,"F(RN'\odot RN)"]\ar[dr,phantom,"\Two F\rho_{N,N'}"] & FRY''\ar[d,equal] \\
FRY\ar[r,bul,"FR(N'\odot N)"]\ar[d,"\varepsilon_Y"']\ar[dr,phantom,"\Two\varepsilon_{N'\odot N}"] & FRY''\ar[d,"\varepsilon_{Y''}"] \\
Y\ar[r,bul,"N'\odot N"'] & Y''
\end{tikzcd}{=}
\begin{tikzcd}
FRY\ar[d,equal]\ar[rr,bul,"F(RN'\odot RN)"]\ar[drr,phantom,"\Two\phi_{RN,RN'}"] && FRY''\ar[d,equal] \\
FRY\ar[d,"\varepsilon_Y"']\ar[dr,phantom,"\Two\varepsilon_N"]\ar[r,bul,"FRN"] & 
FRY'\ar[d,"\varepsilon_{Y'}"']\ar[dr,phantom,"\Two\varepsilon_{N'}"]\ar[r,bul,"FRN'"] & FRY''\ar[d,"\varepsilon_{Y''}"] \\
Y\ar[r,bul,"N"'] & Y'\ar[r,bul,"N'"'] & Y''
\end{tikzcd}\;\;
\begin{tikzcd}
FRY\ar[d,equal]\ar[r,bul,"F(1_{RY})"]\ar[dr,phantom,"\Two F\rho_0"] & FRY\ar[d,equal] \\
FRY\ar[r,bul,"FR(1_Y)"]\ar[d,"\varepsilon_Y"']\ar[dr,phantom,"\Two\varepsilon_{1_Y}"] & FRY\ar[d,"\varepsilon_Y"] \\
Y\ar[r,bul,"1_Y"'] & Y
\end{tikzcd}{=}
\begin{tikzcd}
FRY\ar[r,bul,"F(1_{RY})"]\ar[d,equal]\ar[dr,phantom,"\Two\phi_0"] & FRY\ar[d,equal] \\
FRY\ar[r,bul,"1_{FRY}"]\ar[d,"\varepsilon_Y"']\ar[dr,phantom,"\Two1_{\varepsilon_Y}"] & FRY\ar[d,"\varepsilon_Y"] \\
Y\ar[r,bul,"1_Y"'] & Y
\end{tikzcd}
\end{equation}
\end{itemize}
\end{defi}

\begin{rmk}\label{rem:redundant}
The original definition included the extra axiom
\begin{equation}\label{eq:phibyrho}
\begin{tikzcd}[column sep=.5in]
X\ar[r,bul,"M'\odot M"]\ar[d,"\eta_X"']\ar[dr,phantom,"\Two\eta_{M'\odot M}"] & X''\ar[d,"\eta_{X"}"] \\
RFX\ar[d,equal]\ar[r,bul,"RF(M'\odot M)"]\ar[dr,phantom,"\Two R\phi_{M,M'}"] & RFX''\ar[d,equal] \\
RFX\ar[r,bul,"R(FM'\odot FM)"'] & RFX''
\end{tikzcd}{=}
\begin{tikzcd}
X\ar[r,bul,"M"]\ar[d,"\eta_X"']\ar[dr,phantom,"\Two\eta_M"] & X'\ar[r,bul,"M'"]\ar[d,"\eta_{X'}"']\ar[dr,phantom,"\Two\eta_{M'}"] & 
X''\ar[d,"\eta_{X''}"] \\
RFX\ar[d,equal]\ar[r,bul,"RFM"]\ar[drr,phantom,"\Two\rho_{FM,FM'}"] & RFX'\ar[r,bul,"RFM'"] & RFX''\ar[d,equal] \\
RFX\ar[rr,bul,"R(FM'\odot FM)"'] && RFX''
\end{tikzcd}\;\;
\begin{tikzcd}
X\ar[r,bul,"1_X"]\ar[d,"\eta_X"']\ar[dr,phantom,"\Two\eta_{1_X}"] & X\ar[d,"\eta_X"] \\
RFX\ar[d,equal]\ar[r,bul,"RF(1_X)"]\ar[dr,phantom,"\Two R\phi_0"] & RFX\ar[d,equal] \\
RFX\ar[r,bul,"R(1_{FX})"'] & RFX
\end{tikzcd}{=}
\begin{tikzcd}
X\ar[r,bul,"1_X"]\ar[d,"\eta_X"']\ar[dr,phantom,"\Two 1_{\eta_X}"] & X\ar[d,"\eta_X"] \\
RFX\ar[d,equal]\ar[r,bul,"1_{RFX}"]\ar[dr,phantom,"\Two\rho_0"] & RFX\ar[d,equal] \\
RFX\ar[r,bul,"R(1_{FX})"'] & RFX
\end{tikzcd}
\end{equation}
However, as mentioned also in \cite[Remark 3.3a]{Adjointfordoublecats}, this and condition \cref{eq:rhobyphi} in fact specify the 
structure cells of the 
functors involved in terms of one another. For example, \cref{eq:rhobyphi} express $(\rho,\rho_0)$ of $G$ in terms of 
$(\phi,\phi_0)$ of $F$ by providing their transposes under the adjunction $F_1\dashv G_1$.
One can verify, using naturality of the 
structure maps and (co)units, that just one of \cref{eq:rhobyphi} or \cref{eq:phibyrho} is enough to imply the other, therefore can be safely omitted 
from the definition.
\end{rmk}

This situation is reminiscent of the well-known \emph{doctrinal adjunction}
for adjoints between monoidal categories, where oplax monoidal structures on 
the left adjoint bijectively correspond to lax monoidal structures on the right adjoint. Following the above discussion, the theorem below gives 
simple conditions under which an oplax double functor has a lax double right adjoint.

\begin{thm}\label{thm:doubleadjoint}
 Suppose $F\colon\dc{D}\to\dc{E}$ is an oplax double functor. If
 \begin{itemize}
  \item $F_0$ and $F_1$ have ordinary right adjoints $R_0$ and $R_1$ respectively; and
  \item the source and target functors form maps of adjunctions depicted in
  \begin{displaymath}
\begin{tikzcd}[column sep=.6in, row sep=.4in]
\dc{D}_1\ar[r,shift left=2,"F_1"]\ar[d,"\mathfrak s/\mathfrak t"']\ar[r,phantom,"\bot"] & \dc{E}_1\ar[d,"\mathfrak s/\mathfrak t"]\ar[l,shift 
left=2,"R_1"] \\
\dc{D}_0\ar[r,shift left=2,"F_0"]\ar[r,phantom,"\bot"] & \dc{E}_0,\ar[l,shift left=2,"R_0"]
\end{tikzcd}   
  \end{displaymath}
 \end{itemize}
then $(R_0,R_1)$ forms a lax double functor, which is part of an oplax/lax adjunction $F\dashv R$.
\end{thm}

\begin{proof}
We define the required lax structure 2-cells $\rho_{M,N}\colon RN\odot RM\to R(N\odot M)$ and $\rho_0\colon 1_{RY}\to R(1_Y)$ for the pair $(R_0,R_1)$ by giving their adjuncts precisely as in \cref{eq:rhobyphi}. For this data to form a lax double functor structure as in \cref{defi:doublefunctor}, we already know that $R_0, 
R_1$ strictly commute with the source and 
target functors from the second clause above, and also naturality of $\rho,\rho_0$ follows from naturality of $\phi,\phi_0$ and $\varepsilon$. 
Finally, the coherence axioms can be verified, performing the relevant calculations `under the adjunction'. As a sample check, we verify that 
for composable horizontal 1-cells $M,N,U$, we have an equality of globular 2-cells $\rho_{N\circ M,U}(\mathrm{id}_{RU}\circ\rho_{M,N})=\rho_{M,U\circ 
N}(\rho_{N,U}\circ\mathrm{id}_{RM})$.
Translating under the adjunction $F_1\dashv G_1$, suppressing associativity for horizontal composition, we have the following steps
\begin{displaymath}
\scalebox{0.9}{
 \begin{tikzcd}[column sep=.8in,ampersand replacement=\&]
 FRX\ar[d,equal]\ar[r,bul,"{F(RU\odot RN\odot RM)}"]\ar[dr,phantom,"\Two F(\id\odot\rho)"] \& FRW\ar[d,equal] \\
 FRX\ar[r,bul,"{F(RU\odot R(N\odot M))}"]\ar[d,equal]\ar[dr,phantom,"\Two F\rho"] \& FRW\ar[d,equal] \\
 FRX\ar[d,"\varepsilon"']\ar[r,bul,"{FR(U\odot N\odot M)}"]\ar[dr,phantom,"\Two\varepsilon"] \& FRW\ar[d,"\varepsilon"] \\
 X\ar[r,bul,"U\odot N\odot M"'] \& W
 \end{tikzcd}}\stackrel{\cref{eq:rhobyphi}}{=}
\scalebox{0.9}{ \begin{tikzcd}[ampersand replacement=\&]
 FRX\ar[d,equal]\ar[rr,bul,"{F(RU\odot RN\odot RM)}"]\ar[drr,phantom,"\Two F(\id\odot\rho)"] \&\& FRW\ar[d,equal] \\
 FRX\ar[rr,bul,"{F(RU\odot R(N\odot M))}"]\ar[d,equal]\ar[drr,phantom,"\Two\phi"] \&\& FRW\ar[d,equal] \\
 FRX\ar[d,"\varepsilon"']\ar[r,bul,"{FR(N\odot M)}"]\ar[dr,phantom,"\Two\varepsilon"] \& 
FRZ\ar[r,bul,"FRU"]\ar[d,"\varepsilon"]\ar[dr,phantom,"\Two\varepsilon"] \& FRW\ar[d,"\varepsilon"] \\
 X\ar[r,bul,"N\odot M"'] \& Z\ar[r,bul,"U"'] \& W  
 \end{tikzcd}}\stackrel{(*)}{=}
\scalebox{0.9}{\begin{tikzcd}[column sep=.5in,ampersand replacement=\&]
 FRX\ar[d,equal]\ar[rr,bul,"{F(RU\odot RN\odot RM)}"]\ar[drr,phantom,"\Two \phi"] \&\& FRW\ar[d,equal] \\
 FRX\ar[r,bul,"{F(RN\odot RM)}"]\ar[d,equal]\ar[dr,phantom,"\Two F\rho"] \& FRZ\ar[r,bul,"FRU"]\ar[d,equal] \& FRW\ar[d,equal] \\
FRX\ar[d,"\varepsilon"']\ar[r,bul,"{FR(N\odot M)}"]\ar[dr,phantom,"\Two\varepsilon"] \& 
FRZ\ar[r,bul,"FRU"]\ar[d,"\varepsilon"]\ar[dr,phantom,"\Two\varepsilon"] \& FRW\ar[d,"\varepsilon"] \\
 X\ar[r,bul,"N\odot M"'] \& Z\ar[r,bul,"U"'] \& W  
 \end{tikzcd}}
\end{displaymath}
\begin{displaymath}
\stackrel{\cref{eq:rhobyphi}}{=}
\scalebox{0.9}{\begin{tikzcd}[ampersand replacement=\&]
 FRX\ar[d,equal]\ar[rrr,bul,"{F(RU\odot RN\odot RM)}"]\ar[drrr,phantom,"\Two \phi"] \&\&\& FRW\ar[d,equal] \\
 FRX\ar[rr,bul,"{F(RN\odot RM)}"]\ar[d,equal]\ar[drr,phantom,"\Two \phi"] \&\& FRZ\ar[r,bul,"FRU"]\ar[d,equal] \& FRW\ar[d,equal] \\
FRX\ar[d,"\varepsilon"']\ar[r,bul,"{FRM}"]\ar[dr,phantom,"\Two\varepsilon"] \& 
FRY\ar[d,"\varepsilon"]\ar[dr,phantom,"\Two\varepsilon"]\ar[r,bul,"FRN"] \&
FRZ\ar[r,bul,"FRU"]\ar[d,"\varepsilon"]\ar[dr,phantom,"\Two\varepsilon"] \& FRW\ar[d,"\varepsilon"] \\
 X\ar[r,bul,"M"'] \& Y\ar[r,bul,"N"'] \& Z\ar[r,bul,"U"'] \& W  
 \end{tikzcd}}\stackrel{(**)}{=}
\scalebox{0.9}{\begin{tikzcd}[ampersand replacement=\&]
 FRX\ar[d,equal]\ar[rrr,bul,"{F(RU\odot RN\odot RM)}"]\ar[drrr,phantom,"\Two \phi"] \&\&\& FRW\ar[d,equal] \\
 FRX\ar[r,bul,"{FRM}"]\ar[d,equal] \& FRY\ar[drr,phantom,"\Two \phi"]\ar[d,equal]\ar[rr,bul,"{F(RU\odot RN)}"] \&\& FRW\ar[d,equal] \\
FRX\ar[d,"\varepsilon"']\ar[r,bul,"{FRM}"]\ar[dr,phantom,"\Two\varepsilon"] \& 
FRY\ar[d,"\varepsilon"]\ar[dr,phantom,"\Two\varepsilon"]\ar[r,bul,"FRN"] \&
FRZ\ar[r,bul,"FRU"]\ar[d,"\varepsilon"]\ar[dr,phantom,"\Two\varepsilon"] \& FRW\ar[d,"\varepsilon"] \\
 X\ar[r,bul,"M"'] \& Y\ar[r,bul,"N"'] \& Z\ar[r,bul,"U"'] \& W  
 \end{tikzcd}}
\end{displaymath}
which produces the adjunct of the desired composite, 
where $(*)$ follows by naturality of $\phi$ and $(**)$ by the respective coherence axiom for $\phi$. The other two unit axioms can be proved 
analogously.

To conclude the proof, notice that all clauses of \cref{def:oplaxlaxajd} hold by default.
\end{proof}

Let us remark that the above result could be viewed as a special case of the more general \cite[Thm.~4.8]{Wobbly}, where the right 
adjoints 
$G_0$ and $G_1$ are required to commute with the source and target only up to isomorphism via the corresponding mates, hence in particular cannot 
form a double functor right away.
However, a direct calculation as in the above proof (which does not include the `wobbly functors' of \emph{loc. cit.}) is clearer for our setting. 

Moving on towards our goal, namely specifying a certain type of double adjoint for the tensor product of a monoidal double category, we now discuss 
\emph{parameterized} double adjunctions. Recall the original concept of an ordinary `adjunction with a parameter' found e.g. in \cite[\S IV.7]{MacLane}.

\begin{thm}\label{thm:parameterizeddouble}
 Suppose that $F\colon\dc{A}\times\dc{B}\to\dc{C}$ is an oplax double functor. If
 \begin{enumerate}[(i),ref=\thedefi(\roman*)]
\item $F_0$ and $F_1$ have ordinary parameterized adjoints $R_0$ and $R_1$ respectively; and
\item the source and target functors are maps of adjunctions as follows 
\begin{displaymath}
\begin{tikzcd}[column sep=.6in, row sep=.4in]
\dc{A}_1\ar[r,shift left=2,"{F_1(\mi,B)}"]\ar[d,"\mathfrak s"']\ar[r,phantom,"\bot"] & \dc{C}_1\ar[d,"\mathfrak s"]\ar[l,shift 
left=2,"{R_1(B,\mi)}"] \\
\dc{A}_0\ar[r,shift left=2,"{F_0(\mi,{\mathfrak s}B)}"]\ar[r,phantom,"\bot"] & \dc{C}_0\ar[l,shift left=2,"{R_0({\mathfrak s}B,\mi)}"]
\end{tikzcd}\qquad
\begin{tikzcd}[column sep=.6in, row sep=.4in]
\dc{A}_1\ar[r,shift left=2,"{F_1(\mi,B)}"]\ar[d,"\mathfrak t"']\ar[r,phantom,"\bot"] & \dc{C}_1\ar[d,"\mathfrak t"]\ar[l,shift 
left=2,"{R_1(B,\mi)}"] \\
\dc{A}_0\ar[r,shift left=2,"{F_0(\mi,{\mathfrak t}B)}"]\ar[r,phantom,"\bot"] & \dc{C}_0,\ar[l,shift left=2,"{R_0({\mathfrak t}B,\mi)}"]
\end{tikzcd}
\end{displaymath}
 \end{enumerate}
then $(R_0,R_1)\colon\dc{B}^\op\times\dc{C}\to\dc{A}$ forms a lax double functor, called the \emph{parameterized double adjoint} of $F$. 
\end{thm}

\begin{proof}
By definition of a double functor, $(F_0,F_1)$ is a morphism of two variables in $\mathsf{Cat}^\mathbf{2}$
\begin{displaymath}
\begin{tikzcd}[column sep=.6in, row sep=.4in]
\dc{A}_1\times\dc{B}_1\ar[r,"F_1"]\ar[d,"\mathfrak s\times\mathfrak s"'] & \dc{C}_1\ar[d,"\mathfrak s"] \\
\dc{A}_0\times\dc{B}_0\ar[r,"F_0"'] & \dc{C}_0
\end{tikzcd}
\end{displaymath}
and similarly a morphism from $\mathfrak t\times\mathfrak t$ to $\mathfrak t$. As also briefly recalled in \cref{sec:enrichedfibrations}, 
\cite[Thm.~3.4]{EnrichedFibration} ensures that under these assumptions, the pair of ordinary parameterized adjoints $(R_0,R_1)$ induced already 
by $(i)$ is such that the following square commutes
\begin{displaymath}
 \begin{tikzcd}[column sep=.6in, row sep=.4in]
\dc{B}^\op_1\times\dc{C}_1\ar[r,"R_1"]\ar[d,"\mathfrak s^\op\times\mathfrak s"'] & \dc{A}_1\ar[d,"\mathfrak s"] \\
\dc{B}^\op_0\times\dc{C}_0\ar[r,"R_0"'] & \dc{A}_0
 \end{tikzcd}
\end{displaymath}
and similarly for $\mathfrak t$. Finally, the oplax double structure maps $(\rho,\rho_0)$ for $F$ uniquely induce structure maps $(\phi,\phi_0)$ that make 
$(R_0,R_1)$ into a lax double functor in 
a completely analogous way to the proof of \cref{thm:doubleadjoint}, but now in the parameterized setting. In more detail, the adjuncts of the required maps, for horizontal 
1-cells 
$M\colon X\bular Y$, $N\colon Y\bular Z$, $M'\colon X'\bular Y'$ and $N'\colon Y'\bular Z'$, are given by
\begin{displaymath}
\scalebox{0.75}{
\begin{tikzcd}[column sep=1.4in,ampersand replacement=\&]
F\big(R(X,X'),X\big)\ar[r,bul,"{F\big(R(N,N')\odot R(M,M'),N\odot M\big)}"]\ar[d,equal]\ar[dr,phantom,"\Two {F(\rho,\mathrm{id})}"] \& 
F\big(R(Z,Z'),Z\big)\ar[d,equal] \\
F\big(R(X,X'),X\big)\ar[d,"\varepsilon"']\ar[r,bul,"{F\big(R(N\odot M,N'\odot M'),N\odot M\big)}"]\ar[dr,phantom,pos=.4,"\Two\varepsilon"] \& 
F\big(R(Z,Z'),Z\big)\ar[d,"\varepsilon"] \\
X'\ar[r,bul,"{M'\odot N'}"'] \& Z'
\end{tikzcd}=
\begin{tikzcd}[column sep=.8in,ampersand replacement=\&]
F\big(R(X,X'),X\big)\ar[rr,bul,"{F\big(R(N,N')\odot R(M,M'),N\odot M\big)}"]\ar[d,equal]\ar[drr,phantom,"\Two \phi"] \&\& 
F\big(R(Z,Z'),Z\big)\ar[d,equal] \\
F\big(R(X,X'),X\big)\ar[d,"\varepsilon"']\ar[r,bul,"{F\big(R(M,M'),M\big)}"]\ar[dr,phantom,pos=.4,"\Two\varepsilon"] \& 
F\big(R(Y,Y'),Y\big)\ar[r,bul,"{F\big(R(N,N'),N\big)}"]\ar[d,"\varepsilon"]\ar[dr,phantom,pos=.4,"\Two\varepsilon"] \& 
F\big(R(Z,Z'),Z\big)\ar[d,"\varepsilon"] \\
X'\ar[r,bul,"M'"'] \& Y'\ar[r,bul,"N'"'] \& Z' 
\end{tikzcd}}
\end{displaymath}
\begin{equation}\label{eq:parameterizedcounitcompat}
\scalebox{.75}{
\begin{tikzcd}[column sep=.7in,ampersand replacement=\&]
F\big(R(X,X'),X\big)\ar[r,bul,"{F\big(1_{R(X,X')},1_X\big)}"]\ar[dr,phantom,"\Two {F(\rho_0,\mathrm{id})}"]\ar[d,equal] \& 
F\big(R(X,X'),X\big)\ar[d,equal] \\
F\big(R(X,X'),X\big)\ar[d,"\varepsilon"']\ar[r,bul,"{F\big(R(1_{(X,Y)}),1_X\big)}"]\ar[dr,phantom,pos=.4,"\Two\varepsilon"] \& 
F\big(R(X,X'),X\big)\ar[d,"\varepsilon"] \\
X'\ar[r,bul,"{1_{X'}}"'] \& X'
\end{tikzcd}=
\begin{tikzcd}[column sep=.7in,ampersand replacement=\&]
F\big(R(X,X'),X\big)\ar[r,bul,"{F\big(1_{R(X,X')},1_X\big)}"]\ar[dr,phantom,"\Two \phi_0"]\ar[d,equal] \& F\big(R(X,X'),X\big)\ar[d,equal] \\
F\big(R(X,X'),X\big)\ar[d,"\varepsilon"']\ar[r,bul,"{1_{F\big(R(X,X'),X\big)}}"]\ar[dr,phantom,pos=.4,"\Two1_\varepsilon"] \& 
F\big(R(X,X'),X\big)\ar[d,"\varepsilon"] \\
X'\ar[r,bul,"{1_{X'}}"'] \& X'
\end{tikzcd}}
\end{equation}
\end{proof}

As pointed out above, the induced structure maps $(\phi,\phi_0)$ on $R$ are the parameterized versions of \cref{eq:rhobyphi}, 
and the proof that they form a lax structure follows the parameterized analogue of the doctrinal adjunction for monoidal categories: 
oplax monoidal structures on functors of two variables on the left, bijectively correspond to lax monoidal structures on the right parameterized adjoint of two variables -- see 
\cite[Prop.~3.2.3]{PhDChristina} for an explicit proof or \cite[Prop.~2]{Monoidalbicatshopfalgebroids} for a higher dimensional version.

Finally, notice how in the ordinary setting, a parameterized adjunction is essentially defined via the following formulation, whereas in the double categorical setting it is strictly weaker than the presence of a fully fledged double adjoint.

\begin{cor}\label{cor:impliesobjectwise}
For a parameterized oplax/lax double adjunction $F\dashv G$ between functors of two variables $F\colon\dc{A}\times\dc{B}\to\dc{C}$ and $G\colon\dc{B}^\op\times\dc{C}\to\dc{A}$, in particular we obtain oplax/lax double adjunctions $F(\mi,X)\dashv G(X,\mi)$ for any $X\in\dc{B}$. 
\end{cor}

\begin{proof}
Checking all clauses of \cref{def:oplaxlaxajd} for the oplax and lax, respectively, double functors $(F_0(\mi,X),F_1(\mi,1_X))\colon\dc{A}\to\dc{C}$ and $(G_0(X,\mi),G_1(1_X,\mi))\colon\dc{C}\to\dc{A}$, 
by default we have two ordinary adjunctions $F_0(\mi,X)\dashv G_0(X,\mi)$ and $F_1(\mi,1_X)\dashv G_1(1_X,\mi)$ that respect the source and target. For the second family of adjunctions, notice that $F_1\dashv G_1$ being parameterized adjoints is strictly more general, since the fixed-variable adjunctions exist for arbitrary horizontal 1-cells and not just those of the form $1_X$. Finally, it can be shown that the counits are compatibe with the structure maps, using the respective parameterized compatibilities \cref{eq:parameterizedcounitcompat} and dinaturality of counits.
\end{proof}

We have now laid all necessary groundwork in order to position the notion of a monoidal closed double category in its natural context.

\begin{defi}\label{def:locclosed}
A monoidal double category $(\dc{D},\otimes,I)$ is \emph{closed} if the pseudo 
double functor $\otimes\colon\dc{D}\times\dc{D}\to\dc{D}$ has a parameterized double adjoint, namely a lax double functor
\begin{equation}\label{eq:internalhomdouble}
H=(H_0,H_1)\colon\dc{D}^\op\times\dc{D}\longrightarrow\dc{D}
\end{equation}
such that:
\begin{enumerate}[(i)]
\item $\otimes_0\dashv H_0$ and $\otimes_1\dashv H_1$ are parameterized adjunctions, so $\dc{D}_0$ and $\dc{D}_1$ are monoidal closed 
categories;
\item \label{item-ii}for every horizontal 1-cell $M\colon X\bular Y$ in $\dc{D}$, the pairs of functors $(\Gr{s},\Gr{s})$ and $(\Gr{t},\Gr{t})$ as 
below are maps 
of adjunctions:
\begin{displaymath}
\begin{tikzcd}[column sep=.6in, row sep=.4in]
\dc{D}_1\ar[r,shift left=2,"{\mi\ot M}"]\ar[d,"\mathfrak s"']\ar[r,phantom,"\bot"] & \dc{D}_1\ar[d,"\mathfrak s"]\ar[l,shift left=2,"{H(M,\mi)}"] \\
\dc{D}_0\ar[r,shift left=2,"{\mi\ot X}"]\ar[r,phantom,"\bot"] & \dc{D}_0\ar[l,shift left=2,"{H(X,\mi)}"]
\end{tikzcd}\qquad\qquad
\begin{tikzcd}[column sep=.6in, row sep=.4in]
\dc{D}_1\ar[r,shift left=2,"{\mi\ot M}"]\ar[d,"\mathfrak t"']\ar[r,phantom,"\bot"] & \dc{D}_1\ar[d,"\mathfrak t"]\ar[l,shift left=2,"{H(M,\mi)}"] \\
\dc{D}_0\ar[r,shift left=2,"{\mi\ot Y}"]\ar[r,phantom,"\bot"] & \dc{D}_0.\ar[l,shift left=2,"{H(Y,\mi)}"]
\end{tikzcd}
\end{displaymath}
\end{enumerate}
\end{defi}
The above explicit description follows from \cref{thm:parameterizeddouble}. Condition \ref{item-ii} boils down 
to the fact that the counit of each adjunction $\mi\otimes M\dashv H(M,\mi)$ at every $N\colon Z\bular W$ has components of the following form 
\begin{equation}\label{eq:coev}
\begin{tikzcd}[column sep=5em, row sep=3.5em]
	H(X,Z)\otimes X\ar[d,"\mathrm{ev}_{X,Z}"']\ar[r,bul,"{H(M,N)\otimes M}"]\ar[dr,phantom,"\Two\mathrm{ev}_{M,N}"] & 
H(Y,W)\otimes Y\ar[d,"\mathrm{ev}_{Y,W}"] \\
	Z\ar[r,bul,"{N}"'] & W
\end{tikzcd}
\end{equation}
namely its source/target vertical morphisms are the corresponding components of the counit of the adjunction in $\dc{D}_0$. 
Equivalently (\cite[\S~IV.7]{MacLane}), the source/target of an adjunct 2-cell in $\dc{D}_1$ is the adjunct of the source/target vertical map in 
$\dc{D}_0$.
Also notice that the definition implies that $\Gr{s},\Gr{t}$ are closed functors. 

\begin{rmk}\label{rmk:comparewithBob}
In \cite[\S5.5]{Adjointfordoublecats}, Grandis and Par{\'e} briefly define a \emph{weakly monoidal closed} double category, as one where each oplax double functor $\mi\ot X$ has a lax right double adjoint. In light of \cref{cor:impliesobjectwise}, such a formulation is strictly less general than \cref{def:locclosed} and any monoidal closed double category is in particular a weakly such. 

In fact, following terminology introduced by Niefield in the setting of \emph{cartesian closed} double categories \cite{NiefieldCCl}, one could call such a weakly monoidal closed double category an \emph{object-wise} monoidal closed one. That latter work is certainly relevant to the current material and although independently developed, the observation that a parameterized adjunction can equivalently be defined via object-wise adjunctions in the ordinary setting whereas that equivalence breaks in the double categorical setting, is key.
\end{rmk}

\begin{rmk}\label{rmk:clarifying}
The original definition of a locally closed monoidal double category in \cite[\S~3.3]{VCocats} did not explicitly include condition \ref{item-ii}
-- as well as the parameterized double adjunction framework that backs it up. 
However, the said condition is in fact necessary for certain constructions, e.g. the induced monoidal closed structure on the endomorphism category 
$\dc{D}_{1}^{\bullet}$:
as in \cite[Prop.~3.21]{VCocats}, we restrict the natural isomorphisms $\dc{D}_{1}(M\otimes_{1}N,P)\cong\dc{D}_{1}(M,H_{1}(N,P))$ to endo-1-cells and 
endo-2-morphisms. In 
particular, the latter statement implies that the conjugate of an endo-2-morphism along the adjunction $\otimes_{1}\dashv H_{1}$ is again 
an \emph{endo}-2-morphism.
\end{rmk}

\begin{ex}\label{ex:VMMatclosed}
The double category $\VMMat$ from \cref{ex:VMMatmonoidal} is monoidal closed, when $\ca{V}$ is furthermore monoidal closed with products. 
The relevant structure was detailed in \cite[Cor.~4.4]{VCocats} but in the current setting requires less checks, since for example the lax double structure of $H$ 
follows from \cref{thm:parameterizeddouble}. 

To recall the structure, the clauses of \cref{def:locclosed} are satisfied as follows.
On the level of objects, the internal hom is the ordinary exponentiation functor, whereas for the 
category of arrows we have
\begin{displaymath}
	H_1:\xymatrix @C=1.3in
	{\VMMat_1^\op\times\VMMat_1
		\ar[r] & \VMMat_1}\phantom{ABC}
\end{displaymath}\vspace{-0.2in}
\begin{displaymath}
	\xymatrix @C=.05in @R=.25in
	{(X\ar[rrr]|-{\object@{|}}^S\ar[d]_-f
		&\rtwocell<\omit>{<4>{\alpha}}&& Y\ar[d]^-g
		& , & Z\ar[rrr]^-T|-{\object@{|}}\ar[d]_-h
		&\rtwocell<\omit>{<4>\beta}&& W)\ar[d]^-k
		\ar@{|.>}[rrr] &&& Z^X\ar[rrrr]^-{H(S,T)}|-{\object@{|}}
		\ar[d]_-{h^f} &\rtwocell<\omit>{<4>\qquad H(\alpha,\beta)}
		&&& W^Y\ar[d]^-{k^g} \\
		(X'\ar[rrr]_-{S'}|-{\object@{|}} &&& Y' & , &
		Z'\ar[rrr]_-{T'}|-{\object@{|}} &&& W')
		\ar@{|.>}[rrr] &&& Z'^{X'}\ar[rrrr]_-{H(S',T')}|-{\object@{|}} 
			&&&& {W'}^{Y'}}
\end{displaymath}
which on horizontal 1-cells is given by families of objects in $\ca{V}$, for $n\in Z^X, m\in W^Y$, 
\begin{displaymath}
	H_1(S,T)(m,n)=\prod_{x,y}\left[S(y,x),T\left(m(y),n(x)\right)\right]. 
\end{displaymath}
It can be verified that $\mi\ot_1S\dashv H_1(S,\mi)$ and the source and target are respected by construction. 
Finally, the units and the counits are indeed `above' one another as in \cref{eq:coev}: for example the evaluation maps
\begin{displaymath}
\begin{tikzcd}[column sep=.6in]
Z^X\ot X\ar[r,tick,"{H(S,T)\ot T}"]\ar[d,"\ev"']\ar[dr,phantom,"\Two\ev"] & W^Y\ot Y\ar[d,"\ev"] \\
Z\ar[r,tick,"T"'] & W
\end{tikzcd}
\end{displaymath}
in this case are given by arrows $H(S,T)(m,n)\otimes S(y,x)\to 
T(\ev(m,y),\ev(n,x))$ in $\ca{V}$ for appropriate elements $m,n,y,x$, which correspond to
\begin{displaymath}
\left(\prod_{y',x'}{[S(y',x'),T(m(y),n(x))]}\right)\otimes S(y,x)\xrightarrow{\pi_{y,x}\ot1}[S(y,x),T(m(y),n(x))]\ot 
S(y,x)\xrightarrow{\mathrm{ev}}T(m(y),n(x)). 
\end{displaymath}
\end{ex}

\begin{ex}
 
The double category $\dc{R}\nc{el}(\ca{\nc{Set}})$ is monoidal closed. The vertical part of the required structure comes from the familiar cartesian closure of $\nc{Set}$. In addition, for any $f\in A^{X}$ and $g\in B^{Y}$ define $(f,g)\in S^{R}$ to mean that the 
implication $(x,y)\in R\rightarrow (f(x),g(y))\in S$ holds for all $x\in X$ and $y\in Y$, yielding a lax double 
functor $\dc{R}\nc{el}(\ca{\nc{Set}})^{\op}\times\dc{R}\nc{el}(\ca{\nc{Set}})\to\dc{R}\nc{el}(\ca{\nc{Set}})$. For the clauses of \cref{def:locclosed}, we verify that there is a bijection between 
2-cells of the following forms
\begin{displaymath}
	\begin{tikzcd}
		C\times X\ar[r,bul,"T\times R"]\ar[d,"f"']\ar[dr,phantom,"\Two"] & D\times Y\ar[d,"g"] \\
		A\ar[r,bul,"S"'] & B	
	\end{tikzcd}
	\qquad
	\begin{tikzcd}
		C\ar[r,bul,"T"]\ar[d,"\tilde{f}"']\ar[dr,phantom,"\Two"] & D\ar[d,"\tilde{g}"] \\
		A^X\ar[r,bul,"S^R"'] & B^Y	
	\end{tikzcd}
\end{displaymath}
i.e. the inclusion of relations $g(T\times R)f^{\circ}\subseteq S$ is equivalent to $\tilde{g}T\tilde{f}^{\circ}\subseteq S^R$.
	
Assume the first of these holds and consider any $c\in C$ and $d\in D$ such that $(c,d)\in T$. We need to show that $(\tilde{f}(c),\tilde{g}(d))\in 
S^R$. So let $x\in X$ and $y\in Y$ be such that $(x,y)\in R$. Then $((c,x),(d,y))\in T\times R$ and hence by assumption we have $(f(c,x),g(d,y))\in 
S$, as desired.
	
Conversely, assume that $\tilde{g}T\tilde{f}^{\circ}\subseteq S^R$ and let $(c,x)\in C\times X$ and $(d,y)\in D\times Y$ be such that 
$((c,x),(d,y))\in T\times R$. Then we have $(x,y)\in R$ and $(c,d)\in T$. Since $(\tilde{f}(c),\tilde{g}(d))\in S^R$ by our assumption, we deduce 
that $(\tilde{f}(c)(x),\tilde{g}(d)(y))\in S$, i.e. that $(f(c,x),g(d,y))\in S$.

Finally, it is clear from the description of the bijection given above that the counits of the adjunctions $-\times R\dashv (-)^{R}$ in $\dc{R}\nc{el}(\ca{\nc{Set}})_1$ have vertical morphisms the corresponding counits in $\nc{Set}$. In other words they are of the form
\begin{displaymath}
	\begin{tikzcd}
		A^X\times X\ar[r,bul,"S^{R}\times R"]\ar[d,"\mathrm{ev}_{X,A}"']\ar[dr,phantom,"\Two\mathrm{ev}_{R,S}"] & B^Y\times Y\ar[d,"\mathrm{ev}_{Y,B}"] \\
		A\ar[r,bul,"S"'] & B	
	\end{tikzcd}
\end{displaymath}
which gives the second clause of definition \ref{def:locclosed}.

\end{ex}

\subsection{Monads and comonads}\label{sec:monadscomonadsdouble}

Similarly to monads and comonads in bicategories recalled in \cref{sec:monmodbicats}, we can define such structures in double categories. In fact, 
as objects 
they coincide with those of the corresponding horizontal bicategory, however the morphisms are more flexible.

\begin{defi}\label{Monadindoublecat}
	A \emph{monad} in a double category $\dc{D}$ is a horizontal endo-1-cell $A\colon X\bular X$ equipped with
	\begin{displaymath}
		\begin{tikzcd}[sep=.3in]
			X\ar[r,bul,"A"]\ar[d,equal]\ar[drr,phantom,"\Two\mu"] & X\ar[r,bul,"A"] & X\ar[d,equal] \\
			X\ar[rr,bul,"A"'] && X
		\end{tikzcd}\qquad
		\globtwocell{X}{1_X}{X}{X}{A}{X}{\eta}
	\end{displaymath}
	satisfying the usual associativity and unit laws. A \emph{monad morphism} is a 2-morphism ${}^f\alpha^f\colon A\Rightarrow B$ where
	\begin{equation}\label{monadhom}
		\begin{tikzcd}[sep=.3in]
			X\ar[r,bul,"A"]\ar[d,"f"']\ar[dr,phantom,"\Two\alpha"] & X\ar[r,bul,"A"]\ar[d,"f"]\ar[dr,phantom,"\Two\alpha"] & X\ar[d,"f"] \\
			Y\ar[r,bul,"B"'] \ar[d,equal]\ar[drr,phantom,"\Two\mu"] & Y\ar[r,bul,"B"'] & Y\ar[d,equal] \\
			Y\ar[rr,"B"'] && Y
		\end{tikzcd}=
		\begin{tikzcd}[sep=.3in]
			X\ar[r,bul,"A"]\ar[d,equal]\ar[drr,phantom,"\Two\mu"] & X\ar[r,bul,"A"] & X\ar[d,equal] \\
			X\ar[rr,bul,"A"'] \ar[d,"f"']\ar[drr,phantom,"\Two\alpha"] && X\ar[d,"f"] \\
			Y\ar[rr,"B"'] && Y,
		\end{tikzcd}\qquad\qquad
		\begin{tikzcd}[sep=.3in]
			X\ar[r,bul,"1_X"]\ar[d,equal]\ar[dr,phantom,"\Two\eta"] & X\ar[d,equal] \\
			X\ar[r,bul,"A"']\ar[d,"f"']\ar[dr,phantom,"\Two\alpha"] & X\ar[d,"f"] \\
			Y\ar[r,bul,"B"'] & Y
		\end{tikzcd}=
		\begin{tikzcd}[sep=.3in]
			X\ar[r,bul,"1_X"]\ar[d,"f"']\ar[dr,phantom,"\Two1_f"] & X\ar[d,"f"] \\
			Y\ar[r,bul,"1_Y"']\ar[d,equal]\ar[dr,phantom,"\Two\eta"] & Y\ar[d,equal] \\
			Y\ar[r,bul,"B"'] & Y 
		\end{tikzcd}
	\end{equation}
	Dually, a \emph{comonad} in $\dc{D}$ is an endo-1-cell $C\colon Z\bular Z$ equipped with globular 2-morphisms $\delta\colon C\Rightarrow C\odot 
C$ and $\epsilon\colon C\Rightarrow1_U$
satisfying the usual coassociativity and counit axioms.
A \emph{comonad morphism} is a 2-morphism ${}^f\alpha^f\colon C\Rightarrow D$ 
respecting the comultiplications and counits.
\end{defi}

We obtain a category $\Mnd(\dc{D})$, which is in fact the vertical category of a double category of monads introduced in 
\cite[Def.~2.4]{Monadsindoublecats}, as well as a category $\Cmd(\dc{D})$.

\begin{ex}\label{ex:monadsRel}
Monads in a double category $\dc{R}\nc{el}(\ca{C})$, for $\ca{C}$ a regular category, are relations $A\colon X\bular X$ 
which are transitive and reflexive. Indeed, the monad laws are trivial here and so we are left with the two inclusions $AA\subseteq A$ 
and 
$1_{X}=\Delta_{X}\subseteq A$. Similarly, a monad morphism $A\to B$ reduces to a morphism $f\colon X\to Y\in\ca{C}$ satisfying $f[A]\subseteq B$, 
where $f[A]$ denotes the image of the relation $A$ along the morphism $f$. Thus $\Mnd(\dc{R}\nc{el}(\ca{C}))$ is precisely the category of $\nc{Preord}(\ca{C})$ internal preorders and order-preserving morphisms in $\ca{C}$.
	
On the other hand, a comonad $C\colon Z\bular Z$ is a relation satisfying $C\subseteq\Delta_{U}$ and $C\subseteq CC$. These determine precisely a 
subobject $C\hookrightarrow U$ in $\ca{C}$ 
and then a comonad morphism $C\to D$ is a morphism 
$f\colon U\to V\in\ca{C}$ such that $f[C]\subseteq D$.
\end{ex}

\begin{ex}\label{ex:VCatsaremonadsinVMat}
As discussed in detail in \cite[\S~4.2,4.3]{VCocats}, monads and comonads in the double category $\VMMat$ of $\ca{V}$-matrices (\cref{ex:VMMat}) are precisely 
$\ca{V}$-categories and $\ca{V}$-cocategories\footnote{Notice the difference between $\ca{V}$-cocategories described here, and $\ca{V}$-\emph{opcategories} namely $\ca{V}^\op$-categories, with cocomposition
arrows landing on products. See \cite[Rem.~4.25]{VCocats} for further discussion.}. For 
example, a monad is a family $\{A(y,x)\}_{(y,x)\in X\times X}$ of objects in $\ca{V}$, with unit picking 
out an element $\eta_{x,x}\colon I\to A(x,x)$ for each $x\in X$ and multiplication $\mu_{x,y}\colon\sum_{y\in Y} A(z,y)\otimes A(y,x)\to A(z,x)$ 
equivalently given by a collection of arrows $A(z,y)\otimes A(y,x)\to S(z,x)$ for every $x,y,z\in X$, satisfying the usual axioms. Dually,
	a comonad $C\colon Z\tickar Z$ comes with appropriate cocomposition and coidentity arrows
	\begin{gather*}
		\delta_{x,z}\colon C(x,z)\to\sum_{y\in X}{C(x,y)\otimes C(y,z)}\\
		\epsilon_{x,y}\colon C(x,y)\to 1_X(x,y)\equiv
		\begin{cases}
			C(x,x)\xrightarrow{\epsilon_{x,x}}I,\quad \mathrm{if }\;x=y\\
			C(x,y)\xrightarrow{\epsilon_{x,y}}0,\quad \mathrm{if }\;x\neq y
		\end{cases}
	\end{gather*}
	satisfying coassociativity and counitality laws.
	Since (co)monad morphisms in $\VMMat$ are precisely $\ca{V}$-(co)functors, $\Mnd(\VMMat)\cong\VCat$ and 
$\Cmd(\VMMat)\cong\VCocat$. We note that this illustrates a conceptual advantage of double categories as opposed to bicategories, since we 
can equally well describe $\ca{V}$-categories as monads in the (horizontal) bicategory $\VMat$, however, it is then no longer 
true that monad morphisms bijectively correspond to $\ca{V}$-functors.
\end{ex}

In a similar spirit to \cref{globalvslocal}, the (op)fibration $\langle\Gr{s},\Gr{t}\rangle$ of 
\cref{eq:stbifibration} defining a fibrant double category induces similar structures for endocells and (co)monads, see 
\cite[Prop.~3.3]{Monadsindoublecats} or \cite[Prop.~3.15,3.17]{VCocats}.

\begin{prop}\label{prop:MonComonfibred}
Suppose $\dc{D}$ is a fibrant double category.
\begin{enumerate}[(i),ref=\thedefi(\roman*)]
 \item $\dc{D}_1^\bullet$ is bifibred over $\dc{D}_0$;
 \item $\Mnd(\dc{D})$ is fibred over 
$\dc{D}_0$ and $\Cmd(\dc{D})$ is opfibred over $\dc{D}_0$.
\end{enumerate}
\end{prop}

\begin{proof}[Sketch]
$(i)$ The fiber above an object $X\in\dc{D}_0$ is the endo-hom category $\ca{H}(\dc{D})(X,X)$, and the reindexing functors, for any 
vertical morphism $f\colon X\to Y$, are 
\begin{displaymath}
\begin{tikzcd}[column sep=.6in]
\ca{H}(\dc{D})(X,X)\ar[r,shift left=2,"\wh{f}\odot\mi\odot\wc{f}"]\ar[r,phantom,"\bot"] & \ca{H}(\dc{D})(Y,Y).\ar[l,shift 
left=2,"\wc{f}\odot\mi\odot\wh{f}"]  
\end{tikzcd}
\end{displaymath}

$(ii)$ These are of course the source/target functors down to $\dc{D}_0$, mapping (co)monads to their carrier objects. Recall that a (co)monad $X\bular X$ in $\dc{D}$, equivalently a (co)monad in the horizontal bicategory $\ca{H}(\dc{D})$, is the same as a 
	(co)monoid in the endo-hom category $(\ca{H}(\dc{D})(X,X),\circ,1_X)$. The pseudofunctors that give rise to the (op)fibrations are
	\begin{displaymath}
		\begin{tikzcd}[row sep=.05in,/tikz/column 1/.append style={anchor=base east},/tikz/column 2/.append style={anchor=base west},/tikz/column 
			5/.append style={anchor=base east},/tikz/column 6/.append style={anchor=base west}]
			\dc{D}_0^\op\ar[r] & \Cat &&& \dc{D}_0\ar[r] & \Cat \\
			X\ar[r,mapsto]\ar[dd,"f"'] & \Mon(\ca{H}(\dc{D})(X,X)) &&& X\ar[r,mapsto]\ar[dd,"f"'] & 
			\Comon(\ca{H}(\dc{D})(X,X)) \ar[dd,"\wh{f}\odot\mi\odot\wc{f}"]\\
			\hole \\
			Y\ar[r,mapsto] & \Mon(\ca{H}(\dc{D})(Y,Y))\ar[uu,"\wc{f}\odot\mi\odot\wh{f}"'] &&& Y\ar[r,mapsto] & \Comon(\ca{H}(\dc{D})(Y,Y))
		\end{tikzcd}
	\end{displaymath}
This is due to the fact that the reindexing functors for $\dc{D}_1^\bullet$ from $(i)$ are lax and oplax monoidal respectively,
which leads to their restriction between  the categories of monoids and comonoids therein.
\end{proof}

\begin{ex}
In the context of the double category $\dc{D}=\VMMat$, \cref{prop:MonComonfibred} gives us a bifibration $\VGrph\to\Set$ that maps a graph to its set of 
objects, and also a fibration $\VCat\to\nc{Set}$ and an opfibration $\VCocat\to\nc{Set}$. In more detail, given a function $f\colon X\to Y$ and a 
$\ca{V}$-category $B\colon Y\tickar 
Y$, the cartesian lifting of $f$ to $B$ is the morphism $^{f}1^{f}\colon f^{*}\circ B\circ f_{*}\Rightarrow B$, where $(f^{*}\circ B\circ 
f_{*})(x,x')=B(f(x),f(x'))$, for all $x,x'\in X$. On the other hand, for a function $f\colon U\to V$ and a $\ca{V}$-cocategory $C\colon Z\tickar Z$, 
the cocartesian lifting of $f$ to $C$ is a morphism of the form $C\Rightarrow f_{*}\circ C\circ f^{*}$ where $(f_{*}\circ C\circ 
f^{*})(v,v)=\sum_{f(z)=v,f(z')=v'}C(z,z')$. The required morphism then has components, for all $u,u'\in U$, 
the evident coproduct injections
$$C(u,u')\to\sum\limits_{\substack{f(a)=f(u)\\f(a')=f(u')}}C(a,a').$$
\end{ex}

Next, we record some information on the existence of (co)limits in categories of (co)monads, under suitable assumptions on the double 
category $\dc{D}$, see \cref{def:parallelcolimits}. First recall that by \cref{lem:(co)limits in D_1 bullet}, in the fibrant setting, $\dc{D}_1^\bullet$ 
inherits (co)limits from $\dc{D}_1$ in a parallel (co)complete double category $\dc{D}$.

\begin{prop}\label{prop:(co)limits in (co)monads}
If a fibrant double category $\dc{D}$ has parallel $\ca{I}$-limits, then the category $\Mnd(\dc{D})$ has all $\ca{I}$-limits and the inclusion 
$\Mnd(\dc{D})\to\dc{D}_{1}^{\bullet}$ creates them. Dually, if $\dc{D}$ has parallel $\ca{I}$-colimits, the $\Cmd(\dc{D})$ has $\ca{I}$-colimits 
and 
$\Cmd(\dc{D})\to\dc{D}_{1}^{\bullet}$ creates them.
\end{prop}

\begin{proof}
We prove the result about comonads, as the corresponding one for monads follows by duality. Consider a diagram $D\colon\ca{I}\to\Cmd(\dc{D})$ whose 
colimit 
cocone in $\dc{D}_{1}^{\bullet}$ is 
\begin{equation}\label{eq:phii}
	\begin{tikzcd}
		X_{i}\ar[r,bul,"C_{i}"]\ar[d,"f_{i}"']\ar[dr,phantom,"\Two\phi_{i}"] & X_{i}\ar[d,"f_{i}"] \\
		X\ar[r,bul,"C"'] & X
	\end{tikzcd}
\end{equation}
(which is obtained from the colimit in $\dc{D}_1$ using a normalization explained in \cref{eq:normalizing} due to fibrancy).
We will show that $C$ has an induced comonad structure which exhibits it as the colimit in $\Cmd(\dc{D})$. 
	
First, observe that for all $\alpha\colon i\to j\in\ca{I}$ we have the following equalities
\begin{displaymath}
	\begin{tikzcd}
		X_{i}\ar[rr,bul,"C_{i}"]\ar[d,"h_\alpha"']\ar[drr,phantom,"\Two D\alpha"] && X_{i}\ar[d,"h_\alpha"] \\
		X_{j}\ar[rr,bul,"C_{j}"]\ar[d,equal]\ar[drr,phantom,"\Two\delta_{j}"] && X_{j}\ar[d,equal] \\
		X_{j}\ar[r,bul,"C_{j}"]\ar[d,"f_{j}"']\ar[dr,phantom,"\Two\phi_{j}"] & X_{j}\ar[r,bul,"C_{j}"]\ar[d,"f_{j}"']\ar[dr,phantom,"\Two\phi_{j}"] & 
X_{j}\ar[d,"f_{j}"] \\
		X\ar[r,bul,"C"'] & X\ar[r,bul,"C"'] & X	
	\end{tikzcd}=
	\begin{tikzcd}
		X_{i}\ar[rr,bul,"C_{i}"]\ar[d,equal]\ar[drr,phantom,"\Two\delta_{i}"] && X_{i}\ar[d,equal] \\
		X_{i}\ar[r,bul,"C_{i}"]\ar[d,"h_\alpha"']\ar[dr,phantom,"\Two D\alpha"] & 
X_{i}\ar[r,bul,"C_{i}"]\ar[d,"h_\alpha"']\ar[dr,phantom,"\Two D\alpha"] & X_{i}\ar[d,"h_\alpha"] \\
		X_{j}\ar[r,bul,"C_{j}"]\ar[d,"f_{j}"']\ar[dr,phantom,"\Two\phi_{j}"] & X_{j}\ar[r,bul,"C_{j}"]\ar[d,"f_{j}"']\ar[dr,phantom,"\Two\phi_{j}"] & 
X_{j}\ar[d,"f_{j}"] \\
		X\ar[r,bul,"C"'] & X\ar[r,bul,"C"'] & X	
	\end{tikzcd}=
	\begin{tikzcd}
		X_{i}\ar[rr,bul,"C_{i}"]\ar[d,equal]\ar[drr,phantom,"\Two\delta_{i}"] && X_{i}\ar[d,equal] \\
		X_{i}\ar[r,bul,"C_{i}"]\ar[d,"f_{i}"']\ar[dr,phantom,"\Two\phi_{i}"] & X_{i}\ar[r,bul,"C_{i}"]\ar[d,"f_{i}"']\ar[dr,phantom,"\Two\phi_{i}"] & 
X_{i}\ar[d,"f_{i}"] \\
		X\ar[r,bul,"C"'] & X\ar[r,bul,"C"'] & X
	\end{tikzcd}
\end{displaymath}
due to the fact that $D\alpha$ is a morphism of comonads and $(\phi_i)_{i\in\ca{I}}$ is a cocone in $\dc{D}_1^\bullet$ respectively. Thus, 
$((\phi_{i}\odot\phi_{i})\delta_{i}\colon C_{i}\Rightarrow C\odot C)_{i\in\ca{I}}$ is a cocone over $D$ in $\dc{D}_{1}^{\bullet}$, hence there 
exists a unique morphism ${}^u\delta^u\colon C\Rightarrow C\odot C$ from the colimit $C$
such that for all $i\in\ca{I}$ we have the following equality
\begin{displaymath}
	\begin{tikzcd}
		X_{i}\ar[rr,bul,"C_{i}"]\ar[d,"f_{i}"']\ar[drr,phantom,"\Two\phi_{i}"] && X_{i}\ar[d,"f_{i}"] \\
		X\ar[rr,bul,"C"]\ar[d,"u"']\ar[drr,phantom,"\Two\delta"] && X\ar[d,"u"]\\
		X\ar[r,bul,"C"'] & X\ar[r,bul,"C"'] & X
	\end{tikzcd}=
	\begin{tikzcd}
		X_{i}\ar[rr,bul,"C_{i}"]\ar[d,equal]\ar[drr,phantom,"\Two\delta_{i}"] && X_{i}\ar[d,equal] \\
		X_{i}\ar[r,bul,"C_{i}"]\ar[d,"f_{i}"']\ar[dr,phantom,"\Two\phi_{i}"] & X_{i}\ar[r,bul,"C_{i}"]\ar[d,"f_{i}"']\ar[dr,phantom,"\Two\phi_{i}"] & 
X_{i}\ar[d,"f_{i}"] \\
		X\ar[r,bul,"C"'] & X\ar[r,bul,"C"'] & X
	\end{tikzcd}
\end{displaymath}	
To first ensure that $\delta$ -- the candidate for comultiplication  -- is actually a globular morphism, observe that since $\Gr{s}$ (or 
$\Gr{t}$) preserves colimits, the cocone $(f_{i}\colon X_{i}\to X)_{i\in\ca{I}}$ is a colimit in the category $\dc{D}_{0}$. Then, the above in 
particular implies that $uf_{i}=f_{i}$ for all $i\in\ca{I}$, hence $u=\id_{X}$.

In a very analogous way, the following equalities
\begin{displaymath}
	\begin{tikzcd}
		X_{i}\ar[r,bul,"C_{i}"]\ar[d,"h_\alpha"']\ar[dr,phantom,"\Two D{\alpha}"] & X_{i}\ar[d,"h_\alpha"] \\
		X_{j}\ar[r,bul,"C_{j}"]\ar[d,equal]\ar[dr,phantom,"\Two\epsilon_{j}"] & X_{j}\ar[d,equal] \\
		X_{j}\ar[d,"f_j"']\ar[r,bul,"1_{X_{j}}"']\ar[dr,phantom,"\Two1_{f_j}"] & X_{j}\ar[d,"f_j"] \\
		X\ar[r,bul,"1_X"'] & X
	\end{tikzcd}=
	\begin{tikzcd}
		X_{i}\ar[r,bul,"C_{i}"]\ar[d,equal]\ar[dr,phantom,"\Two\epsilon_{i}"] & X_{i}\ar[d,equal] \\
		X_{i}\ar[r,bul,"1_{X_i}"]\ar[d,"h_\alpha"']\ar[dr,phantom,"\Two1_{h_\alpha}"] & X_{i}\ar[d,"h_\alpha"] \\
		X_{j}\ar[d,"f_j"']\ar[r,bul,"1_{X_{j}}"']\ar[dr,phantom,"\Two1_{f_j}"] & X_{j}\ar[d,"f_j"] \\
		X\ar[r,bul,"1_X"'] & X
	\end{tikzcd}=
	\begin{tikzcd}
		X_{i}\ar[r,bul,"C_{i}"]\ar[d,equal]\ar[dr,phantom,"\Two\epsilon_{i}"] & X_{i}\ar[d,equal] \\
		X_{i}\ar[r,bul,"1_{X_i}"]\ar[d,"f_i"']\ar[dr,phantom,"\Two1_{f_i}"] & X_{i}\ar[d,"f_i"] \\
		X\ar[r,bul,"1_X"'] & X
	\end{tikzcd}
\end{displaymath}
imply the existence of a unique morphism $^v\epsilon^v$ in $\dc{D}_{1}^{\bullet}$ from $C$ to the colimit of the $1_{X_i}$'s
that satisfy $\epsilon\phi_i=(1_{f_i})\epsilon_i$ for all $i\in\ca{I}$,
which in particular gives $v=\id_{X}$, i.e. $\epsilon$ is globular.
	
We will now show that $\delta$ together with $\epsilon$ endow $C$ with a comonad structure. Note that, once 
this goal has been accomplished, the ${}^{f_i}\phi_{i}^{f_i}$ from \cref{eq:phii} will immediately be morphisms of comonads by the properties 
defining $\delta,\epsilon$ above. As an example, for one of the counit axioms we have
\begin{displaymath}
	\begin{tikzcd}
		X_{i}\ar[rr,bul,"C_{i}"]\ar[d,"f_i"']\ar[drr,phantom,"\Two\phi_{i}"] && X_{i}\ar[d,"f_i"] \\
		X\ar[rr,bul,"C"]\ar[d,equal]\ar[drr,phantom,"\Two\delta"] && X\ar[d,equal] \\
		X\ar[r,bul,"C"]\ar[d,equal]\ar[dr,phantom,"\Two\epsilon"] & X\ar[r,bul,"C"]\ar[d,equal] & X\ar[d,equal] \\
		X\ar[r,bul,"1_X"'] & X\ar[r,bul,"C"'] & X
	\end{tikzcd}{=}
	\begin{tikzcd}
		X_{i}\ar[rr,bul,"C_{i}"]\ar[d,equal]\ar[drr,phantom,"\Two\delta_{i}"] && X_{i}\ar[d,equal] \\ 
		X_{i}\ar[r,bul,"C_{i}"]\ar[d,"f_i"']\ar[dr,phantom,"\Two\phi_{i}"] & X_{i}\ar[r,bul,"C_{i}"]\ar[d,"f_i"']\ar[dr,phantom,"\Two\phi_{i}"] & 
X_{i}\ar[d,"f_i"] \\
		X\ar[r,bul,"C"]\ar[d,equal]\ar[dr,phantom,"\Two\epsilon"] & X\ar[r,bul,"C"]\ar[d,equal] & X\ar[d,equal] \\
		X\ar[r,bul,"1_X"'] & X\ar[r,bul,"C"'] & X
	\end{tikzcd}{=}
	\begin{tikzcd}
		X_{i}\ar[rr,bul,"C_{i}"]\ar[d,equal]\ar[drr,phantom,"\Two\delta_{i}"] && X_{i}\ar[d,equal] \\ 
		X_{i}\ar[r,bul,"C_{i}"]\ar[d,equal]\ar[dr,phantom,"\Two\epsilon_{i}"] & X_{i}\ar[r,bul,"C_{i}"]\ar[d,equal] & X_{i}\ar[d,equal] \\
		X_{i}\ar[r,bul,"C_{i}"]\ar[d,"f_i"']\ar[dr,phantom,"\Two1_{f_i}"] & X_{i}\ar[r,bul,"C_{i}"]\ar[d,"f_i"]\ar[dr,phantom,"\Two\phi_{i}"] & 
X_{i}\ar[d,"f_i"] \\
		X\ar[r,bul,"1_X"'] & X\ar[r,bul,"C"'] & X
	\end{tikzcd}{=}
	\begin{tikzcd}
		X_{i}\ar[r,bul,"1_{X_i}"]\ar[d,"f_i"']\ar[dr,phantom,"\Two1_{f_i}"] & 
X_{i}\ar[r,bul,"C_{i}"]\ar[d,"f_i"']\ar[dr,phantom,"\Two\phi_{i}"] & X_{i}\ar[d,"f_i"] \\
	X\ar[r,bul,"1_X"'] & 
X\ar[r,bul,"C"'] & X
	\end{tikzcd}
\end{displaymath}
Since this holds for all $i\in\ca{I}$, by the colimit property in $\dc{D}_{1}^{\bullet}$ we deduce that $(\epsilon\odot1)\delta\cong1$. 
Coassociativity is established in an analogous way through calculations providing an equality $(\Delta\odot1)\Delta\phi_i=(1\odot\Delta)\Delta\phi_i$.

Finally, one can use similar arguments to show that the uniquely induced morphism $C\Rightarrow B$ for a different colimit candidate $B$ is in 
fact a morphism of comonads, as soon as $B\in\Cmd(\dc{D})$ is the vertex of a cocone over the diagram $D$ in $\Cmd(\dc{D})$.
\end{proof}

\begin{cor}\label{cor:MndD0complete}
Let $\dc{D}$ be a fibrant double category. If it is parallel complete, then the fibration $\Mnd(\dc{D})\to\dc{D}_0$ has all fibred limits. If $\dc{D}$ 
is parallel cocomplete, then the opfibration $\Cmd(\dc{D})\to\dc{D}_0$ has all fibred colimits.  
\end{cor} 

\begin{proof}
In a parallel complete double category (\cref{def:parallelcolimits}), the functors $\Gr{s}/\Gr{t}\colon\dc{D}_1\to\dc{D}_0$ are continuous. 
In a fibrant such double category, \cref{prop:(co)limits in (co)monads,lem:(co)limits in D_1 bullet} ensure that limits are preserved through
$\Mnd(\dc{D})\to\dc{D}_{1}^{\bullet}\to\dc{D}_1$. Therefore the fibration $\Mnd(\dc{D})\to\dc{D}_0$ of \cref{prop:MonComonfibred} (ii) which is the composite of the above functors
has all fibred limits, by \cref{prop:fiberwiselimits}. Analogously we verify the case of comonads.

In particular, this means that the reindexing functors $\wc{f}\odot\mi\odot\wh{f}\colon\Mnd(\dc{D})_Y\to\Mnd(\dc{D})_X$ preserve all limits, 
and the reindexing functors $\wh{f}\odot\mi\odot\wc{f}\colon\Cmd(\dc{D})_X\to\Cmd(\dc{D})_Y$ preserve all colimits, for any $f\colon X\to Y$.
\end{proof}

Similarly to lax or oplax monoidal functors between monoidal categories, double functors preserve monads or comonads accordingly in a 
natural way. For a full proof, see \cite[Prop.~3.12]{VCocats}.

\begin{prop}\label{prop:MonFdouble}
	Any lax double functor $F=(F_0,F_1)\colon\dc{D}\to\dc{E}$ induces an ordinary functor
	\begin{displaymath}
		\Mnd F:\Mnd(\dc{D})\to\Mnd(\dc{E})
	\end{displaymath}
	between their categories of monads, by restricting 
	$F_1$ to $\Mnd(\dc{D})$. Dually, any oplax double functor induces $\Cmd F\colon\Cmd(\dc{D})\to\Cmd(\dc{E})$
	between the categories of comonoids.
\end{prop}

\begin{proof}[Sketch]
A monad $(A,\mu,\eta)$ in $\dc{D}$ is mapped to the monad $(FA,F(\mu)\phi,F(\eta)\phi_0)$ in $\dc{E}$, i.e. vertically composing the mapped 
monad structures with the lax double functor structure maps.
\end{proof}

\begin{rmk}\label{rem:oneobjectcase}
Notice that a double category $\dc{D}$ with only one object $*$ and one vertical arrow (its vertical identity) can be thought of as a 
monoidal category $\ca{V}$. Indeed, objects are horizontal arrows $M\colon *\bular *$ and morphisms are the (globular) 2-morphisms $M\Rightarrow N$. 
Horizontal composition then provides a tensor product, with $1_*$ as the unit. More generally, fixing an object $a$ we obtain the 
well-known fact that $(\ca{H}(\dc{D})(a,a),\odot,1_a)$ is a monoidal category for any double category $\dc{D}$. 
	
Clearly a (co)monad in such a special double category becomes a usual (co)monoid in $\ca{V}$, whereas a lax double functor
becomes a lax monoidal functor between the corresponding monoidal categories, therefore reducing \cref{prop:MonFdouble} to the standard 
\cref{eq:MonF}, as was the case for bicategories and lax functors.
\end{rmk}

As a corollary to \cref{prop:MonFdouble}, we deduce that the double functor $\otimes\colon\dc{D}\times\dc{D}\to\dc{D}$ for a 
monoidal double category induces tensor product functors for (co)monads.

\begin{prop}\label{prop:MonDComonDmonoidal}
	If $\dc{D}$ is a (braided/symmetric) monoidal double category, then the categories $\Mnd(\dc{D})$ and $\Cmd(\dc{D})$ inherit a 
(braided/symmetric) monoidal structure from $\dc{D}_1$. 
\end{prop}

For exposition purposes, if $C\colon Z\bular Z$ and $D\colon W\bular W$ are comonads in $\dc{D}$, then $C\ot D$ is a comonad in $\dc{D}$ with 
comultiplication and counit
\begin{displaymath}
 \begin{tikzcd}
Z\ot W\ar[rr,bul,"C\ot D"]\ar[d,equal]\ar[drr,phantom,"\Two\delta\ot\delta"] && Z\ot W\ar[d,equal] \\
Z\ot W\ar[rr,bul,"{(C\circ C)\ot(D\circ D)}"]\ar[d,equal]\ar[drr,phantom,"\Two\cref{eq:monoidaldoubleiso}"] && Z\ot W\ar[d,equal] \\
Z\ot W\ar[r,bul,"C\ot D"'] & Z\ot W\ar[r,bul,"C\ot D"'] & Z\ot W
 \end{tikzcd}\quad
 \begin{tikzcd}
X\ot Y\ar[r,bul,"C\ot D"]\ar[d,equal]\ar[dr,phantom,"\Two\epsilon\ot\epsilon"] & X\ot Y\ar[d,equal] \\
X\ot Y\ar[r,bul,"1_X\ot 1_Y"']\ar[d,equal]\ar[dr,phantom,"\Two\cref{eq:monoidaldoubleiso}"] & X\ot Y\ar[d,equal] \\
X\ot Y\ar[r,bul,"1_{X\ot Y}"'] & X\ot Y
 \end{tikzcd}
\end{displaymath}

In the fibrant setting, the fibration and opfibration of \cref{prop:MonComonfibred} actually become monoidal according to \cref{monoidalfibration}, 
see \cite[Prop.~3.18]{VCocats}.

\begin{prop}\label{prop:Mndmonfib}
For a (braided/symmetric) monoidal fibrant double category $\dc{D}$, the fibration $\Mnd(\dc{D})\to\dc{D}_0$ and the opfibration 
$\Cmd(\dc{D})\to\dc{D}_0$ are (braided/symmetric) monoidal. 
\end{prop}

\begin{rmk}\label{rem:duoidal}
	As recalled in \cref{rem:oneobjectcase}, a double category with a single object 
	and a single vertical arrow reduces to a monoidal category. 
	Clearly, since horizontal composition plays the role of the tensor product, 
	this monoidal category is not expected to be braided. 
		
	When the double category is monoidal, $\dc{D}_1$ 
	has two monoidal structures $\odot$ and $\otimes$ that interact as in \cref{eq:monoidaldoubleiso} 
	and also share their monoidal unit. By an 
	Eckmann-Hilton type argument, see \cite[Prop.~6.11]{Species}, it can be verified 
	that this (one object,one vertical arrow)-case of a monoidal double category 
	gives rise to a braided monoidal category, according also to the so-called 
	\emph{periodic table} \cite{PeriodicTable} for the one-object case of a monoidal 
	bicategory. In that case, $\odot\cong\otimes$ and the categories of (co)monoids 
	-- seen as (co)monads in the trivial case -- become monoidal as discussed 
	earlier. 
	
	Notice that when $\otimes$ is braided, in its suspended case it 
	automatically becomes symmetric \cite[Prop.~6.13]{Species} thus monoids and 
	comonoids inherit it directly as such. 
\end{rmk}

Again due to \cref{prop:MonFdouble}, in a monoidal closed double category $\dc{D}$ the lax double functor $H$ as in 
\cref{eq:internalhomdouble} induces a functor
\begin{equation}\label{eq:MonHdouble}
H\colon\Cmd(\dc{D})^\op\times\Mnd(\dc{D}
)\to\Mnd(\dc{D}) 
\end{equation}
which, for a comonad $C_U$ and a monad $A_X$, produces a monad $H(C,A)\colon H(U,X)\bular H(U,X)$ with multiplication and unit 
\begin{displaymath}
	\begin{tikzcd}[column sep=.4in]
		H(U,X)\ar[r,bul,"{H(C,A)}"]\ar[d,equal]\ar[drr,phantom,"{\Two H_{\odot}}"] & H(U,X)\ar[r,bul,"{H(C,A)}"] & H(U,X)\ar[d,equal] \\
		H(U,X)\ar[rr,bul,"{H(C\odot C, A\odot A)}"] \ar[d, equal]\ar[drr,phantom,"{\Two H(\delta,\mu)}"] && H(U,X)\ar[d,equal] \\
		H(U,X)\ar[rr,bul,"{H(C,A)}"'] && H(U,X)
	\end{tikzcd}\qquad
	\begin{tikzcd}[column sep=.4in]
		H(U,X)\ar[r,bul,"{1_{H(U,X)}}"]\ar[d,equal]\ar[dr,phantom,"{\Two H_U}"] & H(U,X)\ar[d,equal] \\
		H(U,X)\ar[r,bul,"{H(1_U,1_X)}"]\ar[d,equal]\ar[dr,phantom,"{\Two H(\epsilon,\eta)}"] & H(U,X)\ar[d,equal] \\
		H(U,X)\ar[r,bul,"{H(C,A)}"'] & H(U,X)
	\end{tikzcd}
\end{displaymath}
This of course generalizes the analogous fact for monoidal closed categories \cref{defMon[]}, and can be written under the adjunction in the 
analogous convolution structure expression. 

The following proposition establishes that this functor $H$ between the categories of comonads and monads,
is part of a fibred 1-cell between the fibrations of \cref{prop:MonComonfibred} (ii).

\begin{prop}\label{prop:MonHcartesian}
Suppose $\dc{D}$ is a braided monoidal closed and fibrant double category. The top functor 
\begin{displaymath}
\begin{tikzcd}
\Cmd(\dc{D})^\op\times\Mnd(\dc{D})\ar[r,"H"]\ar[d] & \Mnd(\dc{D})\ar[d] \\
\dc{D}_0^\op\times\dc{D}_0\ar[r,"H"] & \dc{D}_0
\end{tikzcd} 
\end{displaymath}
preserves cartesian liftings, and therefore $(H,H)$ is a fibred 1-cell. 
\end{prop}

\begin{proof}
We will show that $H$ preserves cartesian liftings in each variable separately, since cartesian lifts of a product of two fibrations are the pairs of 
lifts in each one, see \cite[Prop.~8.1.14]{Handbook2}.

By \cref{def:locclosed} of a monoidal closed double category, we know that $(\mathfrak s,\mathfrak s)$ and $(\mathfrak t,\mathfrak 
t)$ are maps of adjunctions. Hence restricting to the monoidal closed $\dc{D}_1^\bullet$ (\cref{rmk:clarifying}) 
produces a map of adjunctions
\begin{equation}\label{eq:mapofadjunctions1}
\begin{tikzcd}[column sep=.6in, row sep=.4in]
\dc{D}_1^\bullet\ar[r,shift left=2,"{\mi\ot M}"]\ar[d,"\mathfrak s/\mathfrak t"']\ar[r,phantom,"\bot"] & \dc{D}_1^\bullet\ar[d,"\mathfrak 
s/\mathfrak t"]\ar[l,shift left=2,"{H(M,\mi)}"] \\
\dc{D}_0\ar[r,shift left=2,"{\mi\ot X}"]\ar[r,phantom,"\bot"] & \dc{D}_0\ar[l,shift left=2,"{H(X,\mi)}"]
\end{tikzcd}
\end{equation}
for any endocell $M\colon X\bular X$.
In the fibrant context, we furthermore know that the legs in the above diagram 
are fibrations by \cref{prop:MonComonfibred}. 
Since 
right adjoints in $\mathsf{Cat}^\mathbf{2}$ preserve cartesian arrows (see 
\cite[Ex.~9.4.4]{Jacobs}), we deduce that $H(M,\mi)$ preserves cartesian liftings. Finally, it is verified that its restriction 
$H(C,\mi)\colon\Mnd(\dc{D})\to\Mnd(\dc{D})$ for any comonad $C\colon Z\to Z$ also preserves cartesian liftings, essentially because $\Mnd(\dc{D})$ is 
closed under liftings in $\dc{D}_1^\bullet$. In more detail, starting with a cartesian lifting on the left, we apply $(H(C,\mi),H(Z,\mi))$ on the right
\begin{displaymath}
\begin{tikzcd}[column sep=.2in]
&&&& H(C,\wc{f}\odot B\odot\wh{f})\ar[drr,bend left=15,"{H(1,\Cart(f,B))}"]\ar[d,dashed,"\exists!k"'] &&& \\
\wc{f}\odot B\odot\wh{f}\ar[rr,"{\Cart}"]\ar[d,-,dotted] && B\ar[d,-,dotted] & \textrm{in }\Mnd(\dc{D}) & 
\wc{H(1,f)}\odot H(C,B)\odot\wh{H(1,f)}\ar[rr,"{\Cart}"]\ar[d,-,dotted] && H(C,B)\ar[d,-,dotted] & \textrm{in }\Mnd(\dc{D}) \\
X\ar[rr,"{f}"] && Y & \textrm{in }\dc{D}_0\ar[ur,phantom,near start,"\mapsto"] & H(Z,X)\ar[rr,"{H(1,f)}"] && H(Z,Y) & \textrm{in }\dc{D}_0
\end{tikzcd}
\end{displaymath}
and there exists a unique map $k$ that makes the above triangle commute in $\Mnd(\dc{D})$, therefore in particular in $\dc{D}_1^\bullet$. 
Now since cartesian liftings of the underlying endoarrows in $\dc{D}_1^\bullet$ coincide with those in $\Mnd(\dc{D})$, by the above discussion
$H(1,\Cart(f,B))$ is cartesian in $\dc{D}_1^\bullet$ above $H(1,f)$ therefore there exists a unique isomorphism $u\colon H(C,\wc{f}\odot B\odot\wh{f})\cong 
\wc{H(1,f)}\odot H(C,B)\odot\wh{H(1,f)}$ in $\dc{D}_1^\bullet$ that commutes with $\Cart$ and $H(1,\Cart)$. Although $u$ is not 
required to be a monad map to begin with, universality forces $k$ to be equal to $u$ hence an isomorphism of monads, and the proof is complete.

For preservation of liftings by $H$ on the first variable, consider the following square
\begin{displaymath}
\begin{tikzcd}[column sep=.6in, row sep=.4in]
\dc{D}_1^\bullet\ar[r,shift left=2,"{H(\mi,N)}^\op"]\ar[d,"\mathfrak s"']\ar[r,phantom,"\bot"] & (\dc{D}_1^\bullet)^\op\ar[d,"\mathfrak 
s^\op"]\ar[l,shift left=2,"{H(\mi,N)}"] \\
\dc{D}_0\ar[r,shift left=2,"{H(\mi,Y)^\op}"]\ar[r,phantom,"\bot"] & \dc{D}_0^\op\ar[l,shift left=2,"{H(\mi,Y)}"]
\end{tikzcd} 
\end{displaymath}
that commutes on both sides, where such adjunctions arise in any ordinary braided monoidal closed category. This is also a map of adjunctions, since the counit (the usual 
double-dual embedding) in $\dc{D}_1$ on the left is defined via its tensor-hom adjunct on the right
\begin{displaymath}
\begin{tikzcd}[column sep=.8in]
X\ar[d]\ar[r,tick,"M"]\ar[dr,phantom,pos=.65,"\Two"] & X\ar[d] \\
H(H(X,Z),Z)\ar[r,tick,"{H(H(M,N),N)}"'] & H(H(X,Z),Z)
\end{tikzcd}\qquad
\begin{tikzcd}[column sep=.8in]
H(X,Z)\ot X\ar[d,"\mathrm{ev}_{X,Z}"']\ar[r,tick,"{H(M,N)\ot M}"]\ar[dr,phantom,"\Two\mathrm{ev}_{M,N}"] & H(X,Z)\ot X\ar[d,"\mathrm{ev}_{X,Z}"] \\
Z\ar[r,tick,"N"'] & Z
\end{tikzcd}
\end{displaymath}
hence the source/target of the counit 2-cell on the left is the counit vertical map for $\dc{D}_0$ as per \cref{eq:coev}. Thus viewing 
the legs as fibrations (since $\mathfrak s{=}\mathfrak t$ 
are bifibrations), the right adjoint $H(\mi,N)$ preserves cartesian liftings. In a similar way to above, we can verify that the restricted functor 
between monads $H(\mi,B)\colon\Mnd(\dc{D}^\op)=\Cmd(\dc{D})^\op\to\Mnd(\dc{D})$ for a fixed monad $B$ still preserves cartesian liftings, essentially 
because $\Cmd(\dc{D})$ and $\Mnd(\dc{D})$ are closed under cocartesian and cartesian liftings respectively in $\dc{D}_1^\bullet$ as seen in 
\cref{prop:MonComonfibred}.
\end{proof}

\begin{rmk}\label{rmk:fixedaction}
It is the case that for a monoidal closed double category $\dc{D}$, both functors 
$H_1^\bullet\colon(\dc{D}_1^\bullet)^\op\times\dc{D}_1^\bullet\to\dc{D}_1^\bullet$ which is the internal hom of 
$\dc{D}_1^\bullet$ as per \cref{rmk:clarifying}, and $H\colon\Cmd(\dc{D})^\op\times\Mnd(\dc{D})\to\Mnd(\dc{D})$ induced by $H_1$, are actions (e.g. 
\cite[Lem.~3.22]{VCocats}\footnote{The condition therein that $\dc{D}$ is \emph{braided} as a monoidal double category is in fact superfluous.}). 
We here record the action structure maps in $\dc{D}_1^\bullet$
\begin{equation}\label{eq:actionstructuredouble}
\begin{tikzcd}[column sep=.6in]
H(Z,H(W,Y))\ar[r,bul,"{H(C,H(D,B))}"]\ar[d,"\chi_0"']\ar[dr,phantom,"\Two\chi_1"] & H(Z,H(W,Y))\ar[d,"\chi_0"] \\
H(Z\ot W,Y)\ar[r,bul,"{H(C\ot D,B)}"'] & H(Z\ot W,Y)
\end{tikzcd}\quad
\begin{tikzcd}[column sep=.6in]
Y\ar[r,bul,"B"]\ar[d,"\nu_0"']\ar[dr,phantom,"\Two\nu_1"] & Y\ar[d,"\nu_0"] \\
H(I,Y)\ar[r,bul,"{H(1_I,B)}"'] & H(I,Y)
\end{tikzcd}
\end{equation}
which correspond under the adjunction $\otimes_1\dashv H_1$ to
\begin{equation}\label{eq:chiunder}
\begin{tikzcd}[column sep=1in]
H(Z,H(W,Y))\ot Z\ot W\ar[d,"\mathrm{ev}\ot1"']\ar[r,bul,"{H(C,H(D,B))\ot C\ot D}"]\ar[dr,phantom,"\Two\mathrm{ev}\ot1"] & 
H(Z,H(W,Y))\ot Z\ot W\ar[d,"\mathrm{ev}\ot1"] \\
H(W,Y)\ot W\ar[d,"\mathrm{ev}"']\ar[r,bul,"{H(D,B)\ot D}"']\ar[dr,phantom,"\Two\mathrm{ev}\phantom{\ot111}"] & H(W,Y)\ot W\ar[d,"\mathrm{ev}"] \\
Y\ar[r,bul,"B"'] & Y
\end{tikzcd}\quad
\begin{tikzcd}[row sep=.8in]
Y\ot I\ar[r,bul,"{B\ot1_I}"]\ar[d,"\rho"']\ar[dr,phantom,"\Two\rho"] & Y\ot I\ar[d,"\rho"] \\
Y\ar[r,bul,"B"'] & Y
\end{tikzcd}
\end{equation}
using \cref{eq:coev} and the right unitor, as in any ordinary monoidal closed category. Notice that by the above diagrams, it is ensured that the vertical source and target 
of the 2-cells $\chi_1$ and $\nu_1$ are the structure maps $\chi_0$ and $\nu_0$ of the action $H_0\colon\dc{D}^\op_0\times\dc{D}_0\to\dc{D}_0$.
When $C\colon Z\bular Z$ and $D\colon W\bular W$ are comonads and $B\colon Y\bular Y$ is a monad in $\dc{D}$, the cells \cref{eq:actionstructuredouble} are indeed verified to commute with the monad structures of the involved endo-1-cells, hence they are monad maps.
\end{rmk}

\begin{ex}
In the case $\dc{D}=\VMMat$ for a $\ca{V}$ which is monoidal closed with products, the induced functor \cref{eq:MonHdouble} between the categories of 
(co)monads (\cref{ex:VCatsaremonadsinVMat}) is 
\begin{displaymath}
	H\colon\VCocat^\op\times\VCat\longrightarrow\VCat
\end{displaymath}
mapping a pair $(C_X,B_Y)$ of $\ca{V}$-cocategory and a $\ca{V}$-category to a $\ca{V}$-category $H(C,B)_{Y^X}$ with hom-objects 
$H(C,B)(s,k)=\prod_{x,x'}[C(x,x'),B(sx,kx')]$, and composition arrows
\begin{displaymath}
	\prod_{a,a'}{[C(a,a'),B(sa,ka')]}\otimes\prod_{b,b'}{[C(b,b'),B(kb,tb')]}\to\prod_{c,c'}{[C(c,c'),B(sc,tc')]}
\end{displaymath}
defined via their transpose under the tensor-hom adjunction in $\ca{V}$, using the (co)composition arrows and the fact that $\otimes$ commutes with 
sums accordingly, generalizing the convolution product.
\end{ex}

\subsection{Locally presentable double categories and enrichment}\label{sec:lpdoublecats}

In Theorems 3.23 and 3.24 of \cite{VCocats}, the enrichment of monads in comonads as well as the enriched fibration structure is 
subject to the extra condition that $H$ as well as the pair 
$(H\colon\Cmd(\dc{D})^\op\times\Mnd(\dc{D})\to\Mnd(\dc{D}),H_0\colon\dc{D}_0^\op\times\dc{D}_0\to\dc{D}_0)$ as in \cref{prop:MonHcartesian} have appropriate adjoints. 
Considering whether this is true in a general double context led us to consider a notion of a \emph{locally 
presentable} double category, or essentially one on the `local' level of $\dc{D}_0$ and $\dc{D}_1$, as a 
sufficient setting. Although a general investigation of such a topic would go beyond the scope of the current work, given how even on the level of 
2-categories relevant research is still ongoing (see e.g. \cite{IvanFosco,BourkeAcc2cat}), we here introduce such a notion which in the same spirit 
as \cref{def:locclosed} serves for our purposes. By doing so, we manage to obtain the relevant results in a more general double categorical 
context than \cite{VCocats}, as exhibited in \cref{thm:big1}. 

Recall that ${}^X\dc{D}_1$ and $\dc{D}_1^X$ are the the fixed-domain and fixed-codomain subcategories of $\dc{D}_1$ as defined in the beginning of \cref{sec:doublecats}.

\begin{defi}\label{defi:locallypresentabledoublecat}
A fibrant double category $\dc{D}$ will be called \emph{locally $\lambda$-presentable}, for a regular cardinal $\lambda$, if both categories $\dc{D}_0$, $\dc{D}_1$ are locally $\lambda$-presentable, the functors 
$\Gr{s},\Gr{t}\colon\dc{D}_{1}\to\dc{D}_{0}$ are 
cocontinuous right 
adjoints and, for 
every horizontal 1-cell $M\colon X\bular Y$, the functors $-\odot M\colon^{Y}\dc{D}_{1}\to{}^{X}\dc{D}_{1}$ and $M\odot 
-\colon\dc{D}_{1}^{X}\to\dc{D}_{1}^{Y}$ are accessible.
$\dc{D}$ will be called \emph{locally presentable} if it is locally $\lambda$-presentable for some regular cardinal $\lambda$.
\end{defi}

%

As a direct implication of the above definition, we obtain the following.

\begin{cor}\label{prop:lpparcomp}
	A locally presentable double category $\dc{D}$ is parallel complete and cocomplete. Moreover, $\mathfrak s$ and $\mathfrak t$ both have all fibred limits and opfibred colimits.
\end{cor}

\begin{proof}
	\cref{def:parallelcolimits} is straight away satisfied by the clauses of the locally presentable double category definition. Moreover, the source and
	target fibrations of \cref{globalvslocal} satisfy the clauses of \cref{prop:fiberwiselimits}.
\end{proof}

Perhaps some remarks are in order concerning definition \ref{defi:locallypresentabledoublecat}. One thing that the reader might immediately notice is that the functors $-\odot M\colon^{Y}\dc{D}_{1}\to{}^{X}\dc{D}_{1}$ and $M\odot 
-\colon\dc{D}_{1}^{X}\to\dc{D}_{1}^{Y}$ are not required to be specifically $\lambda$-accessible. In order to elaborate on this, but also because the accessibility of fixed domain/codomain categories $^{X}\dc{D}_1$, 
$\dc{D}_{1}^{X}$ will be of importance later on, we begin by making the following observation concerning fibrations between accessible categories.

\begin{lem}\label{lem:accessiblefibers}
	Suppose that $P\colon\ca{A}\to\ca{X}$ is a 
	fibration where both categories $\ca{A},\ca{X}$ and the functor $P$ are accessible. Then
	for any object $X\in\ca{X}$, the fiber $\ca{A}_{X}$ of $P$ above $X$ is an accessible category.
\end{lem}

\begin{proof}
	For any object $X\in\ca{X}$, the fiber $\ca{A}_{X}$ of $P$ above $X$ can be seen as the following pullback 
	\begin{displaymath}
		\begin{tikzcd}
			\ca{A}_{X}\ar[d]\ar[r] & \ca{A}\ar[d,"P"] \\
			\B{1}\ar[r,"p_X"'] & \ca{X}
		\end{tikzcd}
	\end{displaymath}
	where $\B{1}=\{*\}$ is the terminal category and $p_X$ denotes the functor that picks out $X\in\ca{C}$. It is clear that all limits and colimits 
	exist (and are just the object $*$) and $\B{1}$ is locally $\lambda$-presentable for every $\lambda$. It is furthermore not too hard to see that the 
	functor $p_X$ preserves \emph{connected} (co)limits, so in particular filtered colimits. Hence, we can apply the Limit Theorem for accessible 
	categories \cite[Thm.~5.1.6]{MakkaiPare}, taking into account that $P$ being a fibration implies that the pullback is also a bi-pullback, to deduce that $\ca{A}_X$ is an 
	accessible 
	category. The same works for $P$ an opfibration, since the crucial ingredient for the pullback to coincide with the bi-pullback is that $P$ be an 
	\emph{iso}fibration.
\end{proof}

If we apply this observation to the context of a locally presentable double category, we deduce the following, for the fixed-domain/codomain subcategories
of $\dc{D}_1$.

\begin{prop}\label{prop:lp fibers double cats}
	Suppose that $\dc{D}$ is a locally presentable double category. Then for every $X, Y\in\dc{D}_0$ the categories $^{X}\dc{D}_1$, 
	$\dc{D}_{1}^{Y}$ and ${}^X\dc{D}_1^Y=\ca{H}(\dc{D})(X,Y)$ are locally presentable.
\end{prop} 

\begin{proof}
	The categories $^{X}\dc{D}_1$ and 
	$\dc{D}_{1}^{Y}$ are fibers of the (op)fibrations $\Gr{s}\colon\dc{D}_1\to\dc{D}_0$ and $\Gr{t}\colon\dc{D}_1\to\dc{D}_0$ respectively. Since $\Gr{s},\Gr{t}$ are cocontinuous, \cref{lem:accessiblefibers} shows that both categories are accessible, while in addition being cocomplete by \ref{prop:equivalentdefparcom}.

	Similarly, the category ${}^X\dc{D}_1^Y$ can be viewed as a fiber of the restricted source opfibration $^X\dc{D}_1\hookrightarrow\dc{D}_1\to\dc{D}_0$ (see \cref{globalvslocal}). If $\lambda$ is a regular cardinal for which both $^X\dc{D}_1$ and $\dc{D}_1$ are $\lambda$-accessible, then the inclusion $^X\dc{D}_1\hookrightarrow\dc{D}_1$ is $\lambda$-accessible because it preserves all connected colimits by (the dual of) \cref{connected}. Thus, the composite $^{X}\dc{D}_1\to\dc{D}_0$ is $\lambda$-accessible and the same arguments as above apply.
\end{proof}

As an example, we establish that categories of the form $^X\VMMat_{1}$ and $\VMMat_{1}^{Y}$ are locally $\lambda$-presentable when $\ca{V}$ is, and we also recover the well-known fact that
the categories ${}^X\VMMat_{1}^{Y}=\ca{V}^{X\times Y}$ are locally $\lambda$-presentable (\cite[Corollary~1.54]{LocallyPresentable}).

\begin{rmk}
Notice that if $\dc{D}$ is locally $\lambda$-presentable it does not follow that the categories $^{X}\dc{D}_1$, 
$\dc{D}_{1}^{Y}$ and ${}^X\dc{D}_1^Y$ are themselves locally $\lambda$-presentable for the same $\lambda$. This is because the proof of \ref{prop:lp fibers double cats} relies on \ref{lem:accessiblefibers}, whose proof in turn uses the Limit Theorem for \emph{accessible} categories and in which the index of accessibility $\lambda$ is not invariant. On the other hand, the Limit Theorem for locally presentable categories (in which $\lambda$ remains fixed) cannot be applied because the functor $\B{1}\to\ca{X}$ does not preserve arbitrary (co)limits.

Coming back to the definition of locally $\lambda$-presentable double category in \ref{defi:locallypresentabledoublecat}, it seems reasonable that one could choose to require that all categories $^{X}\dc{D}_{1}$ and $\dc{D}_1^Y$ are also locally $\lambda$-presentable and then that the functors $-\odot M$ and $M\odot-$ are all $\lambda$-accessible. We opted against this choice for two main reasons. Firstly, on a more practical level, the specific cardinal $\lambda$ does not really play a role in the main results of this paper. Secondly, we felt that such considerations would lead in the direction of a general theory of accessible double categories, which is beyond the scope of the current paper and best left for future work. By the same token, we do not require the functor $\B{1}\colon\dc{D}_0\to\dc{D}_1$ to be accessible, even though that would again be a natural condition to consider and indeed holds for $\VMMat$.
\end{rmk}

\begin{rmk}\label{rmk:compbetweenfibers}
	Given any horizontal 1-cell $M\colon X\bular Y$ in a parallel cocomplete fibrant double category $\dc{D}$, the requirement that $-\odot 
	M\colon^{Y}\dc{D}_{1}\to{}^{X}\dc{D}_{1}$ preserve (filtered) colimits is equivalent to requiring that all functors $-\odot 
	M\colon^{Y}\dc{D}_{1}^Z\to{}^{X}\dc{D}^Z_{1}$ preserve (filtered) colimits. This follows from \cref{prop:continuous fibred 1-cells} because
	the functor $-\odot M$ constitutes an opfibred functor
	\begin{displaymath}
		\begin{tikzcd}
			^{Y}\dc{D}_1\ar[dr]\ar[rr,"-\odot M"] && ^{X}\dc{D}_1\ar[dl] \\
			& \dc{D}_0 & 
		\end{tikzcd}
	\end{displaymath}
	with respect to the restricted target opfibrations, which are opfibred cocomplete by \cref{prop:equivalentdefparcom}.
	A similar remark applies to $M\odot-$. This observation can be useful in practice in order to check the accessibility of $-\odot M$ and $M\odot-$ required in \cref{defi:locallypresentabledoublecat}, as the fixed domain and codomain categories ${}^{Y}\dc{D}^Z_{1}=\ca{H}(\dc{D})(Y,Z)$ are often easier to handle.
\end{rmk}

\begin{ex}
The double category $\dc{R}\nc{el}(\nc{Set})$ is locally finitely presentable.
	
To begin with, the source and target functors $\Gr{s},\Gr{t}\colon\dc{R}\nc{el}(\nc{Set})_{1}\to\nc{Set}$ have both a left and a right adjoint. For $\Gr{s}$ the left adjoint maps any set $X$ to the empty relation $X\bular\emptyset$, while the right adjoint maps it to $T_X=X\times 1\colon X\bular 1$. Swapping sides gives the corresponding adjoints for $\Gr{t}$. It is easy to see that a relation $R\colon X\bular Y$ is a finitely presentable object of $\dc{R}\nc{el}(\nc{Set})_{1}$ precisely if it is a finite subset of $X\times Y$, and every relation can be written as the filtered colimit of its finite sub-relations. If $R_{i}\colon X_{i}\bular Y, i\in\ca{I}$ is a filtered diagram in $\dc{R}\nc{el}(\nc{Set})_{1}^Y$, then its colimit is the relation $\rm{Colim}R_{i}\colon\rm{Colim}X_{i}\bular Y$ given by $\rm{Colim}R_{i}=\{(q_{i}(x),y)|i\in\ca{I}, (x,y)\in R_i\}$, where $q_{i}\colon X_{i}\to\rm{Colim}X_i$ are the legs of the coproduct in $\nc{Set}$. It is then clear that this colimit is preserved by $M\odot -$ for any $M\colon Y\bular Z$ and similarly for the functors $-\odot M$.
\end{ex}


As far as we are concerned in the present paper, the most important aspect of \cref{defi:locallypresentabledoublecat} lies in the fact that it manages to capture the example of $\dc{D}=\VMMat$ of \cref{ex:VMMat}. We particularly
thank Steve Lack for the first step of the proof below.

\begin{prop}\label{prop:VMatlp}
	Suppose that $\ca{V}$ is locally presentable with colimits which are preserved by $\otimes$ in each variable. Then the double category $\VMMat$ is locally presentable.
\end{prop}

\begin{proof}
	Suppose that $\ca{V}$ is locally 
	$\lambda$-presentable. Of course $\VMMat_{0}=\nc{Set}$ is locally $\lambda$-presentable, but we furthermore have a pullback square as follows:
	\begin{displaymath}
		\begin{tikzcd}
			\VMMat_{1}\ar[d,"{(\Gr{s},\Gr{t})}"']\ar[r] & \nc{Fam}(\ca{V})\ar[d] \\
			\nc{Set}^{2}\ar[r,"\times"'] & \nc{Set}.
		\end{tikzcd}
	\end{displaymath}
	The bottom functor is a $\lambda$-accessible right adjoint, while the right vertical $\nc{Fam}(\ca{V})\to\nc{Set}$ has both a left and a right 
	adjoint 
	and the category $\nc{Fam}(\ca{V})$ is itself locally $\lambda$-presentable. In addition, $\nc{Fam}(\ca{V})\to\nc{Set}$ is an isofibration, which 
	implies that the pullback is also a bipullback. Thus, by the Limit Theorem for locally presentable categories (e.g. \cite[Ex.~2.l]{LocallyPresentable} or \cite[Thm.~2.18]{Bird}) we deduce that $\VMMat_{1}$ is locally $\lambda$-presentable and the two pullback projections are $\lambda$-accessible right adjoints. The 
	latter implies in particular that both $\Gr{s}$, $\Gr{t}$ are themselves $\lambda$-accessible right adjoints. Furthermore, by \cref{ex:parallelcolimitsinVMMat} we know that $\Gr{s}$ and $\Gr{t}$ preserve all colimits.
	
	Finally, we claim that for every $M\colon Y\bular Z$, the 
	functor $M\odot-\colon\VMMat_{1}^{Y}\to\VMMat_{1}^{Z}$ preserves filtered colimits. This follows from the description of colimits in such fibers 
	established in \ref{ex:colimitsinVMMatfibers} and the following sequence of isomorphisms
	\begin{align*}
		M\odot C(z,x) &=\sum\limits_{y\in Y}M(z,y)\otimes C(y,x)\cong\sum\limits_{y\in 
			Y}\left(M(z,y)\otimes\Big(\mathrm{Colim}_{i}\sum\limits_{q_{i}(x_i)=x}Di(y,x_i)\Big)\right) \\
		&\cong\sum\limits_{y\in 
			Y}\left(\mathrm{Colim}_{i}\Big(M(z,y)\otimes\sum\limits_{q_{i}(x_i)=x}Di(y,
		x_i)\Big)\right)\cong\sum\limits_{y\in 
			Y}\mathrm{Colim}_{i}\left(\sum\limits_{q_{i}(x_i)=x}M(z,y)\otimes 
		Di(y,x_i)\right) \\
		&\cong\mathrm{Colim}_{i}\sum\limits_{q_{i}(x_i)=x}\sum\limits_{y\in 
			Y}\Big(M(z,y)\otimes 
		Di(y,x_i)\Big)\cong\mathrm{Colim}_{i}\sum\limits_{q_{i}(x_i)=x}(M\odot 
		Di)(z,x_i)
	\end{align*}
	and the latter is precisely the value of $\mathrm{Colim}_{i}(M\odot Di)$ at $(z,x)$.
	
	A similar calculation is valid for the functors $-\odot M\colon{}^{Z}\VMMat_{1}\to{}^{Y}\VMMat_{1}$.
\end{proof}

\begin{prop}\label{thm:D_1 bullet lp}
	Suppose that $\dc{D}$ is a locally presentable double category. Then the category $\dc{D}_{1}^{\bullet}$ is also locally presentable.
\end{prop}
\begin{proof}
	As already observed in \cref{rem:D1bul} for any fibrant $\dc{D}$, $\dc{D}_{1}^{\bullet}$ can be viewed as the following pullback 
	\begin{displaymath}
		\begin{tikzcd}
			\dc{D}_{1}^{\bullet}\ar[r]\ar[d] & \dc{D}_{1}\ar[d,"{\langle \Gr{s},\Gr{t}\rangle}"] \\
			\dc{D}_{0}\ar[r,"\Delta"'] & \dc{D}_{0}\times\dc{D}_{0}
		\end{tikzcd}
	\end{displaymath}
	In this square we observe first of all that the diagonal functor $\Delta\colon\dc{D}_{0}\to\dc{D}_{0}\times\dc{D}_{0}$ is accessible and has a left 
	adjoint, the latter because $\dc{D}_{0}$ has binary coproducts. Similarly, $\langle\Gr{s},\Gr{t}\rangle$ is accessible and has a left adjoint because 
	both $\Gr{s}$ and $\Gr{t}$ have left adjoints and $\dc{D}_{1}$ has binary coproducts. Thus, the 
	Limit 
	Theorem for locally presentable categories applies to yield that 
	$\dc{D}_{1}^{\bullet}$ is locally presentable, while the inclusion functor $\dc{D}_{1}^{\bullet}\to\dc{D}_{1}$ (the top horizontal map in the 
	pullback) is an accessible right adjoint.	
\end{proof}

As an example, in the case of $\dc{D}=\VMMat$, the above proposition provides an elegant argument proving the fact that the category of $\ca{V}$-graphs and $\ca{V}$-graph homomorphisms $\VGrph$ is a locally presentable category
when $\ca{V}$ is, see \cite[Prop.~4.4]{KellyLack}.

\begin{cor}\label{cor:MonComonfibreslp}
	If $\dc{D}$ is a locally presentable double category, the fibres of $\Mnd(\dc{D})\to\dc{D}_0$ and $\Cmd(\dc{D})\to\dc{D}_0$ are locally 
	presentable categories. 
\end{cor}
\begin{proof}
	The fibres of $\Mnd(\dc{D})\to\dc{D}_0$ and $\Cmd(\dc{D})\to\dc{D}_0$ above any $X\in\dc{D}_0$ are the categories $\Mon(\ca{H}(\dc{D})(X,X))$ and $\Comon(\ca{H}(\dc{D})(X,X))$ respectively, see \cref{prop:MonComonfibred}. By \cref{prop:MonComonlp}, these are locally presentable: $\ca{H}(\dc{D})(X,X)$ itself is locally presentable by \cref{prop:lp fibers double cats}, and tensoring -- in this case, taking the horizontal composition -- with a fixed endo-1-cell $M\colon X\bular X$ on either side preserves filtered colimits by \cref{rmk:compbetweenfibers}. 
\end{proof}

Below we establish monoidal closedness of $\Cmd(\dc{D})$ under the running assumptions, with a proof strategy is reiterated in the proof of our main 
\cref{thm:big1}. We note that this result is new, since the weaker previous assumptions in \cite{VCocats} did not suffice to obtain it in a general context.

\begin{prop}\label{prop:Cmdclosed}
Let $\dc{D}$ be a monoidal closed double category which is locally presentable. Then the category of comonads $\Cmd(\dc{D})$ is a monoidal 
closed category. 
\end{prop}

\begin{proof}
Recall by \cref{prop:MonDComonDmonoidal} that $\Cmd(\dc{D})$ inherits the monoidal structure of $\dc{D}_1$, therefore the following diagram involving tensors commutes 
 \begin{displaymath}
  \begin{tikzcd}[column sep=.5in]
 \Cmd(\dc{D})\ar[r,"{\mi\ot C}"]\ar[d] & \Cmd(\dc{D})\ar[d] \\
 \dc{D}_0\ar[r,shift left=2,"{\mi\ot Z}"]\ar[r,phantom,"\bot"] & \dc{D}_0\ar[l,shift left=2,"{H(Z,\mi)}"]
  \end{tikzcd}
 \end{displaymath}
for any comonad $C\colon Z\bular Z$, 
and the adjunction between the bases is the monoidal closure of $\dc{D}_0$. In the fibrant setting, the legs are opfibrations by 
\cref{prop:MonComonfibred}, and moreover we now verify that the top functor preserves cocartesian liftings. This can be done either by a direct computation
or via an argument analogous to the proof of \cref{prop:MonHcartesian} 
as follows. Since the diagram 
\cref{eq:mapofadjunctions1} is a map of adjunctions in any monoidal closed double category, and the legs are bifibrations thus in particular 
opfibrations, the total left adjoint $\mi\otimes C\colon\dc{D}_1^\bullet\to\dc{D}_1^\bullet$ preserves all cocartesian morphisms (see e.g. \cite[Ex.~9.4.4]{Jacobs}).
Since $\Cmd(\dc{D})$ 
is closed under cocartesian liftings in $\dc{D}_1^\bullet$, the restriction of $\mi\otimes C$ to comonads is verified to also be a cocartesian 
functor.

We can now apply \cref{thm:totaladjointthm} to obtain the adjoint of $\mi\otimes C$. Indeed, the composite functor between the fibers
\begin{displaymath}
\Cmd(\dc{D})_{H(Z,W)}\xrightarrow{(\mi\otimes C)_{H(Z,W)}}\Cmd(\dc{D})_{H(Z,W)\otimes Z}\xrightarrow{(\varepsilon_W)_!}\Cmd(\dc{D})_W
\end{displaymath}
has a right adjoint as follows. The domain is a locally presentable category by \cref{cor:MonComonfibreslp} and the reindexing functor 
$(\varepsilon_W)_!$ is cocontinuous by \cref{cor:MndD0complete}. Finally, the functor $(\mi\otimes C)_{H(Z,W)}$ 
between the fibers is cocontinuous by \cref{prop:continuous fibred 1-cells}, since the total functor is cocontinuous 
by the following commutative square
\begin{displaymath}
\begin{tikzcd}[column sep=.6in]
\Cmd(\dc{D})\ar[d]\ar[r,"{\mi\ot C}"] & \Cmd(\dc{D})\ar[d] \\
\dc{D}_1^\bullet\ar[r,"{\mi\ot C}"] & \dc{D}_1^\bullet
\end{tikzcd}
\end{displaymath}
where $C$ below is the underlying endoarrow of the comonad, both legs create colimits by \cref{prop:(co)limits in (co)monads} and the bottom functor 
preserves colimits since it has a right adjoint in the monoidal closed $\dc{D}_1^\bullet$.

Therefore we get an adjunction $\mi\otimes C\dashv\Hom_\Cmd(C,\mi)$ for any comonad $C$, giving rise to
\begin{displaymath}
 \Hom_\Cmd\colon\Cmd(\dc{D})^\op\times\Cmd(\dc{D})\to\Cmd(\dc{D})
\end{displaymath}
where the comonad $\Hom_\Cmd(C,D)$ for comonads $C\colon Z\bular  Z$ and $D\colon W\bular W$ is a horizontal 1-cell $H(Z,W)\bular H(Z,W)$.
\end{proof}

We close up this section with a central result, that exhibits the effectiveness of the new parts of the theory on locally presentable double categories (\cref{defi:locallypresentabledoublecat}) and monoidal closed double categories (\cref{def:locclosed}) discussed in these past few sections. In more detail, we re-establish the main \cite[Theorems 3.23 \& 3.24]{VCocats} in a broader context, dropping technical conditions that previously needed to be verified by hand in each specific double category of interest.

\begin{thm}\label{thm:big1}(Sweedler theory of (co)monads)
Let $\dc{D}$ be a braided monoidal closed double category, which is locally presentable.
\begin{enumerate}
 \item The category of monads $\Mnd(\dc{D})$ is tensored and cotensored enriched in the category of comonads $\Cmd(\dc{D})$.
 \item The fibration $\Mnd(\dc{D})\to\dc{D}_0$ is enriched in the opfibration $\Cmd(\dc{D})\to\dc{D}_0$.
\end{enumerate}
\end{thm}

\begin{proof} $(1)$ We apply \cref{thm:cotensorenrich} for the functor 
$H\colon\Cmd(\dc{D})^\op\times\Mnd(\dc{D})\to\Mnd(\dc{D})$ as in 
\cref{eq:MonHdouble}, induced by the monoidal closed structure of 
$\dc{D}$.

First recall that $\Cmd(\dc{D})$ is indeed a braided monoidal category by \cref{prop:MonDComonDmonoidal}. Moreover, $H$ is 
an action as discussed in \cref{rmk:fixedaction}. Finally, the functor $H^\op(\mi,B)$ for any monad $B\colon Y\bular Y$ has a right adjoint using 
\cref{thm:totaladjointthm} as follows. There is a square of categories and functors
\begin{displaymath}
\begin{tikzcd}[column sep=.6in]
\Cmd(\dc{D})\ar[r,"{H^\op(\mi,B)}"]\ar[d] & \Mnd(\dc{D})^\op\ar[d] \\
\dc{D}_0\ar[r,"{H^\op(\mi,Y)}"'] & \dc{D}_0^\op 
\end{tikzcd}
\end{displaymath}
which is an opfibred 1-cell between the opfibrations of \cref{prop:MonComonfibred}, since $H^\op(\mi,B)$ preserves cocartesian 
liftings by \cref{prop:MonHcartesian}. There is an adjunction $H^\op(\mi,Y)\dashv H(\mi,Y)$ between the base categories, as is the case in any 
braided monoidal closed category -- here $\dc{D}_0$. Finally, the composite functor between the fibres
\begin{displaymath}
\Cmd(\dc{D})_{H(X,Y)}\xrightarrow{H^\op(\mi,B)_{H(X,Y)}}\Mnd(\dc{D})^\op_{H(H(X,Y),Y)}\xrightarrow{(\varepsilon_Y)_!}\Mnd(\dc{D})^\op_{X}
\end{displaymath}
has a right adjoint as follows. First of all, the domain $\Cmd(\dc{D})_{H(X,Y)}=\Comon(\ca{H}(\dc{D})(H(X,Y),H(X,Y))$ is locally presentable by 
\cref{cor:MonComonfibreslp}, and the reindexing functor $\wc{\varepsilon_Y}\odot\mi\odot\wh{\varepsilon_Y}$ is continuous by 
\cref{cor:MndD0complete}, hence its opposite for the opfibration $\Mnd(\dc{D})^\op\to\dc{D}_0^\op$ is cocontinuous. Finally, the functor $H^\op(\mi,B)_{H(X,Y)}$ between 
the fibers is cocontinuous\footnote{As an alternative proof, one could consider the cocontinuous functor 
$H(\mi,Y)^\op_{H(X,Y)}$ induced between the fibers of 
$\dc{D}_1^\bullet$ by $H(\mi,Y)$, and take its lifting between (co)monadic categories of (co)monoids respectively, see \cref{prop:MonComonfibred}.} 
by \cref{prop:continuous fibred 1-cells}, because $H^\op(\mi,B)$ is so by the following commutative square
\begin{displaymath}
\begin{tikzcd}[column sep=.6in]
\Cmd(\dc{D})\ar[d]\ar[r,"{H^\op(\mi,B)}"] & \Mnd(\dc{D})^\op\ar[d] \\
\dc{D}_1^\bullet\ar[r,"{H^\op(\mi,B)}"] & (\dc{D}_1^\bullet)^\op
\end{tikzcd}
\end{displaymath}
where both legs create colimits by \cref{prop:(co)limits in (co)monads} and the base functor is cocontinuous as the internal 
hom of the braided monoidal closed $\dc{D}_1^\bullet$.
Therefore $\Mnd(\dc{D})$ is enriched in $\Cmd(\dc{D})$, with enriched hom functor the induced parameterized 
adjoint 
\begin{equation}\label{eq:P}
P\colon\Mnd(\dc{D})^\op\times\Mnd(\dc{D})\to\Cmd(\dc{D})
\end{equation}
of $H$ via $H(\mi,B)^\op\dashv P(\mi,B)$, where $P(A_X,B_Y)$ for two monads is a comonad with carrier object $H(X,Y)$.

Next, $\Cmd(\dc{D})$ is a monoidal closed category by \cref{prop:Cmdclosed}, therefore the enrichment admits cotensors $H(C,B)\colon 
H(Z,Y)\bular H(Z,Y)$ for any comonad $C$ and monad $B$. Finally, the enrichment also admits tensors $C\triangleright B$ since 
$H(C,\mi)$ has a left adjoint $C\triangleright\mi$, using the dual of \cref{thm:totaladjointthm} as follows. The following commutative diagram
\begin{displaymath}
 \begin{tikzcd}[column sep=.6in]
\Mnd(\dc{D})\ar[r,"{H(C,\mi)}"]\ar[d] & \Mnd(\dc{D})\ar[d] \\
\dc{D}_0\ar[r,"{H(Z,\mi)}"] & \dc{D}_0
 \end{tikzcd}
\end{displaymath}
is a fibred 1-cell since $H(C,\mi)$ preserves cartesian liftings by \cref{prop:MonHcartesian}, and moreover there is an adjunction $Z\ot\mi\dashv 
H(Z,\mi)$ between the bases with unit $\eta\colon X\to H(Z,X\ot Z)$ since $\dc{D}_0$ is braided monoidal closed. Furthermore, the composite functor 
between fibers
\begin{displaymath}
\Mnd(\dc{D})_{X\ot Z}\xrightarrow{H(C,\mi)_{X\ot Z}}\Mnd(\dc{D})_{H(Z,X\ot Z)}\xrightarrow{(\eta_X)^*}\Mnd(\dc{D})_{X}
\end{displaymath}
has a left adjoint, because its domain $\Mnd(\dc{D})_{X\ot Z}=\Mon(\ca{H}(\dc{D})(X\ot Z,X\ot Z))$ is locally presentable by 
\cref{cor:MonComonfibreslp}, the reindexing functor is continuous by \cref{cor:MndD0complete}, and the functor between
the fibers is continuous by \cref{prop:continuous fibred 1-cells} similarly to above: in the commutative
\begin{displaymath}
 \begin{tikzcd}[column sep=.6in]
 \Mnd(\dc{D})\ar[r,"{H(C,\mi)}"]\ar[d] & \Mnd(\dc{D})\ar[d] \\
 \dc{D}_1^\bullet\ar[r,"{H(C,\mi)}"] & \dc{D}_1^\bullet
 \end{tikzcd}
\end{displaymath}
both legs create limits by \cref{prop:(co)limits in (co)monads} and the bottom functor is a right adjoint.
As a result, $H(C,\mi)$ has a left adjoint $C\triangleright\mi$ between the total categories for any $C$,
which induces a left parameterized adjoint
\begin{displaymath}
\triangleright\colon\Mnd(\dc{D})\times\Cmd(\dc{D})\to\Mnd(\dc{D})
\end{displaymath}
that gives tensors $C\triangleright B\colon Z\ot Y\bular Z\ot Y$ for the enrichment.

$(2)$ This is established using \cref{thm:enrichedfib}. First recall that the opfibration $\Cmd(\dc{D})\to\dc{D}_0$ of \cref{prop:MonComonfibred} for a monoidal fibrant double category $\dc{D}$ is monoidal by \cref{prop:Mndmonfib}. 
Under the above assumptions, it acts on the opfibration $\Mnd(\dc{D})^\op\to\dc{D}^\op_0$ according to \cref{Trepresentation} via the opfibred 1-cell
\begin{displaymath}
\begin{tikzcd}
\Cmd(\dc{D})\times\Mnd(\dc{D})^\op\ar[r,"H^\op","\cref{eq:MonHdouble}"']\ar[d] & 
\Mnd(\dc{D})^\op\ar[d] \\
\dc{D}_0\times\dc{D}^\op_0\ar[r,"H^\op"'] & \dc{D}^\op_0
\end{tikzcd}
\end{displaymath}
since both functors labelled by $H^\op$ are actions -- as opposites of actions -- and their structure 
isomorphisms are compatible to one another as discussed in \cref{rmk:fixedaction}.
Finally, $(H^\op,H^\op)$ has an ordinary right parameterized adjoint \cref{eq:parameterizedCat2} as reasoned in the above part of proof, namely
\begin{displaymath}
 \begin{tikzcd}
\Mnd(\dc{D})^\op\times\Mnd(\dc{D})\ar[d]\ar[r,"P","\cref{eq:P}"'] & \Cmd(\dc{D})\ar[d] \\
\dc{D}_0^\op\times\dc{D}_0\ar[r,"H"] & \dc{D}_0
\end{tikzcd} 
\end{displaymath}
which was obtained via \cref{thm:totaladjointthm} that in fact ensures compatibility of (co)units.
\end{proof}

\begin{rmk}
Although at first sight, the braiding of $\dc{D}$ is not needed for the second part of the above proof (also according to \cref{rmk:braidingnotneeded}), it is actually needed for the fact that $H^\op$ has a parameterized adjoint $H$ for the ordinary braided monoidal closed category $\dc{D}_0$.
\end{rmk}

The corollary below recovers \cite[Theorems 4.37 \& 4.38]{VCocats} in a straightforward way, without the extra verifications needed therein.

\begin{cor}\label{cor:VCatenrVCocat}
Suppose that $\ca{V}$ is a symmetric monoidal closed category which is locally presentable. The category $\ca{V}\mi\Cat$ is tensored and cotensored enriched in the braided monoidal closed $\VCocat$, and the fibration $\ca{V}\mi\Cat\to\Set$ is enriched in the monoidal opfibration $\VCocat\to\Set$.
\end{cor}

\begin{proof}
Under the above assumptions, the double category $\VMMat$ of $\ca{V}$-matrices is fibrant by \cref{ex:VMMatfibrant}, braided (in fact symmetric) monoidal by \cref{ex:VMMatmonoidal}, monoidal closed by \cref{ex:VMMatclosed} and locally presentable by \cref{prop:VMatlp}. Hence all clauses of \cref{thm:big1} hold, and the results follow. 
\end{proof}

\begin{rmk}
A toy example where the conditions of \ref{thm:big1} are also satisfied is given by the double category $\dc{R}\nc{el}(\nc{Set})$. Recall  by \ref{ex:monadsRel} that $\Mnd(\dc{R}\nc{el}(\nc{Set}))$ is the category $\nc{Preord}$ of preordered sets and monotone functions. On the other hand, objects of $\Cmd(\dc{R}\nc{el}(\nc{Set}))$ are sets with a given subset, while morphisms are functions which map the prescribed subset of their domain inside the corresponding subset of their codomain. Now the enrichment of monads in comonads says something very easy and natural.

Given two monads $A\colon X\bular X$ and $A'\colon Y\bular Y$, we have that $\nc{Preord}(A,A')$ is the set of monotone functions $X\to Y$. This can canonically be viewed as a subset of $Y^X$, i.e. as a comonad $Y^X\bular Y^X$, which is precisely what transpires if one follows the proof of \ref{thm:big1} above. Similarly, given a third monad $A''\colon Z\bular Z$, the fact that the composition morphism $\nc{Preord}(A,A')\times\nc{Preord}(A',A'')\to\nc{Preord}(A,A'')$ is a morphism in $\Cmd(\dc{R}\nc{el}(\nc{Set}))$ is simply the statement that composing two monotone functions $X\to Y$ and $Y\to Z$ yields a monotone function $X\to Z$.
\end{rmk}

\begin{rmk}\label{rmk:thisrmk}
Following \cref{rem:oneobjectcase,rem:duoidal}, the (one-object,one-vertical map) case of \cref{thm:big1} states that for a symmetric monoidal category $\ca{V}$ which is monoidal closed and 
locally presentable (since by \cref{def:locclosed,defi:locallypresentabledoublecat} these structures are directly inherited by the ordinary category of arrows which here becomes $\ca{V}$),
the category of monoids is tensored and cotensored enriched in the category of comonoids, namely \cref{thm:MonenrichedComon}.

On the other hand, the well-known fact that an one-object $\ca{V}$-category `is' precisely a monoid in $\ca{V}$ 
highlights a slightly different aspect of these one- versus many-object cases we have so far discussed: in the context of the double
category $\VMMat$, the tensored and cotensored enrichment of specifically $\VCat=\Mnd(\VMMat)$ in the symmetric monoidal 
$\VCocat=\Cmd(\VMMat)$ of \cref{cor:VCatenrVCocat} is indeed a many-object generalization of \cref{thm:MonenrichedComon} for a fixed monoidal category $\ca{V}$.

This discussion also propagates in the context of modules, as explained in the analogous \cref{rmk:thatrmk}.
\end{rmk}

\section{Modules and comodules in double categories}\label{sec:modscomods}

This chapter begins, in \cref{sec:ModsComodsdoublecats}, by setting up the general double categorical framework for modules and comodules, determining 
their key categorical properties and obtaining results that shall be of use in what follows. The setting of $\ca{V}$-matrices continues to be the leading example.
Some important results include \cref{prop:Mod(D)monadicComod(D)comonadic} which renders both categories of (co)modules (co)monadic over specific pullback categories, \cref{prop:ModfibredoverMnd,prop:sourcetargetbifib} where (op)fibration structures of (co)modules over (co)monads and objects are established in the fibrant setting, 
and \cref{prop:ModMndmonoidal,prop:ModD0monfib} verifying these (op)fibrations are monoidal.

In \cref{sec:enrichedfibdouble}, we turn to more specific structures that induce an enriched fibration of modules over monads in comodules over comonads, in a general double categorical setting under certain suitable conditions. The central \cref{thm:big2} which establishes the enrichment of modules in comodules is what we could refer to as `Sweedler theory for (co)modules' in double categories,
providing a vast generalization of \cref{thm:ModenrichedComod,thm:ModenrichedComodfib} for monoidal categories. \cref{cor:bigthm2} and \cref{rmk:VMod} discuss the examples of enriched modules and comodules, in the context of enriched matrices.

\subsection{Basic constructions}\label{sec:ModsComodsdoublecats}

Analogously to monads in double categories (\cref{Monadindoublecat}), modules in double categories as objects coincide with those in their horizontal 
bicategory (\cref{def:modulebicat}), whereas the morphisms between them are more general than their bicategorical counterpart. 

\begin{defi}\label{def:leftmodulesdouble}
 A \emph{left} $A$-\emph{module} for a monad $A\colon X\bular X$ in a double category $\dc{D}$ is a horizontal 1-cell $M\colon U\bular X$ equipped 
with a 2-morphism
 \begin{displaymath}
  \begin{tikzcd}
U\ar[r,bul,"M"]\ar[d,equal]\ar[drr,phantom,"\Two\lambda"] & X\ar[r,bul,"A"] & X\ar[d,equal] \\
U\ar[rr,bul,"M"'] && X
  \end{tikzcd}
 \end{displaymath}
called the $A$-\emph{action}, satisfying the following two axioms
\begin{equation}\label{eq:moduleaxioms}
 \begin{tikzcd}[column sep=.3in]
U\ar[r,bul,"M"]\ar[d,equal] & X\ar[r,bul,"A"]\ar[d,equal]\ar[drr,phantom,"\Two\mu"] & X\ar[r,bul,"A"] & X\ar[d,equal] \\
U \ar[d,equal]\ar[r,bul,"M"]\ar[drrr,phantom,"\Two\lambda"] & X\ar[rr,bul,"A"] && X\ar[d,equal] \\
U\ar[rrr,bul,"M"'] &&& X
 \end{tikzcd}=
 \begin{tikzcd}[column sep=.3in]
U\ar[r,bul,"M"]\ar[d,equal]\ar[drr,phantom,"\Two \lambda"] & X\ar[r,bul,"A"] & X\ar[r,bul,"A"]\ar[d,equal] & X\ar[d,equal] \\
U\ar[rr,bul,"M"]\ar[d,equal]\ar[drrr,phantom,"\Two\lambda"] && X\ar[r,bul,"A"] & X\ar[d,equal] \\
U\ar[rrr,bul,"M"'] &&& X,
 \end{tikzcd}\qquad
 \begin{tikzcd}[column sep=.3in]
U\ar[r,bul,"M"]\ar[d,equal] & X\ar[r,bul,"1_X"]\ar[d,equal]\ar[dr,phantom,"\Two\eta"] & X\ar[d,equal] \\
U\ar[r,bul,"M"]\ar[d,equal]\ar[drr,phantom,"\Two\lambda"] & X\ar[r,bul,"A"] & X\ar[d,equal] \\
U\ar[rr,bul,"M"'] && X
 \end{tikzcd}=\ell_M
\end{equation}
We usually denote a left $A$-module as $M_A$. A \emph{module morphism} $M_A\to N_B$ consists of a monad morphism $^f\alpha^f\colon A\Rightarrow B$ along 
with a 2-morphism (with specified right boundary)
 \begin{displaymath}
  \begin{tikzcd}
 U\ar[r,bul,"M"]\ar[dr,phantom,"\Two\zeta"]\ar[d,"g"'] & X\ar[d,"f"] \\
 T\ar[r,bul,"N"'] & Y
  \end{tikzcd}
 \end{displaymath}
 that satisfies the action compatibility
\begin{equation}\label{eq:modulemapaxiom}
\begin{tikzcd}
 U\ar[r,bul,"M"]\ar[d,equal]\ar[drr,phantom,"\Two\lambda"] & X\ar[r,bul,"A"] & X\ar[d,equal] \\
 U\ar[rr,bul,"M"]\ar[drr,phantom,"\Two\zeta"]\ar[d,"g"'] && X\ar[d,"f"] \\
 T\ar[rr,bul,"N"'] && Y
\end{tikzcd}=
\begin{tikzcd}
 U\ar[r,bul,"M"]\ar[d,"g"']\ar[dr,phantom,"\Two\zeta"] & X\ar[d,"f"]\ar[dr,phantom,"\Two\alpha"]\ar[r,bul,"A"] & X\ar[d,"f"] \\
 T\ar[r,bul,"N"]\ar[d,equal]\ar[drr,phantom,"\Two\lambda"] & Y\ar[r,bul,"B"] & Y\ar[d,equal] \\
 T\ar[rr,bul,"N"'] && Y
\end{tikzcd}
\end{equation}
These form the \emph{global category of (left) modules}, denoted by $\Mod(\dc{D})$ for any double category $\dc{D}$.
\end{defi}

The terminology `left module' comes of course from writting the action as $\lambda\colon A\odot M\Rightarrow M$. Notice that \emph{bimodules} in the 
setting of fibrant double categories appeared in \cite[\S~11]{Framedbicats}, where it was shown that they are the horizontal 1-cells of a fibrant 
double category of monads (therein called monoids), under certain assumptions on $\dc{D}$. In our context, we work with one-sided modules as defined above, in a general (and not 
necessarily fibrant) double category.

We can consider a subcategory $_A\Mod(\dc{D})$ of left $A$-modules for a fixed monad $A$, where morphisms are of the form $^g\zeta^\id$ and have as 
underlying monad map the identity. 
Another subcategory of interest is $^U\Mod(\dc{D})$ of modules of arbitrary monads, but with fixed domain $U$, and the corresponding module morphisms 
are now of the form $^\id\zeta^f$. 
Finally, there is a subcategory $^U_A\Mod(\dc{D})$ with objects all $A$-modules for a monad $A$ and with fixed source 
$U$, and morphisms globular 2-cells. This category coincides with the category of modules in the horizontal bicategory $\ca{H}(\dc{D})$ as seen 
in \cref{def:modulebicat}, which by \cref{rem:modulesaremonadic} is the category of Eilenberg-Moore algebras for the ordinary 
`post-composition with $A$' monad on the hom-category $\ca{H}(\dc{D})(U,X)={}^U\dc{D}_1^X$.

Dually, we have the notion of a (left) comodule over a comonad in a double category.
\begin{defi}
A \emph{left} $C$-\emph{comodule} for a comonad $C\colon Z\bular Z$ in a double category $\dc{D}$ is a horizontal 1-cell $K\colon V\bular Z$ equipped 
with a globular 2-morphism
\begin{displaymath}
		\begin{tikzcd}
			V\ar[rr,bul,"K"]\ar[d,equal]\ar[drr,phantom,"\Two\gamma"] && Z\ar[d,equal] \\
			V\ar[r,bul,"K"'] & Z\ar[r,bul,"C"'] & Z
		\end{tikzcd}
\end{displaymath}
satisfying the following two axioms
\begin{displaymath}
	\begin{tikzcd}[column sep=.3in]
		V\ar[rrr,bul,"K"]\ar[drrr,phantom,"\Two\gamma"]\ar[d,equal] &&& Z\ar[d,equal] \\			 
		V\ar[d,equal]\ar[r,bul,"K"] & Z\ar[rr,bul,"C"]\ar[drr,phantom,"\Two\delta"]\ar[d,equal] && Z\ar[d,equal] \\
			V\ar[r,bul,"K"']& Z\ar[r,bul,"C"'] & Z\ar[r,bul,"C"']& Z
	\end{tikzcd}=
	\begin{tikzcd}[column sep=.3in]
		V\ar[rrr,bul,"K"]\ar[drrr,phantom,"\Two\gamma"]\ar[d,equal] &&& Z\ar[d,equal] \\
		V\ar[d,equal]\ar[rr,bul,"K"]\ar[drr,phantom,"\Two\gamma"] && Z\ar[r,bul,"C"]\ar[d,equal] & Z\ar[d,equal]\\
		V\ar[r,bul,"K"']& Z\ar[r,bul,"C"'] & Z\ar[r,bul,"C"']& Z,
	\end{tikzcd}\qquad
	\begin{tikzcd}[column sep=.3in]
		V\ar[rr,bul,"K"]\ar[d,equal]\ar[drr,phantom,"\Two\gamma"] && Z\ar[d,equal] \\
		V\ar[r,bul,"K"]\ar[d,equal] & Z\ar[r,bul,"C"]\ar[d,equal]\ar[dr,phantom,"\Two\epsilon"] & Z\ar[d,equal] \\
		V\ar[r,bul,"K"'] & Z\ar[r,bul,"1_{Z}"'] & Z
	\end{tikzcd}=\ell^{\mi1}_K
\end{displaymath}
We usually denote a left $C$-comodule $(K,\gamma)$ as $K_C$. A \emph{comodule morphism} $K_C\to L_D$ consists of a comonad morphism $^f\alpha^f\colon 
C\to D$ along with a 2-morphism
\begin{displaymath}
	\begin{tikzcd}
		V\ar[r,bul,"K"]\ar[dr,phantom,"\Two\phi"]\ar[d,"g"'] & Z\ar[d,"f"] \\
		S\ar[r,bul,"L"'] & W
	\end{tikzcd}
\end{displaymath}
that satisfies a coaction compatibility condition dual to \cref{eq:modulemapaxiom}.
We thus have the \emph{global category of (left) comodules}, denoted by $\Comod(\dc{D})$ for any double category $\dc{D}$.
\end{defi}

As for modules, by fixing the comonad or the domain of the comodules or both, we respectively obtain categories $_C\Comod(\dc{D})$, 
$^V\Comod(\dc{D})$ and $^V_C\Comod(\dc{D})$. The latter coincides with the category of comodules in the horizontal bicategory $\ca{H}(\dc{D})$, 
namely 
the category of Eilenberg-Moore coalgebras for the comonad `post-composition with $C$' on the hom-category $\ca{H}(\dc{D})(V,Z)={}^V\dc{D}_1^Z$.

\begin{rmk}\label{rmk:modoneobject}
Similarly to \cref{rem:oneobjectcase}, in the (one-object,one-vertical arrow) case of a double category, the concepts of a module and a comodule
reduce to ordinary (left) modules and comodules over some monoid and comonoid in a monoidal category. We will later see that appropriate structure (e.g. monoidal) of the double category
also passes onto the categories of modules and comodules, in an analogous way as for monads and comonads of \cref{sec:monadscomonadsdouble}, hence also reduce to respective well-known facts in monoidal categories of 
\cref{sec:monoidsmodules}.
\end{rmk}

\begin{ex}
A left module in $\dc{R}\nc{el}(\ca{C})$ for a regular category $\ca{C}$ over a 
monad $A\colon X\bular X$ (see \cref{ex:monadsRel}) is a relation $M\colon U\bular X$ such that 
$AM\subseteq M$. Denoting the preorder relation $A$ by $\leq_{X}$ instead, this 
would read in terms of elements as the implication $(z,x)\in M\;\wedge 
\; x\leq_{X}x'\implies (z,x')\in M$. i.e. $M$ is increasing on the 
right with respect to the preorder $\leq_{X}$.
	
Dually, a left comodule $K\colon V\bular Z $ over a comonad $C\colon Z\bular Z$ 
is a relation $V\bular Z$ whose image is contained in the subobject $C\subseteq 
Z$.
\end{ex}

\begin{ex}\label{ex:twoVMod}
Let us consider modules in the double category $\VMMat$ of \cref{ex:VMMat}, where a monad $A\colon X\tickar X$ is a $\ca{V}$-category 
(\cref{ex:VCatsaremonadsinVMat}). Then a left $A$-module $M\colon U\tickar X$ consists of a collection of objects $\{M(x,u)\}_{(x,u)\in X\times U}$ 
in $\ca{V}$ together with an action given by a family  
 \begin{displaymath}
 \lambda_{x,u}\colon\sum_{x'\in X}A(x,x')\otimes M(x',u)\to M(x,u)
 \end{displaymath}
 or equivalently maps $\lambda_{x,x',u}\colon A(x,x')\ot M(x',u)\to M(x,u)$ for all $x,x'\in X, u\in U$, 
satisfying the following axioms (where associativity is suppressed)\:
\begin{equation}\label{eq:modassun}
\begin{tikzcd}
A(x,x')\otimes A(x',x'')\otimes M(x'',u)\ar[r,"{\mu\otimes 1}"]\ar[d,"{1\ot\lambda_{x',x'',u}}"'] & 
A(x,x'')\otimes M(x'',u)\ar[d,"{\lambda_{x,x'',u}}"]\\
A(x,x')\otimes M(x',u)\ar[r,"{\lambda_{x,x',u}}"'] & M(x,u)
\end{tikzcd}\quad 
\begin{tikzcd}
A(x,x)\otimes M(x,u)\ar[r,"\lambda_{x,x,u}"] & M(x,u) \\
I\ot M(x)\ar[ur,"\cong"']\ar[u,"\eta\ot1"] &
\end{tikzcd}
\end{equation}
Notice that such an object looks like the classical notion of a $\ca{V}$-\emph{bimodule}, see e.g. \cite{Lawvereclosedcats}, however without a right 
action from another $\ca{V}$-category. Equivalently, under our running assumptions on $\ca{V}$, left $A$-modules with domain $U$ in $\VMMat$ can be 
viewed as ($A$,$\mathrm{disc}(U)$)-bimodules, where $\mathrm{disc}(U)$ is the discrete $\ca{V}$-category on the set $U$ with hom-objects 
$\mathrm{disc}(U)(u,u')=I$ if $u=u'$ or the initial object $0$ if $u\neq u'$.

A module morphism $M_A\to N_B$ between an $A$-module $M\colon U\tickar X$ and a 
$B$-module $N\colon T\tickar Y$ is a $\ca{V}$-functor $F\colon A\to B$ together 
with a function $g\colon U\to T$ and a family of morphisms 
\begin{equation}\label{eq:Modmaps}
\zeta_{x,u}\colon M(x,u)\to N(Fx,gu)_{(x,u)\in X\times U}
\end{equation}
rendering diagrams of the following 
form commutative
\begin{displaymath}
	\begin{tikzcd}[column sep=.6in]
		A(x,x')\otimes M(x',u)\ar[r,"\lambda_{x,x',u}"]\ar[d,"F_{x,x'}\otimes\zeta_{x',u}"'] & M(x,u)\ar[d,"\zeta_{x,u}"] \\
		B(F(x),F(x'))\otimes N(F(x'),g(u))\ar[r,"\lambda_{Fx,Fx',gu}"'] & N(F(x),g(u))
	\end{tikzcd}
\end{displaymath}
for all $x'\in X$, 
where the horizontal arrows are the components of the corresponding actions on $M$ and $N$.

If we fix the domain of the module as the singleton set $1=\{*\}$, then a left $A$-module in $\VMMat$ is a family $\{M(x)\}_{x\in X}$ with action 
morphisms $$\lambda_{x}\colon\sum\limits_{x'\in X}A(x,x')\otimes M(x')\to M(x)$$ making appropriate diagrams commute.
Furthermore, a module morphism $M_A\to N_B$ between a left $A$-module and a left $B$-module for a monad $B\colon Y\tickar Y$ with domains $1$ 
consists of a monad morphism $F\colon A\to B$ namely a $\ca{V}$-functor, and a family of arrows
\begin{displaymath}
	\gamma_{x}\colon M(x)\to N(F(x'))
\end{displaymath}
for all $x\in X$ making a respective diagram commute.
These objects are classically called $\ca{V}$-\emph{modules}: they are 
special `one-sided' cases of enriched bimodules where one of the 
$\ca{V}$-categories acting is the terminal $1$, and can be equivalently viewed as 
$\ca{V}$-copresheaves
$M\colon 
A\to\ca{V}$ when $\ca{V}$ is a monoidal closed category.

Both categories of left modules $\Mod(\VMMat)$ and fixed singleton 
domain modules $^{1}\Mod(\VMMat)$ in the double category of $\ca{V}$-matrices 
can be thought of as many-object versions of the global category $\Mod_{\ca{V}}$ 
of left $\ca{V}$-modules discussed in \cref{sec:monoidsmodules}, and clearly the 
reason for considering only one-sided modules in this work is that they set the 
generalized framework for modules for monoids, e.g. modules for rings in 
$\ca{V}=\mathsf{Ab}$.
We shall denote by $^1\Mod(\VMMat)=\VMod$ the category of $\ca{V}$-modules, and simply by $\Mod(\VMMat)$ the more general module category whose objects we call \emph{two-indexed} $\ca{V}$-\emph{modules}. 
\end{ex}

\begin{ex}\label{ex:twoVComod}
For comodules in the double category $\VMMat$, where a comonad $C\colon Z\tickar Z$ is a $\ca{V}$-cocategory 
(\cref{ex:VCatsaremonadsinVMat}), in the general case we have \emph{two-indexed} left $\ca{V}$-\emph{comodules} $K\colon V\tickar Z$ given by
objects $\{K(z,v)\}_{z\in Z, v\in V}$ in $\ca{V}$ equipped with a cocategory coaction given by
\begin{displaymath}
\gamma_{z,v}\colon K(z,v)\longrightarrow\sum_{z'\in Z}C(z,z')\ot K(z',v)
\end{displaymath}
satisfying coassociativity and counitality; since structure maps are now landing on sums, the axioms cannot be written precisely dual to \cref{eq:modassun}, rather
they become more involved as in
\begin{displaymath}
\begin{tikzcd}[column sep=.4in]
K(z,v)\ar[r,"\gamma_{z,v}"]\ar[d,"\gamma_{z,v}"'] & \sum\limits_{z'\in Z} C(z,z'){\ot} K(z',v)\ar[d,"\sum1\ot\gamma_{z',v}"] \\
\sum\limits_{z''\in Z}C(z,z''){\ot} K(z'',v)\ar[r,"\sum\delta\ot1"'] & \sum\limits_{z',z''\in Z} C(z,z'){\ot} C(z',z''){\ot} K(z'',v)
\end{tikzcd}
\begin{tikzcd}
K(z,v)\ar[r,"\gamma_{z,v}"]\ar[dr,"\cong"'] & \sum\limits_{z'\in Z} C(z,z'){\ot} K(z',v)\ar[d,"\sum\epsilon\ot 1"] \\
 & {\phantom{\sum\limits_{zZ}}}I{\ot} K(z,v)
\end{tikzcd}
\end{displaymath}
using that $\otimes$ preserves coproducts in both variables. With the respective notion of comodule morphism, we denote this category as $\Comod(\VMMat)$, whereas the special case where the domain is the singleton is denoted as $\VComod$.
\end{ex}

The following result, generalizing \cref{prop:Modmonadic}, shows that the global category of modules in any double category can naturally be expressed as the category of algebras over an 
appropriate pullback of the category of arrows and the category of monads.

\begin{prop}\label{prop:Mod(D)monadicComod(D)comonadic}
The functor $\Mod(\dc{D})\to\dc{D}_1\times_{\dc{D}_0}\Mnd(\dc{D})$ mapping $M_A$ to the pair $(M,A)$ is monadic. Dually, the functor 
$\Comod(\dc{D})\to\dc{D}_1\times_{\dc{D}_0}\Cmd(\dc{D})$ mapping $K_C$ to $(K,C)$ is comonadic.
\end{prop}
\begin{proof}
We first construct the left adjoint $F\colon\dc{D}_1\times_{\dc{D}_0}\Mnd(\dc{D})\to\Mod(\dc{D})$, where the pullback is over the target functor 
$\Gr{t}\colon\dc{D}_1\to\dc{D}_0$ and the source (or equivalently target) functor $\Mnd(\dc{D})\to\dc{D}_0$. Given any horizontal 1-cell $M\colon 
U\bular X$ and any monad $A\colon X\bular X$, we set $F(A,M)=A\odot M$ where $A\odot M\colon U\bular X$ becomes an $A$-module with the following action
\begin{displaymath}
\begin{tikzcd}		
	U\ar[d,equal]\ar[r,bul,"M"] & X\ar[d,equal]\ar[r,bul,"A"]\ar[drr,phantom,"\Two\mu"] & X\ar[r,bul,"A"] & X\ar[d,equal] \\
		U\ar[r,bul,"M"'] & X\ar[rr,bul,"A"'] && X
\end{tikzcd}
\end{displaymath}	
The module axioms reduce precisely to the monad axioms for $A$. Furthermore, if $(\phi,\alpha)$ is a morphism $(M,A)\to (N,B)$ in 
$\dc{D}_1\times_{\dc{D}_0}\Mnd(\dc{D})$, i.e. a 2-morphism $^f\phi^g\colon M\Rightarrow N$ and a monad morphism $^g\alpha^g\colon A\Rightarrow B$, 
then $\alpha\odot\phi\colon A\odot M\to B\odot N$ with underlying monad morphism $\alpha$ is a module homomorphism with respect to the actions 
defined above. This is simply the equality  
\begin{displaymath}
\begin{tikzcd}
	U\ar[d,equal]\ar[r,bul,"M"] & \ar[d,equal]\ar[r,bul,"A"]\ar[drr,phantom,"\Two\mu"] & X\ar[r,bul,"A"] & X\ar[d,equal] \\
	U\ar[d,"f"']\ar[r,bul,"M"']\ar[dr,phantom,"\Two\phi"] & X\ar[d,"g"]\ar[rr,bul,"A"']\ar[drr,phantom,"\Two\alpha"] && Y\ar[d,"g"] \\
	T\ar[r,bul,"N"'] & Y\ar[rr,bul,"B"'] && Y
\end{tikzcd}\stackrel{\cref{monadhom}}{=}
\begin{tikzcd}
	U\ar[d,"f"']\ar[r,bul,"M"]\ar[dr,phantom,"\Two\phi"] & X\ar[d,"g"]\ar[r,bul,"A"]\ar[dr,phantom,"\Two\alpha"] & 
X\ar[d,"g"]\ar[r,bul,"A"]\ar[dr,phantom,"\Two\alpha"] & X\ar[d,"g"] \\
	T\ar[d,equal]\ar[r,bul,"N"'] & Y\ar[d,equal]\ar[r,bul,"B"']\ar[drr,phantom,"\Two\mu"] & Y\ar[r,bul,"B"'] & Y\ar[d,equal] \\
	T\ar[r,bul,"N"'] & Y\ar[rr,bul,"B"'] && Y
\end{tikzcd}
\end{displaymath}
We next prove that there is a natural bijection between morphisms $F(M,A)\to N_B$ in $\Mod(\dc{D})$ and morphisms $(M,A)\to (N,B)$ in 
$\dc{D}_1\times_{\dc{D}_0}\Mnd(\dc{D})$. By first considering such a morphism $(\phi,\alpha)$ in the latter category, namely a pair of a 2-morphism 
$^f\phi^g\colon M\Rightarrow N$ and a monad morphism $^g\alpha^g\colon A\Rightarrow B$, we obtain a module morphism $\bar{\phi}\colon (A\odot M)_A\to 
N_{B}$ with underlying monad map $\alpha$, defined as 
\begin{displaymath}
\bar{\phi}:=
\begin{tikzcd}
	U\ar[d,"f"']\ar[r,bul,"M"]\ar[dr,phantom,"\Two\phi"] & X\ar[d,"g"]\ar[r,bul,"A"]\ar[dr,phantom,"\Two\alpha"] & X\ar[d,"g"] \\
	T\ar[d,equal]\ar[r,bul,"N"']\ar[drr,phantom,"\Two\lambda"] & Y\ar[r,bul,"B"'] & Y\ar[d,equal] \\
	T\ar[rr,bul,"N"'] && Y
\end{tikzcd}
\end{displaymath}
We can verify that this is indeed a module morphism namely it satisfies \cref{eq:modulemapaxiom}, using the fact that $\alpha$ is 
monad morphism \cref{monadhom} and $N$ is a $B$-module \cref{eq:moduleaxioms}.
In the converse direction, for any module morphism $\zeta\colon F(M,A){=}(A\odot M)_A\Rightarrow N_B$ with monad map $^g\alpha^g$, we have an induced 
morphism $(M,A)\to (N,B)\in\dc{D}_1\times_{\dc{D}_0}\Mnd(\dc{D})$ given by the same monad map and the following 2-morphism
\begin{displaymath}
\begin{tikzcd}
	U\ar[d,equal]\ar[r,bul,"M"] & X\ar[d,equal]\ar[r,bul,"1_X"]\ar[dr,phantom,"\Two\eta"] & X\ar[d,equal] \\
	U\ar[d,"f"']\ar[r,bul,"M"']\ar[drr,phantom,"\Two\zeta"] & X\ar[r,bul,"A"'] & X\ar[d,"g"] \\
	T\ar[rr,bul,"N"'] && Y
\end{tikzcd}
\end{displaymath}
The fact that these two processes are inverse to each other is contained in the following:
\begin{displaymath}
\scalebox{.7}
{\begin{tikzcd}[ampersand replacement=\&]
	U\ar[d,equal]\ar[r,bul,"M"] \& X\ar[d,equal]\ar[r,bul,"1_X"]\ar[dr,phantom,"\Two\eta_A"] \& X\ar[d,equal] \\
	U\ar[d,"f"']\ar[r,bul,"M"']\ar[dr,phantom,"\Two\phi"] \& X\ar[r,bul,"A"']\ar[d,"g"]\ar[dr,phantom,"\Two\alpha"] \& X\ar[d,"g"] \\
	T\ar[d,equal]\ar[r,bul,"N"']\ar[drr,phantom,"\Two\lambda_N"] \& Y\ar[r,bul,"B"'] \& Y\ar[d,equal] \\
	T\ar[rr,bul,"N"'] \&\& Y
\end{tikzcd}}{\scriptstyle\stackrel{\cref{monadhom}}{=}}
\scalebox{.7}
{\begin{tikzcd}[ampersand replacement=\&]
	U\ar[dd,"f"']\ar[r,bul,"M"]\ar[ddr,phantom,"\Two\phi"] \& X\ar[d,"g"]\ar[r,bul,"1_X"]\ar[dr,phantom,"\Two 1_g"] \& X\ar[d,"g"] \\
	\& Y\ar[d,equal]\ar[r,bul,"1_Y"']\ar[dr,phantom,"\Two\eta"]\& Y\ar[d,equal] \\
	T\ar[d,equal]\ar[r,bul,"N"']\ar[drr,phantom,"\Two\lambda"] \& Y\ar[r,bul,"B"'] \& Y\ar[d,equal] \\
	T\ar[rr,bul,"N"'] \&\& Y
\end{tikzcd}}{\scriptstyle\stackrel{\cref{eq:moduleaxioms}}{=}}{\scriptstyle\phi,}\quad
\scalebox{.7}
{\begin{tikzcd}[ampersand replacement=\&]
	U\ar[d,equal]\ar[r,bul,"M"] \& X\ar[d,equal]\ar[r,bul,"1_X"]\ar[dr,phantom,"\Two\eta"] \& 
X\ar[d,equal]\ar[r,bul,"A"]\ar[ddr,phantom,"\Two\alpha"] \& X\ar[dd,"g"] \\
	U\ar[d,"f"']\ar[r,bul,"M"']\ar[drr,phantom,"\Two\zeta"] \& X\ar[r,bul,"A"'] \& X\ar[d,"g"] \& \\
	T\ar[d,equal]\ar[rr,bul,"N"']\ar[rrrd,phantom,"\Two\lambda"] \&\& Y\ar[r,bul,"B"'] \& Y\ar[d,equal] \\
	T\ar[rrr,bul,"N"'] \&\&\& Y
\end{tikzcd}}{\scriptstyle\stackrel{\cref{eq:modulemapaxiom}}{=}}
\scalebox{.7}
{\begin{tikzcd}[ampersand replacement=\&]
	U\ar[d,equal]\ar[r,bul,"M"] \& X\ar[d,equal]\ar[r,bul,"1_X"]\ar[dr,phantom,"\Two\eta"] \& X\ar[d,equal]\ar[r,bul,"A"] \& X\ar[d,equal] \\
	U\ar[d,equal]\ar[r,bul,"M"'] \& X\ar[d,equal]\ar[r,bul,"A"']\ar[drr,phantom,"\Two\mu"] \& X\ar[r,bul,"A"'] \& X\ar[d,equal] \\
	U\ar[d,"f"']\ar[r,bul,"M"]\ar[drrr,phantom,"\Two\zeta"] \& X\ar[rr,bul,"A"] \&\& X\ar[d,"g"] \\
	T\ar[rrr,bul,"N"'] \&\&\& Y
\end{tikzcd}}{\scriptstyle\stackrel{\cref{monadhom}}{=}}{\scriptstyle\zeta}	
\end{displaymath}
Note that by the above, the unit of the adjunction has components $\eta_{(M,A)}\colon (M,A)\to (A\odot M, A)$ given by pairs 
$(\eta_A\odot\rm{id}_M,\rm{id}_A)$.

Now consider an algebra for the monad induced by the above adjunction, namely a pair $(M,A)\in\dc{D}_1\times_{\dc{D}_0}\Mnd(\dc{D})$ equipped with a 
pair of a 2-morphism and a monad morphism $(\lambda\colon A\odot M\Rightarrow M,\beta\colon A\Rightarrow A)$. The unit axiom for the algebra shows 
in particular that $\beta=\rm{id}_A$. Then the remainder of the two algebra axioms translate precisely to the fact that $\lambda\colon A\odot 
M\Rightarrow M$ is an action making $M$ into an $A$-module. Similarly, the property of $(\phi,\alpha)\colon (M,A)\to (N,B)$ being an algebra morphism 
is then precisely the statement that this pair defines a morphism in $\Mod(\dc{D})$. 
\end{proof}

By running the same proof as above, but now with either a fixed monad $A\colon X\bular 
X$ or a fixed domain $U$ for our modules, we obtain the following two results respectively.

\begin{prop}\label{_AModmonadic_CComodcomonadic}
Given any monad $A\colon X\bular X$ in a double category $\dc{D}$, the forgetful 
functor $_{A}\Mod(\dc{D})\to \dc{D}_1^{X}$ is monadic. Dually, for every comonad $C\colon Z\bular Z$ the forgetful functor $_{C}\Comod(\dc{D})\to\dc{D}^{Z}_{1}$ is comonadic.
\end{prop}
\begin{proof}
	First of all, observe that when 
	$\Mod(\dc{D})\to\dc{D}_1\times_{\dc{D}_0}\Mnd(\dc{D})$ 
	restricts to $_{A}\Mod(\dc{D})$ its image is $\dc{D}_{1}^{X}$. Furthermore, the natural bijection giving the adjunction in the previous proposition clearly restricts to a bijection between morphisms in $_{A}\Mod(\dc{D})$ and morphisms in $\dc{D}_1^{X}$. Thus, $_{A}\Mod(\dc{D})\to \dc{D}_1^{X}$ has a left adjoint which maps any $M\colon U\bular X$ to $A\odot M$ with the action induced by the multiplication of $A$. The remainder of the previous proof now goes through unchanged.
	
	We observe here, for future reference, that by the above arguments the monad on $\dc{D}_1^{X}$ for which $_{A}\Mod(\dc{D})$ is the corresponding Eilenberg-Moore category is precisely the functor $T\colon\dc{D}_1^{X}\to\dc{D}_1^{X}$ with $T(M)=A\odot M$.
\end{proof}

\begin{prop}
	Given any object $U$ in a double category $\dc{D}$, the forgetful functor $^{U}\Mod(\dc{D})\to{}^U\dc{D}_{1}\times_{\dc{D}_{0}}\Mnd(\dc{D})$ is 
monadic. Dually, $^{V}\Comod(\dc{D})\to{}^{V}\dc{D}_{1}\times_{\dc{D}_{0}}\Cmd(\dc{D})$ is comonadic.
\end{prop}

The results so far concern modules in general double categories. Moving to the 
setting of fibrant double categories, we can establish certain natural fibration 
structures that involve categories of fixed monad or domain modules as follows.

\begin{prop}\label{prop:ModfibredoverMnd}
Suppose $\dc{D}$ is a fibrant double category. The forgetful functor $\Mod(\dc{D})\to\Mnd(\dc{D})$ that maps an $A$-module $M$ to its monad 
$A$ is a fibration; dually $\Comod(\dc{D})\to\Cmd(\dc{D})$ is an opfibration.
\end{prop}

\begin{proof}
In compact terms, $\Mod(\dc{D})\to\Mnd(\dc{D})$ is the Grothendieck fibration that corresponds to the pseudofunctor 
\begin{equation}\label{eq:_Mod}
	\begin{tikzcd}[row sep=.02in,/tikz/column 1/.append style={anchor=base east},/tikz/column 2/.append style={anchor=base west}]
		\Mnd(\dc{D})^\op\ar[r] & \Cat \\
		A\ar[r,mapsto]\ar[dd,"^f\alpha^f"'] & _A\Mod(\dc{D}) \\
		\hole \\
		B\ar[r,mapsto] & _B\Mod(\dc{D})\ar[uu,"\wc{f}\odot\mi"']
	\end{tikzcd}
\end{equation}
Indeed, given a $B$-module $N\colon T\bular Y$, we make 
$\begin{tikzcd}[sep=small,cramped]T\ar[r,bul,"N"] 
& Y\ar[r,bul,"\wc{f}"] & X \end{tikzcd}$ into an $A$-module via 
\begin{displaymath}
 \begin{tikzcd}
T\ar[r,bul,"N"]\ar[d,equal] & Y\ar[r,bul,"\wc{f}"]\ar[d,equal]\ar[drr,phantom,"\Two\wc{\alpha}"] & X\ar[r,bul,"A"] & X\ar[d,equal] \\
T\ar[r,bul,"N"]\ar[d,equal]\ar[drr,phantom,"\Two\lambda"] & Y\ar[r,bul,"B"] & Y\ar[d,equal]\ar[r,bul,"\wc{f}"] & X\ar[d,equal] \\
T\ar[rr,bul,"N"'] && Y\ar[r,bul,"\wc{f}"'] & X
 \end{tikzcd}
\end{displaymath}
where $\wc{\alpha}$ corresponds to the monad map ${}^f\alpha^f$ as in 
\cref{eq:glob2map}. Axioms \cref{eq:moduleaxioms} are satisfied due to $\alpha$ being a monad map \cref{monadhom} and $N$ being a $B$-module.
This mapping forms a pseudofunctor because for composable 
vertical arrows, there is an isomorphism $\wc{gf}\cong\wc{f}\odot\wc{g}$ by 
\cref{lem:fibrantproperties}.

Equivalently, it can be verified that the cartesian liftings of
$\Mod(\dc{D})\to\Mnd(\dc{D})$ are given by
\begin{equation}\label{eq:Modcartlifts}
 \begin{tikzcd}[column sep=.3in]
\wc{f}\odot N  \ar[rr,"{\Cart(\alpha,N)}"]\ar[d,-,dotted] && N\ar[d,-,dotted] & \textrm{in }\Mod(\dc{D}) \\
A\ar[rr,"{^f\alpha^f}"] && B & \textrm{in }\Mnd(\dc{D})
 \end{tikzcd}\quad\textrm{ where }\Cart(\alpha,N)=
 \begin{tikzcd}
  T\ar[r,bul,"N"]\ar[d,equal]\ar[dr,phantom,"\scriptstyle\id"] & Y\ar[d,equal]\ar[r,bul,"\wc{f}"]\ar[dr,phantom,"\Two q_1"] & X\ar[d,"f"] \\
 T\ar[r,bul,"N"'] & Y\ar[r,bul,"1_Y"'] & Y
 \end{tikzcd}
\end{equation}

Similarly, the fibres of the opfibration $\Comod(\dc{D})\to\Cmd(\dc{D})$ are $_{C}\Comod(\dc{D})$ with reindexing functors $\wh{f}\odot-$, while the 
cocartesian liftings are
\begin{equation}\label{eq:Comodcartlifts}
 \begin{tikzcd}[column sep=.3in]
\K\ar[rr,"{\Cocart(\alpha,\K)}"]\ar[d,-,dotted] && \wh{f}\odot \K\ar[d,-,dotted] & \textrm{in }\Comod(\dc{D}) \\
C\ar[rr,"{^f\alpha^f}"] && D & \textrm{in }\Cmd(\dc{D})
 \end{tikzcd}\quad\textrm{ where }\Cocart(\alpha,\K)=
 \begin{tikzcd}
  V\ar[r,bul,"\K"]\ar[d,equal]\ar[dr,phantom,"\scriptstyle\id"] & Z\ar[d,equal]\ar[r,bul,"1_Z"]\ar[dr,phantom,"\Two p_2"] & Z\ar[d,"f"] \\
 V\ar[r,bul,"\K"'] & Z\ar[r,bul,"\wh{f}"'] & V
 \end{tikzcd}
\end{equation}
\end{proof}

In fact, both structures above have sub(op)fibrations involving the fixed-domain (co)module categories in the following sense.

\begin{prop}\label{prop:ZModfibredoverMnd}
Suppose $\dc{D}$ is a fibrant double category. For any 0-cell $U$ in $\dc{D}$, 
$^U\Mod(\dc{D})\to\Mnd(\dc{D})$ is a fibration and dually $^V\Comod(\dc{D})\to\Cmd(\dc{D})$ is an opfibration.
\end{prop}

\begin{proof}
This follows from the fact that the category 
$^U\Mod(\dc{D})\subseteq\Mod(\dc{D})$ 
is closed under cartesian liftings over $\Mnd(\dc{D})$.
Indeed, as is clear from the above proof, the reindexing functors 
$\wc{f}\odot\mi$ restrict between the subcategories
$_A^U\Mod(\dc{D})\subseteq {}_A\Mod(\dc{D})$ and $_B^U\Mod(\dc{D})\subseteq {}_B\Mod(\dc{D})$ of the fibres with fixed domain 
$U$.
\end{proof}

\cref{prop:ModfibredoverMnd,prop:ZModfibredoverMnd} establish $\Mod(\dc{D})$ and $^U\Mod(\dc{D})$ as fibred over $\Mnd(\dc{D})$ mapping a module to 
its monad, and dually for the opfibrations mapping a comodule to its comonad. However, we can also view $\Mod(\dc{D})$ and $\Comod(\dc{D})$ as fibred over 
$\dc{D}_0$, and in fact in more than one ways, as can be seen below.

\begin{prop}\label{prop:sourcetargetbifib}
	Suppose $\dc{D}$ is a fibrant double category. The source and target functors 
	$\Mod(\dc{D})\rightrightarrows\dc{D}_0$ are a bifibration and a fibration respectively, and the source 
	and target functors $\Comod(\dc{D})\rightrightarrows\dc{D}_0$ are a bifibration and an opfibration respectively.
\end{prop}

\begin{proof}
It can be easily verified that in a fibrant double category, appropriately restricting the source bifibration \cref{eq:stbifibrations} induces a sub-bifibration $\Mod(\dc{D})\subseteq\dc{D}_1\to\dc{D}_0$. In more detail, given a $B$-module $N\colon T\bular Y$ in the fibre $^T\Mod(\dc{D})$ and 
$f\colon U\to T$, $N\odot\wh{f}$ is still a $B$-module via the action
\begin{displaymath}
	\begin{tikzcd}
		U\ar[r,bul,"\wh{f}"]\ar[d,equal] & 
		T\ar[drr,phantom,"{\Two\lambda}"]\ar[r,bul,"N"]\ar[d,equal] & Y\ar[r,bul,"B"] & 
		Y\ar[d,equal] \\
		U\ar[r,bul,"\wh{f}"'] & T\ar[rr,bul,"N"'] && Y
	\end{tikzcd}
\end{displaymath}
Moreover, the adjunction $(\mi\odot\wc{f})\dashv(\mi\odot\wh{f})$ between ${}^U\dc{D}_1$ and ${}^T\dc{D}_1$ from \cref{globalvslocal} restricts to 
	$\begin{tikzcd}{}^U\Mod(\dc{D})\ar[r,shift 
		left=2,"\mi\odot\wc{f}"]\ar[r,phantom,"\bot"] &
		{}^T\Mod(\dc{D})\ar[l,shift left=2,"\mi\odot\wh{f}"]\end{tikzcd}$ thus $\Gr{s}\colon\Mod(\dc{D})\to\dc{D}_0$ is a bifibration.
In fact, by construction this further reduces to a sub-bifibration $\Gr{s}\colon {}_B\Mod(\dc{D})\to\dc{D}_0$ with fibres $^T_B\Mod(\dc{D})$.
	
Next, the restricted target fibration \cref{eq:stbifibrations} also induces a sub-fibration of modules over the category of objects $\Gr{t}\colon\Mod(\dc{D})\to\dc{D}_0$, which is however more involved in its construction -- and also is not a bifibration anymore. A natural way to see it is as a composite of fibrations (\cite[Prop.~8.1.12]{Handbook2})
\begin{displaymath}
	\begin{tikzcd}[column sep=.5in]
		\Mod(\dc{D})\ar[rr,bend left,"\mathfrak t"]\ar[r,"{\ref{prop:ModfibredoverMnd}}"] & \Mnd(\dc{D})\ar[r,"{\ref{prop:MonComonfibred}}"] & 
\dc{D}_0 
	\end{tikzcd}
\end{displaymath}
mapping a $B$-module $N\colon T\bular Y$ first to its monad $B\colon Y\bular Y$ and then to its carrier object $Y$, namely to $N$'s 
codomain. The reindexing functors are of the form $\wc{g}\odot\mi$ for $g\colon X\to Y$, and notice how the lifting $\wc{g}\odot N$ is a $(\wc{g}\odot B\odot\wh{g})$-module via the action
\begin{displaymath}
	\begin{tikzcd}
		T\ar[r,bul,"N"]\ar[d,equal] & 
		Y\ar[r,bul,"\wc{g}"]\ar[d,equal]\ar[drr,phantom,"\Two\varepsilon"] & 
		X\ar[r,bul,"\wh{g}"] & Y\ar[r,bul,"B"]\ar[d,equal] & 
		Y\ar[r,bul,"\wc{g}"]\ar[d,equal] & X\ar[d,equal] \\
		T\ar[r,bul,"N"]\ar[drrrr,phantom,"\Two\lambda"]\ar[d,equal] & Y\ar[rr,bul,"1_Y"] 
		&& Y\ar[r,bul,"B"] & Y\ar[r,bul,"\wc{g}"]\ar[d,equal] & X\ar[d,equal] \\
		T\ar[rrrr,bul,"N"'] &&&& Y\ar[r,bul,"\wc{g}"'] & X
	\end{tikzcd}
\end{displaymath}
Analogous results hold for $\Cmd(\dc{D})$ over $\dc{D}_0$. More specifically (and since both comodules and modules currently have a left action) $\Gr{s}\colon\Cmd(\dc{D})\to\dc{D}_0$ is a sub-bifibration of the source functor, and $\Gr{t}\colon\Cmd(\dc{D})\to\dc{D}_0$ is a sub-opfibration of the target functor, formed as a composite of the opfibrations of \cref{prop:ModfibredoverMnd,prop:MonComonfibred}.
\end{proof}

\begin{rmk}
Notice how in the above bifibration structures, opposite to those of \cref{prop:ZModfibredoverMnd,prop:ModfibredoverMnd}, fixed-domain (co)module categories are the fibers and fixed-(co)monad (co)module categories are the total categories. Pictorially, several of the above structures appear together as follows, where we suppress $\dc{D}$ from the respective categories:
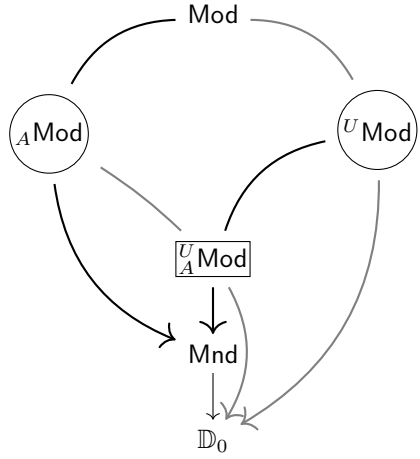
\begin{displaymath}
\begin{tikzcd}
& \Mod\ar[dl,bend right,no head,thick]\ar[dr,bend left,no head,gray,thick] & \\
\circled{{}_A\Mod}\ar[ddr,bend right,thick]\ar[dr,bend left=5,no head,gray,thick] && 
\circled{{}^U\Mod}\ar[dddl,bend left,gray,thick]\ar[dl,bend right,no head,thick] \\
& \tboxed{{}_A^U\Mod}\ar[d,thick]\ar[dd,bend left,gray,thick] & \\
& \Mnd\ar[d] & \\
& \dc{D}_0 &
\end{tikzcd}
\end{displaymath}
The gray arrow on the right is the source bifibration of \cref{prop:sourcetargetbifib} with fibers ${}^U\Mod$, which also restricts to a bifibration emanating from ${}_A\Mod$ with fibers ${}_A^U\Mod$. The black arrow on the left is the underlying monad fibration of \cref{prop:ModfibredoverMnd} with fibers ${}_A\Mod$, which also restricts to a fibration emanating from ${}^U\Mod$ with fibers ${}_A^U\Mod$. Finally, composing the underlying monad fibration with the bottom fibration of \cref{prop:MonComonfibred} we obtain the target fibration of \cref{prop:sourcetargetbifib}.
\end{rmk}

\begin{ex}
In the case of the double category $\dc{D}=\VMMat$ considered in \cref{ex:twoVMod}, \cref{prop:ModfibredoverMnd} says that the functor $\Mod(\VMMat)\to\VCat$ is a fibration. In more detail, 
given $\ca{V}$-categories $A\colon X\tickar X$ and $B\colon Y\tickar Y$, a 
$\ca{V}$-functor $^{f}\alpha^{f}\colon A\to B$ and a $B$-module $N\colon T\tickar Y$, the cartesian lifting of $\alpha$ to $N_B$ has domain the 
$A$-module $f^{*}\odot N$ given by $(f^{*}\odot N)(x,t)\cong N(f(x),t)$ for all 
$(x,t)\in X\times T$, and whose action is determined by the compositions
$$A(x,x')\otimes N(f(x'),t)\xrightarrow{\alpha_{x,x'}\otimes 
		1}B(f(x),f(x'))\otimes 
	N(f(x'),t)\xrightarrow{\lambda}N(f(x),t).$$ 
The morphism $\Cart(\alpha,N)\colon f^{*}\odot N\to N$ as in \cref{eq:Modcartlifts} in $\Mod(\VMMat)$ consists of the $\ca{V}$-functor $^f\alpha^f$, the identity function on $T$ and the family of identity morphisms $(1_{x,t}\colon N(f(x),t)\to N(f(x),t))_{(x,t)\in X\times T}$ as described in \cref{eq:Modmaps}.
Moreover, \cref{prop:ZModfibredoverMnd} for $U=1$ the singleton says that the functor $\VMod\to\VCat$ is also a fibration\footnote{\label{note8}This was established by hand in \cite[Prop.~7.6.9\&7.6.10]{PhDChristina}.}. 
	
Dually, for the category of $\ca{V}$-comodules discussed in \cref{ex:twoVComod}, the functors $\Comod(\VMMat)\to\VCocat$ and $\VComod\to\VCocat$ are opfibrations. Given $\ca{V}$-cocategories $C\colon Z\tickar Z$ and $D\colon W\tickar W$, a $\ca{V}$-cofunctor $^{f}\alpha^{f}\colon C\to D$ and a $C$-comodule $K\colon V\tickar Z$, the cocartesian lifting of $\alpha$ to $K_C$ has codomain the $D$-comodule $f_{*}\odot K$ defined by $(f_{*}\odot K)(w,v)=\sum\limits_{f(z)=w}K(z,v)$ and whose coaction is given by the following composition
\begin{center}
	\begin{tikzcd}
		\sum\limits_{f(z)=w}K(z,v)\ar[r,"\sum\gamma_{K}"] & \sum\limits_{f(z)=w}\left(\sum\limits_{z'}C(z,z')\otimes 
		K(z',v)\right)\ar[r,"\sum(\sum\alpha\otimes 1)"] & \sum\limits_{f(z)=w}\left(\sum\limits_{z'}D(f(z),f(z'))\otimes K(z',v)\right)\ar[d,hook] \\
		&& \sum\limits_{w'}\left( 
		D(w,w')\otimes\left(\sum\limits_{f(z')=w'}K(z',v)\right)\right)
	\end{tikzcd}
\end{center}
The morphism $\Cocart(\alpha,K)\colon K\to f_{*}\odot K$ as in \cref{eq:Comodcartlifts} in $\VComod$ is given by the $\ca{V}$-cofunctor ${}^f\alpha^f$, the identity function on $V$ and the family of coproduct injections 
$(K(z,v)\to\sum\limits_{f(z')=f(z)}K(z',v))_{(z,v)\in Z\times V}$. Again, for $V=1$ the singleton,  \cref{prop:ZModfibredoverMnd} says that the functor $\VComod\to\VCocat$ is an opfibration\cref{note8}.
	
Moving to the results of \cref{prop:sourcetargetbifib}, we deduce that there is a bifibration $\Gr{s}\colon\Mod(\VMMat)\to\nc{Set}$ and a fibration 
$\Gr{t}\colon\Mod(\VMMat)\to\nc{Set}$. For the source functor, the cartesian lifting of a function $f\colon U\to T$ to a $B$-module $N\colon T\tickar Y$ has as 
domain the $B$-module $N\odot f_{*}\colon U\tickar Y$, where $(N\odot f_{*})(y,u)=N(y,f(u))$ and with action the restriction of the action of $N$, namely $B(y,y')\ot N(y',f(u))\to N(y,f(u))$ for all $y,y',u$.  
The cartesian lifting of $f$ is then essentially given by the family of identity morphisms $N(y,f(u))\to N(y,f(u))$, $(y,u)\in Y\times U$. For the target functor, the cartesian lifting of $g\colon X\to Y$ to a $B$-module $N\colon T\tickar Y$ has domain the $f^{*}\odot N\odot f_{*}$-module $f^{*}\odot N$, where $(f^{*}\odot N\odot 
f_{*})(x,x')=B(f(x),f(x'))$ for all $x,x'\in X$ and $f^{*}\odot N$ is as described earlier. The action is again a restriction of that of $N$ and 
similarly the lifting itself is a family of identity morphisms.
\end{ex}

Under the assumption of fibrancy on $\dc{D}$, one can use the earlier monadicity results to easily deduce (co)completeness for categories of 
(co)modules. 

\begin{prop}\label{prop:Mod(D)complete}
	Let $\dc{D}$ be a fibrant double category which is parallel complete. Then $\Mod(\dc{D})$ is complete and both forgetful functors 
$\Mod(\dc{D})\to\Mnd(\dc{D})$ and $\Mod(\dc{D})\to\dc{D}_1$ preserve limits.
\end{prop}
\begin{proof}
Recall by \cref{def:parallelcolimits} that both $\dc{D}_0$, $\dc{D}_1$ are complete categories and that $\Gr{s},\Gr{t}$ are continuous functors.
We know by \cref{prop:Mod(D)monadicComod(D)comonadic} that $\Mod(\dc{D})$ is monadic over $\dc{D}_{1}\times_{\dc{D}_0}\Mnd(\dc{D})$, hence has any 
limits that the latter category has. Observe furthermore that, since $\Gr{t}$ is a fibration in the fibrant setting, the pullback 
\begin{displaymath}
\begin{tikzcd}
		\dc{D}_{1}\times_{\dc{D}_0}\Mnd(\dc{D})\ar[d]\ar[r] & \Mnd(\dc{D})\ar[d,"\Gr{s}/\Gr{t}"] \\
		\dc{D}_1\ar[r,"\Gr{t}"'] & \dc{D}_0
\end{tikzcd}
\end{displaymath}
is also a bi-pullback. Thus, since both $\Mnd(\dc{D})$ (\cref{prop:(co)limits in (co)monads}) and $\dc{D}_1$ have limits and 
the functors $\Mnd(\dc{D})\to\dc{D}_0$ (\cref{cor:MndD0complete}) and $\dc{D}_1\to\dc{D}_0$ preserve them, we deduce that $\dc{D}_{1}\times_{\dc{D}_0}\Mnd(\dc{D})$ has all 
small limits and that the functors $\dc{D}_{1}\times_{\dc{D}_0}\Mnd(\dc{D})\to\Mnd(\dc{D})$ and $\dc{D}_{1}\times_{\dc{D}_0}\Mnd(\dc{D})\to\dc{D}_1$ 
preserve them (e.g. by \cite[Thm.~2.6]{2-dimmonadtheory}).	
\end{proof}

Combining the above with the fact that $\Mod(\dc{D})\to\Mnd(\dc{D})$ is a fibration by \cref{prop:ModfibredoverMnd}, we obtain the following.

\begin{cor}\label{Mod(D)completefibration}
	Let $\dc{D}$ be a parallel complete and fibrant double category. Then the fibration $\Mod(\dc{D})\to\Mnd(\dc{D})$ has all small fibred limits.
\end{cor}
\begin{proof}
By \ref{prop:Mod(D)complete} we have that the total category $\Mod(\dc{D})$ has all small limits and that these are preserved by the functor 
$\Mod(\dc{D})\to\Mnd(\dc{D})$. The result then follows from \ref{prop:fiberwiselimits}.	In particular, this says that all fibers ${}_A\Mod(\dc{D})$ are complete and that
the reindexing functors $\wc{f}\odot\mi$ are continuous functors.
\end{proof}

Similarly, we obtain the following results for comodules.

\begin{prop}\label{prop:Comod(D)cocomplete}
Let $\dc{D}$ be a fibrant double category which is parallel cocomplete. Then $\Comod(\dc{D})$ is cocomplete, $\Comod(\dc{D})\to\dc{D}_1$ preserves colimits and 
the opfibration $\Comod(\dc{D})\to\Cmd(\dc{D})$ has all small opfibred colimits.
\end{prop}

\begin{proof}
The arguments are dual to those from above. In particular, for the cocomplete $\Comod(\dc{D})\to\Cmd(\dc{D})$, the fiber categories ${}_C\Comod(\dc{D})$ are cocomplete and the reindexing functors $\wh{f}\odot\mi$ are cocontinuous.
\end{proof}

In order to establish our main results in this paper, we will further assume that our double category $\dc{D}$ is locally presentable, as per \cref{defi:locallypresentabledoublecat}. A 
preliminary result which will be of use is the following.

\begin{prop}\label{_AMod_CComodLocallyPresentable}
Suppose $\dc{D}$ is a locally presentable double category. For any monad $A\colon X\bular X$,
the category $_{A}\Mod(\dc{D})$ is locally presentable. Similarly, for every comonad $C\colon Z\bular Z$ in $\dc{D}$ the category $_{C}\Comod(\dc{D})$ is locally presentable.
Moreover, all categories of the form $^U_A\Mod(\dc{D})$ and $^V_C\Comod(\dc{D})$ are locally presentable.
\end{prop}

\begin{proof}
By \cref{_AModmonadic_CComodcomonadic}, $_{A}\Mod(\dc{D})$ is monadic over the category $\dc{D}_{1}^{X}$, where the monad is $A\odot -$. The latter 
functor preserves filtered colimits and the category $\dc{D}_{1}^{X}$ is locally presentable (\cref{prop:lp fibers double cats}), both because the double category $\dc{D}$ is assumed 
locally presentable. The result now follows immediately, by \cref{thm:monadalgebras}.
For $_{C}\Comod(\dc{D})$ the argument is similar, since in this case the relevant endofunctor of $\dc{D}_1^{Z}$ exhibiting comonadicity is $-\odot C$, 
which also preserves filtered colimits by local presentability of $\dc{D}$. 

Finally, for the fixed-domain, fixed-(co)monad categories, the result follows again by \cref{prop:lp fibers double cats} since as mentioned in the beginning of this section, $^U_A\Mod(\dc{D})$ and $^V_C\Comod(\dc{D})$ are 
categories of Eilenberg-Moore algebras and coalgebras on the hom-categories ${}^U\dc{D}_1^X$ and ${}^V\dc{D}_1^Z$ respectively.
\end{proof}

Notice that the above generalizes the first part of \cref{prop:ModComodlp}, whereas the respective result for the global category of (co)modules is work in progress.

Moving towards monoidal structure of categories of modules and comodules, as a first step we show that double functors preserve modules, similarly to \cref{prop:MonFdouble}. 

\begin{prop}\label{prop:laxdoublefunmod}
Any lax double functor $F=(F_0,F_1)\colon\dc{D}\to\dc{E}$ induces an ordinary functor
\begin{displaymath}
	\Mod F:\Mod(\dc{D})\to\Mod(\dc{E})
\end{displaymath}
between their categories of modules, by restricting 
$F_1$. Dually, any oplax double functor induces a functor $\Comod F\colon\Comod(\dc{D})\to\Comod(\dc{E})$
between the categories of comodules.
\end{prop}

\begin{proof}
Consider a module $M\colon U\bular X$ over the monad $A\colon X\bular X$. We can define an action of the induced monad $FA\colon FX\bular FX$
on $FM\colon FU\bular FX$ via the following composite 2-cell
\begin{displaymath}
	\begin{tikzcd}[sep=0.3in]
		FU\ar[r,bul,"FM"]\ar[d,equal]\ar[drr,phantom,"\Two \phi"] & FX\ar[r,bul,"FA"] & FX\ar[d,equal] \\
		FU\ar[rr,bul,"F(A\odot M)"]\ar[d,equal]\ar[drr,phantom,"\Two F\lambda"] && FX\ar[d,equal] \\
		FU\ar[rr,bul,"FM"'] && FX
	\end{tikzcd}
\end{displaymath}
The module axioms now follow from the corresponding axioms for the $A$-module $M$ together with the naturality of $\phi,\phi_0$. Similarly, a 
module morphism $M_{A}\to N_{B}$ consisting of a monad morphism $^f\alpha^f$ and a 2-cell $^g\beta^f$ as in \cref{def:leftmodulesdouble}
induces a module morphism $(FM)_{FA}\to (FN)_{FB}$, given by the monad map 
$F\alpha$ and the 2-cell $F\beta$.

Finally, all the above can be dualized to statements about comodules. Notice that by construction, $\Mod F$ and $\Comod F$ ``sit above'' $\Mnd F$ and $\Cmd F$, namely the image of
an $A$-module $M$ under $F$ is an $FA$-module and similarly for the maps.
\end{proof}	

Let us also note that the induced functor $\Mod F\colon\Mod(\dc{D})\to\Mod(\dc{E})$ clearly specializes to functors 
$^U\Mod(\dc{D})\to{}^{FU}\Mod(\dc{E})$, $_A\Mod(\dc{D})\to{} _{FA}\Mod(\dc{E})$ and $_A^U\Mod(\dc{D})\to{}_{FA}^{FU}\Mod(\dc{E})$, for any object 
$U\in\dc{D}_0$ and monad $A\colon X\bular X$ in $\dc{D}$.

As a particular class of examples of \cref{prop:laxdoublefunmod}, we can 
consider the pseudo double functor $\otimes:\dc{D}\times\dc{D}\to\dc{D}$ of a 
monoidal double category $\dc{D}$. Recall from \cref{prop:MonDComonDmonoidal} 
that we then have an induced monoidal structure on $\Mnd(\dc{D})$ and 
$\Cmd(\dc{D})$, obtained by restricting $\otimes_1$. Now we obtain a similar 
result for the global categories of modules and comodules in $\dc{D}$.

\begin{prop}\label{prop:ModComodmonoidal}
If $\dc{D}$ is a monoidal double category, then the categories $\Mod(\dc{D})$ and $\Comod(\dc{D})$ inherit a monoidal structure from $\dc{D}_1$. 
Moreover, they are braided or symmetric monoidal when $\dc{D}$ is.
\end{prop}

Explicitly, the tensor product of an $A$-module $M\colon U\bular X$ and a $B$-module $N\colon T\bular Y$ is the $A\ot B$-module $M\otimes N\colon U\ot T\bular X\ot Y$ with action 
$$(A\ot B)\odot(M\ot N)\stackrel{(\ref{eq:monoidaldoubleiso})}{\cong}
(A\odot M)\ot(B\odot N)\xrightarrow{\lambda\ot\lambda}(M\ot N)$$
and the monoidal unit is $1_I$. Notice how this monoidal structure does not 
restrict to any of the fixed domain or fixed monad module categories in a 
straightforward way, since in return it provides functors
\begin{equation}\label{eq:restrictedtensor}
^U\Mod(\dc{D})\times {}^U\Mod(\dc{D})\to {}^{U\ot U}\Mod(\dc{D})\quad\textrm{ and }\quad
_A\Mod(\dc{D})\times {}_A\Mod(\dc{D})\to {}_{A\ot A}\Mod(\dc{D}).
\end{equation}

By \cref{prop:Mndmonfib}, in a fibrant monoidal double 
category $\dc{D}$ the categories of monads and comonads form a monoidal 
fibration and opfibration respectively, above the 
category of objects. For modules and comodules, we have the following analogous 
result for the (op)fibration of (co)modules over (co)monads established in 
\cref{prop:ModfibredoverMnd}.

\begin{prop}\label{prop:ModMndmonoidal}
 If $\dc{D}$ is a fibrant (braided/symmetric) monoidal double category, the fibration 
$\Mod(\dc{D})\to\Mnd(\dc{D})$ and opfibration $\Comod(\dc{D})\to\Cmd(\dc{D})$ are (braided/symmetric) monoidal.
\end{prop}

\begin{proof}
To verify \cref{monoidalfibration}, first notice that all involved categories inherit their monoidal product from $\dc{D}$. By 
construction, the underlying monad of $M_A\ot N_B$ is $A\ot B$, the tensor of 
the underlying monads, hence both functors are strict monoidal. Moreover, the 
tensor preserves cartesian liftings: indeed, consider monad morphisms 
$^{f}\alpha^{f}\colon A'\to A$ and $^{g}\beta^{g}\colon B'\to B$ and modules 
$M_A$ and $N_B$. Then the pair of cartesian liftings $\Cart(\alpha,M)$ and 
$\Cart(\beta,N)$ given in \cref{eq:Modcartlifts} gets tensored as the top arrow below 
\begin{displaymath}
	\xymatrix @C=.6in @R=.3in
	{(\wc{f}\odot M)\otimes(\wc{g}\odot 
N)\ar[rr]^{\qquad\quad\Cart(\alpha,M)\otimes\Cart(\beta,N)}
		\ar @{-->}[d]_-{\cong} && M\otimes N\ar @{.>}[dd] &\\
		\wc{(f\otimes g)}\odot(M\otimes 
N)\ar[urr]_-{\quad\Cart(\alpha\otimes\beta,M\otimes N)}\ar @{.>}[d]
		&&& \textrm{in }\Mod(\dc{D}) \\
		A'\otimes B'\ar[rr]_-{\alpha\otimes\beta} && A\otimes B & \textrm{in 
}\Mnd(\dc{D})}
\end{displaymath}
The vertical isomorphism on the left-hand side is the 
interchange isomorphism composed with the canonical 
isomorphism $\wc{f}\otimes\wc{g}\cong\wc{f\otimes g}$ in any fibrant double category. It then follows that the triangle in the above diagram commute,
and thus $\Mod(\dc{D})\to\Mnd(\dc{D})$ is a monoidal fibration. Braiding and symmetry are easily verified to agree with the structure.

Dual arguments establish that $\Comod(\dc{D})\to\Cmd(\dc{D})$ is a monoidal opfibration.
\end{proof}

Using a similar methodology, we can also verify that the source bifibrations of 
(co)modules over the category of objects are 
monoidal.

\begin{prop}\label{prop:ModD0monfib}
If $\dc{D}$ is a fibrant monoidal double category, the bifibrations 
$\Gr{s}\colon\Mod(\dc{D})\to\dc{D}_0$ and $\Gr{s}\colon\Comod(\dc{D})\to\dc{D}_0$ are 
monoidal.
\end{prop}

\begin{proof}
Consider the source fibration $\Mod(\dc{D})\to\dc{D}_0$ of \cref{prop:sourcetargetbifib}. Indeed, the tensor 
product of two modules $M_A\ot N_B\colon U\ot T\bular X\ot Y$ is mapped to 
$U\ot T$ hence the fibration is a strict monoidal functor. Moreover, for 
$M_A\colon U\bular X$, $N_B\colon T\bular Y$ and $f\colon U'\to U, g\colon 
T'\to T$, there is a unique globular isomorphism  
\begin{displaymath}
\xymatrix @C=.6in @R=.3in{(M\odot\wh{f})\otimes 
(N\odot\wh{g})\ar[rr]^{\qquad\quad\Cart(f,M)\otimes\Cart(g,N)}
\ar @{-->}[d]_-{\cong} && M\otimes N\ar @{.>}[dd] &\\
(M\otimes N)\odot\wh{(f\otimes 
g)}\ar[urr]_-{\quad\Cart(f\otimes g,M\otimes N)}\ar @{.>}[d]
&&& \textrm{in }\Mod(\dc{D}) \\
U'\otimes T'\ar[rr]_-{f\otimes g} && U\otimes T & \textrm{in 
}\dc{D}_0}
\end{displaymath}
coming from the monoidal double category structure, which makes the triangle 
commute thus establishing that the tensor product preserves liftings. The other three structures -- monoidality of the source opfibration of modules, and of the source fibration and opfibration of comodules -- are established analogously.
\end{proof}

It was mentioned above that fixed-domain module categories are not monoidal in general, since 
the tensor product functor induces functors between modules with different 
domains as in \cref{eq:restrictedtensor}. For our purposes, we investigate 
certain assumptions under which fixed-domain (co)module categories of monoidal double categories do inherit the 
tensor product.

\begin{lem}\label{lem:lemmaabove}
If $\dc{D}$ is a fibrant double category and $T\cong U$ as objects in $\dc{D}_0$, then there is an (adjoint) equivalence between the respective categories of 
fixed-domain modules $^U\Mod(\dc{D})\simeq{}^T\Mod(\dc{D})$. 
\end{lem}

\begin{proof}
If $f\colon T\xrightarrow{\sim} U$ is the vertical isomorphism, then the induced functor 
$\mi\odot\wh{f}\colon{}^U\Mod(\dc{D})\to{}^T\Mod(\dc{D})$ is an (adjoint) equivalence with adjoint $\mi\odot\wc{f}$, essentially because $\wh{f}\dashv\wc{f}$ is an adjoint 
equivalence in the horizontal bicategory by \cref{lem:fibrantproperties}, namely $\wh{f}\odot\wc{f}\cong1_U$ and $\wc{f}\odot\wh{f}\cong1_T$ via globular 2-cells in $\dc{D}_1$. 
\end{proof}

Using this equivalence, we can restrict the tensor product functor accordingly 
and obtain the following result for modules with domain the monoidal 
unit $I$ -- and not an arbitrary object $U\in\dc{D}_0$. This is due to the fact that $1_I\colon I\bular I$ is a monoidal unit in $\dc{D}_1$.

\begin{prop}\label{prop:IModmonoidal}
For a fibrant monoidal double category $\dc{D}$, $^I\Mod(\dc{D})$ 
and $^I\Comod(\dc{D})$ obtain a monoidal structure. It is braided or symmetric when $\dc{D}$ is braided or symmetric monoidal.
\end{prop}

\begin{proof}
Take the canonical vertical isomorphism $\rho_I=\lambda_I\colon I\ot I\cong I$ in the monoidal 
category $\dc{D}_0$. For the desired monoidal structure, we define the unit to be $1_I\colon I\bular I$ and the 
tensor product to be 
\begin{displaymath}
\diamond\colon{}^I\Mod(\dc{D})\times{}^I\Mod(\dc{D})\xrightarrow{\cref{eq:restrictedtensor}}{}^{I\ot 
I}\Mod(\dc{D})\simeq{}^I\Mod(\dc{D})
\end{displaymath}
where the equivalence is $\mi\odot\wc{\rho_I}$. It can be verified this is a monoidal structure, where for example the left 
unitor $I\diamond M\cong M$ is built as
\begin{displaymath}
 \begin{tikzcd}[column sep=.6in,row sep=.5in]
I\ar[r,bul,"\wc{\rho_I}"]\ar[d,equal]\ar[dr,phantom,"\stackrel{q_2}{\cong}"] & I\ot 
I\ar[dr,phantom,"\stackrel{\rho_M}{\cong}"]\ar[r,bul,"M\ot 1_I"]\ar[d,"\rho"description] & X\ot 
I\ar[d,"\rho_X"] \\
I\ar[r,bul,"1_I"'] & I\ar[r,bul,"M"'] & X
 \end{tikzcd}
\end{displaymath}
where the structure 2-morphism $q_2$ of the conjoint of $\rho_I$ is invertible by \cref{lem:fibrantproperties} and the right hand side isomorphism is the right unitor of the 
monoidal double category $\dc{D}$.

For the braiding or symmetry, 
we can build a braiding as in
\begin{displaymath}
 \begin{tikzcd}[column sep=.6in] 
I\ar[r,bul,"\wc{\rho_I}"]\ar[d,equal]\ar[dr,phantom,"\cong"] & I\ot 
I\ar[dr,phantom,"\tau"]\ar[r,bul,"M\ot N"]\ar[d,"\tau"description] & X\ot Y\ar[d,"\Two\tau"] \\
I\ar[r,bul,"\wc{\rho_I}"'] & I\ot I\ar[r,bul,"N\ot M"'] & Y\ot X
 \end{tikzcd}
\end{displaymath}
where $\tau$ is the braiding of $\Mod(\dc{D})$ and the left-hand side isomorphism comes from the axioms the monoidal category $\dc{D}_0$.
Using analogous arguments, ${}^I\Comod(\dc{D})$ can be shown to be monoidal.
\end{proof}

Regarding an arbitrary object $U$ as the fixed domain, one could consider
precomposition along companions or conjoints of some appropriate vertical map $U\ot U\rightleftarrows U$ to obtain composite functors 
$^U\Mod(\dc{D})\times{}^U\Mod(\dc{D})\xrightarrow{\otimes}{}^{U\ot U}\Mod(\dc{D})\to {}^U\Mod(\dc{D})$ in an effort to construct a tensor product in some coherent way.
It turns out that in a certain class of monoidal double categories, fixed-domain modules and comodules always 
inherit a monoidal structure like that.

\begin{prop}\label{prop:cartesianD0}
Suppose that $\dc{D}$ is a fibrant monoidal double category, whose category of objects 
is cartesian monoidal $(\dc{D}_0,\times,1)$. Then all fixed-domain module 
and comodule categories ${}^U\Mod(\dc{D})$ and ${}^V\Comod(\dc{D})$ for arbitrary objects $U,V$ are monoidal. Moreover, they are braided/symmetric if $\dc{D}$ is.
\end{prop}

\begin{proof}
According to the so-called monoidal Grothendieck construction, since the source 
fibration $\Mod(\dc{D})\to\dc{D}_0$ is monoidal by \cref{prop:ModD0monfib}, \cite[Thm.~3.13]{MonGroth} ensures that it  
corresponds to a lax monoidal structure on the corresponding pseudofunctor 
$\dc{D}_0^\op\to\Cat$ which sends an object $U$ to the category of fixed-domain modules $^U\Mod(\dc{D})$ and
a vertical map $f$ to the functor $\mi\odot\wh{f}$ (see \cref{prop:sourcetargetbifib}). The lax monoidal structure maps are precisely the functors $\phi_{U,T}\colon {}^{U}\Mod(\dc{D})\times {}^T\Mod(\dc{D})\to {}^{U\ot 
T}\Mod(\dc{D})$ discussed above, along with $\phi_0\colon \mathbf{1}\to 
{}^I\Mod(\dc{D})$ choosing $1_I$.

If the base monoidal category $\dc{D}_0$ is cartesian monoidal, 
by \cite[Thm.~12.7]{Framedbicats} or \cite[Thm.~4.2]{MonGroth} the monoidal structure of the fibration further 
corresponds to a monoidal structure on the fibres, via
\begin{displaymath}
^U\Mod(\dc{D})\times ^U\Mod(\dc{D})\to {}^{U\times 
U}\Mod(\dc{D})\xrightarrow{\mi\odot\wh{\Delta}} {}^U\Mod(\dc{D}),\qquad 
\mathbf{1}\to {}^1\Mod(\dc{D})\xrightarrow{\mi\odot\wh{!}}{}^U\Mod(\dc{D})
\end{displaymath}
where $\Delta\colon U\to U\times U$ is the diagonal and $!\colon U\to 1$ the 
unique map to the terminal object. Therefore in that case, all fixed-domain 
module categories are indeed monoidal. 

Analogous arguments hold for $^V\Comod(\dc{D})$, again using the monoidal source fibration $\Comod(\dc{D})\to\dc{D}_0$.
Finally, when $\dc{D}$ is braided monoidal then so are the monoidal (op)fibrations in question, 
hence again by \cite[Thm.~12.7]{Framedbicats} the fibers ${}^U\Mod(\dc{D})$ and $^V\Comod(\dc{D})$ become braided monoidal.
\end{proof}

The condition of the above proposition is satisfied, for example, in any cartesian double category in the sense of \cite{Aleiferi}. Moreover, it holds in the case of $\VMMat$ whose category of objects is just sets.

\begin{cor}\label{cor:VModmonoidal}
Suppose that $\ca{V}$ is a braided monoidal category with small coproducts preserved by $\ot$ in each variable. Then the categories of two-indexed modules and comodules $\Mod(\VMMat)$ and $\Comod(\VMMat)$, as well as the categories
of enriched modules $\VMod$ and comodules $\VComod$ are monoidal. 
If $\ca{V}$ is symmetric monoidal, then so are all four categories of modules and comodules.
\end{cor}

\begin{proof}
Under the initial assumptions on $\ca{V}$, the double category $\VMMat$ is fibrant and monoidal (\cref{ex:VMMatfibrant,ex:VMMatmonoidal}).
Hence the first part for the two-indexed (co)modules of \cref{ex:twoVMod} follows from \cref{prop:ModComodmonoidal}, and the second part follows from \cref{prop:IModmonoidal}, for $\VMod{=}{}^1\Mod(\VMMat)$ and $\VComod{=}{}^1\Comod(\VMMat)$. More explicitly, for a $\ca{V}$-module $M\colon1\tickar X$ of a $\ca{V}$-category $A\colon X\tickar X$ and $N\colon1\tickar Y$ of a $\ca{V}$-category $B\colon Y\tickar Y$, their tensor product is $M\ot N\colon 1\tickar X\ot Y$ given by $(M\ot N)(x,y)=M(x)\ot N(y)$ with an action from the $\ca{V}$-category $A\ot B$ -- which is of course the natural structure for enriched modules. The monoidal unit is given by $1(*,*)=I$ with a trivial action from the terminal $\ca{V}$-category $\mathbf{1}$. If $\ca{V}$ is in fact symmetric, then $\VMMat$ is a symmetric monoidal doube category and so categories $\VMod$ and $\VComod$ become symmetric monoidal.

In fact, since $\VMMat$ has $(\Set,\times,\{*\})$ as its vertical category, \cref{prop:cartesianD0} also applies and gives any $^U\Mod(\dc{D})$ and $^V\Comod(\dc{D})$ a monoidal structure. In more detail, for two modules $M\colon U\tickar X$ and $M'\colon U\tickar Y$ in ${}^U\Mod(\dc{D})$, their tensor product is $M\otimes N\colon U\tickar X\times Y$ given by 
$(M\ot M')(u,(x,y))=M(u,x)\ot M'(u,y)$, whereas the monoidal unit is $J\colon U\tickar 1$ given by $J(u)=I$ for all $u\in U$.
\end{proof}

\subsection{Enrichment of modules in comodules}\label{sec:enrichedfibdouble}

We now turn towards the main goal of this work, namely establishing an enrichment of modules over monads in comodules over comonads, in the setting
of a fibrant monoidal double category which is moreover monoidal closed and locally presentable. Using results obtained in the previous section, we gradually
build up the theory required to apply \cref{thm:cotensorenrich,thm:enrichedfib} in order to show that $\Mod(\dc{D})$ is tensored and cotensored enriched in 
$\Comod(\dc{D})$, and that the fibration $\Mod(\dc{D})\to\Mnd(\dc{D})$ is enriched in the opfribration $\Comod(\dc{D})\to\Cmd(\dc{D})$ (\cref{thm:big2}). The techniques are
very analogous to those employed for the respective result including monads and comonads, \cref{thm:big1}.

Suppose that $\dc{D}$ is a monoidal closed double category (\cref{def:locclosed}). Recall that there is an induced 
functor $H\colon \Cmd(\dc{D})^\op\times\Mnd(\dc{D})\to\Mnd(\dc{D})$ between the categories of (co)monads defined as in \cref{eq:MonHdouble}, 
denoted $H$. Now due to \cref{prop:laxdoublefunmod}, we also obtain a functor
\begin{equation}\label{eq:ModHdouble}
 H\colon \Comod(\dc{D})^\op\times\Mod(\dc{D})\to\Mod(\dc{D})
\end{equation}
which maps a $C$-comodule $K\colon V\bular Z$ and an $A$-module $M\colon U\bular X$ to the $H(C,A)$-module $H(K,M)\colon $ 
$H(V,U)\bular H(Z,X)$ with action
\begin{equation}\label{eq:indact}
 \begin{tikzcd}[column sep=.5in]
H(V,U)\ar[d,equal]\ar[r,bul,"{H(K,M)}"]\ar[drr,phantom,"{\Two H_{\odot}}"] & H(Z,X)\ar[r,bul,"{H(C,A)}"] & H(Z,X)\ar[d,equal] \\
H(V,U)\ar[d,equal]\ar[rr,bul,"{H(C\odot K,A\odot M)}"]\ar[drr,phantom,"{\Two H(\gamma,\lambda)}"] && H(Z,X)\ar[d,equal] \\
H(V,U)\ar[rr,bul,"{H(K,M)}"'] && H(Z,X)
 \end{tikzcd}
\end{equation}
where $H_{\odot}$ is the structure map of the lax double functor $H$, $\gamma\colon\K\Rightarrow C\odot \K$ is the $C$-coaction and $\lambda\colon 
A\odot M\Rightarrow M$ is the $A$-action. This can be better understood by taking the adjunct composite arrow under $\mi\ot K\dashv H(K,\mi)$ in $\dc{D}_1$, namely the double
categorical analogue of \cref{eq:KMaction}:
\begin{equation}\label{eq:modunder}
\begin{tikzcd}[column sep=.7in]
H(V,U)\ot V\ar[rr,bul,"{(H(C,A)\odot H(K,M))\ot K}"]\ar[d,equal]\ar[drr,phantom,"\Two1\ot\gamma"] && H(Z,X)\ot Z\ar[d,equal] \\
H(V,U)\ot V\ar[d,equal]\ar[rr,bul,"{(H(C,A)\odot H(K,M))\ot (C\odot K)}"']\ar[drr,phantom,"\Two\tau"] && H(Z,X)\ot Z\ar[d,equal] \\
H(V,U)\ot V\ar[d,"\ev"']\ar[r,bul,"{H(K,M)\ot K}"']\ar[dr,phantom,"\Two\ev"] & H(Z,X)\ot Z\ar[dr,phantom,"\Two\ev"]\ar[r,bul,"{H(C,A)\ot C}"']\ar[d,"\ev"] & H(Z,X)\ot Z\ar[d,"\ev"] \\
U\ar[r,bul,"M"']\ar[drr,phantom,"\Two\lambda"]\ar[d,equal] & X\ar[r,bul,"A"'] & X\ar[d,equal] \\
U\ar[rr,bul,"M"']&&X 
\end{tikzcd}
\end{equation}
where $\tau$ is the interchange isomorphism of the monoidal double category.

The following result establishes that the above functor $H$ is an action of the category of comodules on the category of modules.

\begin{prop}\label{prop:ModHaction}
Suppose $\dc{D}$ is a monoidal closed double category. Then $H\colon\Comod(\dc{D})^\op\times\Mod(\dc{D})\to\Mod(\dc{D})$ is an action.
\end{prop}

\begin{proof}
Similarly to \cref{rmk:fixedaction}, since $\dc{D}_1$ is monoidal closed we have structure isomorphisms as in \cref{eq:actions} which are here written
\begin{equation}\label{eq:chimod}
\begin{tikzcd}[column sep=.7in]
H(V,H(S,T))\ar[r,bul,"{H(\K,H(\LL,N))}"]\ar[d,"\chi_0"']\ar[dr,phantom,"\Two\psi_1"] & H(Z,H(W,Y))\ar[d,"\chi_0"] \\
H(V\ot S,T)\ar[r,bul,"{H(\K\ot\LL,N)}"'] & H(Z\ot W,Y)
\end{tikzcd}\quad
\begin{tikzcd}[column sep=.6in]
T\ar[r,bul,"N"]\ar[d,"\nu_0"']\ar[dr,phantom,"\Two\mu_1"] & Y\ar[d,"\nu_0"] \\
H(I,T)\ar[r,bul,"{H(1_I,N)}"'] & H(I,Y)
\end{tikzcd}
\end{equation}
for any horizontal 1-cells $\K\colon V\bular Z$, $\LL\colon S\bular W$ and $N\colon T\bular Y$, that correspond under the adjunction to 
$\mathrm{ev}(\mathrm{ev}\ot1)$ and $\rho$ analogously to \cref{eq:chiunder}. 

It can then be verified that when $\K$ is a $C_Z$-comodule, $\LL$ is 
a $D_W$-comodule and $N$ is a $B_Y$-module for comonads $C,D$ and monad $B$, these isomorphisms are in fact module morphisms as in \cref{def:leftmodulesdouble}. 
More specifically, $\psi_1$ and $\mu_1$ commute with the 
induced module structures on the $H(C,H(D,B))$-module $H(\K,H(\LL,N))$ and the $H(C\ot D,B)$-module $H(\K\ot\LL,N)$, the $H(1_I,B)$-module $H(1,N)$ and the $B$-module $N$ respectively, with underlying monad morphisms $\chi_1$ and $\nu_1$ from \cref{eq:actionstructuredouble} between the involved monads. 
For example, the compatibility \cref{eq:modulemapaxiom} for $\psi_1$ in this case is written as 
\begin{displaymath}
\scalebox{0.7}{\begin{tikzcd}[column sep=.6in,ampersand replacement=\&]
H(V,H(S,T))\ar[d,equal]\ar[r,bul,"{H(K,H(L,N))}"]\ar[drr,phantom,"\Two"] \& H(Z,H(W,Y))\ar[r,bul,"{H(C,H(D,B))}"] \& H(Z,H(W,Y))\ar[d,equal] \\
H(V,H(S,T))\ar[rr,bul,"{H(K,H(L,N))}"]\ar[d,"\chi_0"']\ar[drr,phantom,"\Two\psi_1"] \&\& H(Z,H(W,Y))\ar[d,"\chi_0"] \\
H(V\ot S,T)\ar[rr,bul,"{H(K\ot L,N)}"'] \&\& H(Z\ot W,Y)
\end{tikzcd}=
\begin{tikzcd}[column sep=.6in,ampersand replacement=\&]
H(V,H(S,T))\ar[d,"\chi_0"']\ar[r,bul,"{H(K,H(L,N))}"]\ar[dr,phantom,"\Two\psi_1"] \& H(Z,H(W,Y))\ar[d,"\chi_0"]\ar[dr,phantom,"\Two\chi_1"]\ar[r,bul,"{H(C,H(D,B))}"] \& H(Z,H(W,Y))\ar[d,"\chi_0"] \\
H(V\ot S,T)\ar[d,equal]\ar[r,bul,"{H(K\ot L,N)}"]\ar[drr,phantom,"\Two"] \& H(Z\ot W,Y)\ar[r,bul,"{H(C\ot D,B)}"'] \& H(Z\ot W,Y)\ar[d,equal] \\
H(V\ot S,T)\ar[rr,bul,"{H(K\ot L,N)}"'] \&\& H(Z\ot W,Y)
\end{tikzcd}}
\end{displaymath}
where the unlabelled 2-cells are the actions formed according to \cref{eq:indact}. To check this axiom, one works under the adjunction $\mi\ot A\dashv H(A,\mi)$ for the monoidal closed $\dc{D}_1$
using the adjunct 2-cells as in \cref{eq:chiunder,eq:modunder} and applying various axioms like functoriality of $\ot_1$, naturality of the interchange law and the `pentagon' axiom
for interchange that holds in any monoidal double category. 

The maps $\psi_1$ and $\mu_1$ satisfy the action axioms because they do so between arbitrary objects in $\dc{D}_1$, therefore $H\colon\Comod(\dc{D})^\op\times\Mod(\dc{D})\to\Mod(\dc{D})$ is an action -- and so is its opposite $H^\op$.
\end{proof}

The following result, similarly to \cref{prop:MonHcartesian}, establishes that this functor $H$ preserves cartesian liftings of the respective 
fibrations over monads and comonads\footnote{For $\VMMat$, such a calculation is sketched in \cite[Lem.~7.7.3]{PhDChristina}. Again, without the current assumptions on a general double category,
such a result could not then be obtained in a broad context.}.

\begin{prop}\label{prop:ModHcartesian}
Suppose $\dc{D}$ is a braided monoidal closed and fibrant double category. The total functor $H$ on the top of
\begin{equation}\label{eq:HHmodfibred}
\begin{tikzcd}
\Comod(\dc{D})^\op\times\Mod(\dc{D})\ar[r,"H"]\ar[d] & \Mod(\dc{D})\ar[d] \\
\Cmd(\dc{D})^\op\times\Mnd(\dc{D})\ar[r,"H"] & \Mnd(\dc{D})
\end{tikzcd} 
\end{equation}
preserves cartesian liftings, therefore makes $(H,H)$ into a fibred 1-cell.
\end{prop}

\begin{proof}
The square commutes by definition of the functor $H$. 
We will show that the top $H$ preserves cartesian liftings in each variable separately, for the fibrations of \cref{prop:ModfibredoverMnd}. That's enough since
cartesian liftings of a product of fibrations are pairs of liftings in each one, see \cite[Prop.~8.1.14]{Handbook2}.
Recall that part of the definition of a monoidal closed double category (\cref{def:locclosed}) is a map of adjunctions
\begin{displaymath}
\begin{tikzcd}[column sep=.6in, row sep=.4in]
\dc{D}_1\ar[r,shift left=2,"{\mi\ot \K}"]\ar[d,"\mathfrak t"']\ar[r,phantom,"\bot"] & \dc{D}_1\ar[d,"\mathfrak t"]\ar[l,shift 
left=2,"{H(\K,\mi)}"] \\
\dc{D}_0\ar[r,shift left=2,"{\mi\ot Z}"]\ar[r,phantom,"\bot"] & \dc{D}_0\ar[l,shift left=2,"{H(Z,\mi)}"]
\end{tikzcd}
\end{displaymath}
for any horizontal arrow $\K$ with target $Z$. In the fibrant context, $\mathfrak t$ is a fibration (\cref{globalvslocal})
and since right adjoints in $\Cat^\mathbf{2}$ preserve cartesian liftings (see e.g. \cite[Ex.~9.4.4]{Jacobs}), 
we have that $H(\K,\mi)$ does so: for any $N$ with target $Y$ and any vertical map $f\colon Y'\to Y$, cartesian maps are mapped to cartesian 
maps as follows
\begin{equation}\label{eq:here}
\begin{tikzcd}[column sep=.3in]
&&&& H(\K,\wc{f}\odot N)\ar[drr,bend left=15,"{H(1,\Cart(f,N))}"]\ar[d,dashed,"\exists!u"',"\cong"] &&& \\
\wc{f}\odot N\ar[rr,"{\Cart(f,N)}"]\ar[d,-,dotted] && N\ar[d,-,dotted] & \textrm{in }\dc{D}_1 & 
\wc{H(1_Z,f)}\odot H(\K,N)\ar[rr,"{\Cart}"]\ar[d,-,dotted] && H(\K,N)\ar[d,-,dotted] & \textrm{in }\dc{D}_1 \\
Y'\ar[rr,"{f}"] && Y & \textrm{in }\dc{D}_0\ar[ur,phantom,near start,"\mapsto"] & H(Z,Y')\ar[rr,"{H(1_Z,f)}"] && H(Z,Y) & \textrm{in }\dc{D}_0
\end{tikzcd} 
\end{equation}
Due to the nature of cartesian liftings for the fibration $\Mod(\dc{D})\to\Mnd(\dc{D})$ given by 
\cref{eq:_Mod}, we are able to use the above universal triangle to show that in fact the top of
\begin{displaymath}
\begin{tikzcd}[column sep=.5in]
\Mod(\dc{D})\ar[r,"{H(\K,\mi)}"]\ar[d] & \Mod(\dc{D})\ar[d] \\
\Mnd(\dc{D})\ar[r,"{H(\K,\mi)}"] & \Mnd(\dc{D})
\end{tikzcd}
\end{displaymath}
for any $C_Z$-comodule $\K$ is a cartesian functor. Indeed, take the image of a lifting of some $B_Y$-module $N$
\begin{displaymath}
\begin{tikzcd}[column sep=.25in]
&&&& H(\K,\wc{f}\odot N)\ar[drr,bend left=15,"{H(1,\Cart(\alpha,N))}"]\ar[d,dashed,"\exists!k"'] &&& \\
\wc{f}\odot N\ar[rr,"{\Cart(\alpha,N)}"]\ar[d,-,dotted] && N\ar[d,-,dotted] & \textrm{in }\Mod(\dc{D}) & 
\wc{H(1_Z,f)}\odot H(\K,N)\ar[rr,"{\Cart}"]\ar[d,-,dotted] && H(\K,N)\ar[d,-,dotted] & \textrm{in }\Mod(\dc{D}) \\
B'_{Y'}\ar[rr,"{\alpha_f}"] && B_Y & \textrm{in }\Mnd(\dc{D})\ar[ur,phantom,near start,"\mapsto"] & H(C,B')\ar[rr,"{H(1_C,\alpha)_{H(1_Z,f)}}"] && 
H(C,B) & \textrm{in }\Mnd(\dc{D})
\end{tikzcd} 
\end{displaymath}
that looks very similar to \cref{eq:here} but the base category is of horizontal 1-cells (and specifically monads) rather than 0-cells.
By the universal property of cartesian liftings in $\Mod(\dc{D})$, there exists a unique module map $k\colon H(\K,\wc{f}\odot N)\to \wc{H(1,f)}\odot H(\K,N)$ that makes the triangle
commute.
First of all, we notice that $k$ has the same domain and codomain 
as the $\dc{D}_1$-isomorphism $u$ from \cref{eq:here}. Then, the cartesian liftings of $\Mod(\dc{D})\to\Mnd(\dc{D})$ given as in 
\cref{eq:Modcartlifts} in fact coincide with those of $\dc{D}_1\xrightarrow{\mathfrak t}\dc{D}_0$ given as the right side of \cref{eq:D1targetlifts}, 
essentially because the former only employs the underlying vertical map of the monad maps involved. As a result, since $\Cart\circ k=H(1,\Cart)$ in $\Mod(\dc{D})$ in particular means that
the equality holds in $\dc{D}_1$, and $u$ is the unique map in $\dc{D}_1$ that makes the same diagram commute there, the module map $k$ is forced to be equal, as a horizontal 1-cell, to $u$ and 
thus $k$ is an isomorphism.

Using a very analogous argument, we can show that in the commutative
\begin{equation}\label{eq:cartesianfirstvar}
\begin{tikzcd}[column sep=.5in]
\Comod(\dc{D})^\op\ar[r,"{H(\mi,N)}"]\ar[d] & \Mod(\dc{D})\ar[d] \\
\Cmd(\dc{D})^\op\ar[r,"{H(\mi,B)}"] & \Mnd(\dc{D})
\end{tikzcd}
\end{equation}
the top functor $H(\mi,N)$ for any $B_Y$-module is cocartesian, essentially using the fact that 
\begin{displaymath}
\begin{tikzcd}[column sep=.6in, row sep=.4in]
\dc{D}_1\ar[r,shift left=2,"{H(\mi,N)^\op}"]\ar[d,"\mathfrak t"']\ar[r,phantom,"\bot"] & \dc{D}_1^\op\ar[d,"\mathfrak t^\op"]\ar[l,shift 
left=2,"{H(\mi,N)}"] \\
\dc{D}_0\ar[r,shift left=2,"{H(\mi,Y)^\op}"]\ar[r,phantom,"\bot"] & \dc{D}_0^\op\ar[l,shift left=2,"{H(\mi,Y)}"]
\end{tikzcd}
\end{displaymath}
is a map of adjunctions for any braided monoidal closed double category -- the braiding needed for the existence of the bottom adjunction. Then $H(\mi,N)$, again as a right adjoint in $\Cat^2$, preserves cartesian liftings 
from $\mathfrak t^\op$ to the fibration $\mathfrak t$ in the fibrant context, meaning that it maps cocartesian liftings to 
cartesian liftings as follows:
\begin{displaymath}
\begin{tikzcd}[column sep=.3in]
&&&& H(\wh{g}\odot\K,N)\ar[drr,bend left=15,"{H(\Cocart(g,\K),1)}"]\ar[d,dashed,"\exists!v"',"\cong"] &&& \\
\K\ar[d,-,dotted]\ar[rr,"{\Cocart(g,\K)}"] && \wh{g}\odot\K\ar[d,-,dotted] & \textrm{in }\dc{D}_1 & 
\wc{H(g,1)}\odot H(\K,N)\ar[rr,"{\Cart}"]\ar[d,-,dotted] && H(\K,N)\ar[d,-,dotted] & \textrm{in }\dc{D}_1 \\
Z\ar[rr,"g"] && Z' & \textrm{in }\dc{D}_0\ar[ur,phantom,near start,"\mapsto"] & H(Z',Y)\ar[rr,"{H(g,1_Y)}"] && 
H(Z',Y) & \textrm{in }\dc{D}_0
\end{tikzcd} 
\end{displaymath}
The top triangle on the right appears in the exact same form inside $\Mod(\dc{D})\subseteq\dc{D}_1$ when $\K$ is a $C_Z$-comodule, $N$ is a 
$B_Y$-module and the cocartesian lifting is above a comonad map $\beta_g\colon C_Z\to C'_{Z'}$ for the opfibration $\Comod(\dc{D})\to\Cmd(\dc{D})$ 
as in \cref{eq:Comodcartlifts}, therefore the uniquely induced monad map is forced once again to be an isomorphism and 
$H(\mi,N)\colon\Comod(\dc{D})^\op\to\Mod(\dc{D})$ is cartesian for the fibrations above $\Cmd(\dc{D})^\op$ and $\Mnd(\dc{D})$ respectively.
\end{proof}

We now establish that the global category of comodules in a double category is monoidal closed under certain assumptions, 
similarly to the category of comonads in \cref{prop:Cmdclosed}. This generalizes \cref{prop:ComodVclosed} from monoidal to double categories.

\begin{prop}\label{prop:Comodclosed}
Let $\dc{D}$ be a monoidal closed double category, which is fibrant and locally presentable. Then the category of 
comodules is monoidal closed.
\end{prop}

\begin{proof}
Recall by \cref{prop:ModMndmonoidal} that $\Comod(\dc{D})$ inherits the monoidal structure of $\dc{D}_1$, as does the category of comonads 
$\Cmd(\dc{D})$ by \cref{prop:MonDComonDmonoidal}, therefore the following diagram commutes 
 \begin{displaymath}
  \begin{tikzcd}
 \Comod(\dc{D})\ar[r,"{\mi\ot\K}"]\ar[d] & \Comod(\dc{D})\ar[d] \\
 \Cmd(\dc{D})\ar[r,"{\mi\ot C}"] & \Cmd(\dc{D})
  \end{tikzcd}
 \end{displaymath}
for any fixed left $C_Z$-comodule $\K\colon V\bular Z$. There is an adjunction $\mi\otimes C\dashv\Hom_\Cmd(C,\mi)$ between the bases with counit 
$\varepsilon$, since $\Cmd(\dc{D})$ is monoidal closed under these assumptions by \cref{prop:Cmdclosed}. Moreover, the legs are opfibrations 
by \cref{prop:ModfibredoverMnd}, and the top functor preserves cocartesian liftings. This can be verified either by a direct computation or via an 
analogous argument to the proof of \cref{prop:ModHcartesian} as follows. In any monoidal closed double category (\cref{def:locclosed}), the 
rightmost square is a map of adjunctions for any horizontal 1-cell $\K\colon V\bular Z$, hence the total left adjoint $\mi\otimes\K\colon\dc{D}_1\to\dc{D}_1$ is 
cocartesian between the opfibrations $\mathfrak t$ of \cref{globalvslocal}. This means that a cocartesian lifting on the left is mapped to a 
cocartesian lifting on the right
\begin{displaymath}
\begin{tikzcd}[column sep=.3in]
&&&&&& (\wh{g}\odot\LL)\ot\K & \\
\LL\ar[d,-,dotted]\ar[rr,"{\Cocart(g,\LL)}"] && \wh{g}\odot\LL\ar[d,-,dotted] & \textrm{in }\dc{D}_1 & 
\LL\otimes\K\ar[urr,bend left=15,"{\Cocart(g,\LL)\ot 1_\K}"]\ar[rr,"{\Cocart}"]\ar[d,-,dotted] && 
\wh{(g\ot1)}\odot(\LL\otimes\K)\ar[u,dashed,"\exists!k"',"\cong"]\ar[d,-,dotted] & 
\textrm{in }\dc{D}_1 \\
W\ar[rr,"g"] && W' & \textrm{in }\dc{D}_0\ar[ur,phantom,near start,"\mapsto"] & W\ot Z\ar[rr,"{g\ot1_Z}"] && 
W'\ot Z & \textrm{in }\dc{D}_0
\end{tikzcd}  
\end{displaymath}
for any horizontal 1-cell with target $W$, therefore there exists a unique isomorphism $k$ in $\dc{D}_1$ as above, making the diagram commute. In the 
case that $\LL$ is in fact a $D_W$-comodule, $\K$ is a $C_Z$-comodule and $g$ is the underlying vertical map of a comonad morphism $\beta_g\colon 
D_W\to D'_{W'}$, by the description of the opfibration $\Comod(\dc{D})\to\Cmd(\dc{D})$ as in \cref{eq:Comodcartlifts} we obtain an identical triangle 
in $\Comod(\dc{D})$, thus the uniquely induced comodule map 
$u\colon\wh{(g\ot1)}\odot(\LL\ot\K)\to(\wh{g}\odot\LL)\ot K$ ends up being an isomorphism. 

We can now apply the dual of \cref{thm:totaladjointthm} to obtain the adjoint of $\mi\otimes\K$. The composite functor between the fibers
\begin{displaymath}
{}_{H(C,B)}\Comod(\dc{D})\xrightarrow{(\mi\otimes\K)_{H(C,B)}}{}_{H(C,B)\otimes C}\Comod(\dc{D})\xrightarrow{(\varepsilon_B)_!}{}_B\Comod(\dc{D})
\end{displaymath}
has a right adjoint. First of all, the domain is locally presentable by \cref{_AMod_CComodLocallyPresentable}. Then, the reindexing functor 
$(\varepsilon_B)_!$ is cocontinuous by \cref{prop:Comod(D)cocomplete}. Finally, the functor $(\mi\otimes\K)_{H(C,B)}$ is cocontinuous as follows.
Both legs of the commutative diagram below create colimits by \cref{_AModmonadic_CComodcomonadic}
\begin{displaymath}
\begin{tikzcd}[column sep=.8in]
{}_{H(C,B)}\Comod(\dc{D})\ar[d]\ar[r,"{(\mi\ot K)_{H(C,B)}}"] & {}_{H(C,B)\ot C}\Comod(\dc{D})\ar[d] \\
\dc{D}_1^{H(Z,Y)}\ar[r,"{(\mi\ot K)_{H(Z,Y)}}"] & \dc{D}_1^{H(Z,Y)\ot Z}
\end{tikzcd}
\end{displaymath}
and the bottom functor preserves colimits by \cref{prop:continuous fibred 1-cells}: its corresponding total functor at the top of
\begin{displaymath}
\begin{tikzcd}[column sep=.6in]
\dc{D}_1\ar[d,"\mathfrak t"']\ar[r,"\mi\ot K"] & \dc{D}_1\ar[d,"\mathfrak t"]\\
\dc{D}_0\ar[r,"\mi\ot Z"'] & \dc{D}_0
\end{tikzcd}
\end{displaymath}
is a cocartesian functor (as a left adjoint in $\Cat^2$) and both horizontal functors preserve colimits (since they have a right adjoint for the monoidal closed $\dc{D}_1$ and $\dc{D}_0$) between
opfibrations that preserves them, since $\dc{D}$ is locally presentable (\cref{prop:lpparcomp}). 
Notice that we could obtain the same by considering the cocomplete opfibrations $\mathfrak s$ instead of $\mathfrak t$. 

Therefore we get an adjunction $\mi\otimes K\dashv\Hom_\Comod(K,\mi)$ for any comodule $K$, giving rise to
\begin{displaymath}
 \Hom_\Comod\colon\Comod(\dc{D})^\op\times\Comod(\dc{D})\to\Comod(\dc{D})
\end{displaymath}
where the comodule $\Hom_\Comod(K,L)$ for a $C_Z$-comodule $K\colon V\bular  Z$ and $D_W$-comodule $L\colon S\bular W$ is a $\Hom_{\Cmd}(C,D)$-comodule with target $H(Z,W)$. Notice that from the above theorem, we cannot deduce straight away the source of the comodule since it is not involved in the resulting commutative square.
\end{proof}

We have now set up all pieces in order to deduce that the fibration of modules over monads (\cref{prop:ModfibredoverMnd}) is 
enriched  in the monoidal opfibration of comodules over comonads (\cref{prop:ModMndmonoidal}).
This is the (co)module analogue of the relevant result for (co)monads over the category of 
objects, see \cref{thm:big1}. This result is the many-object generalization of \cref{thm:ModenrichedComod,thm:ModenrichedComodfib}, moving from monoidal categories to double categories.

\begin{thm}\label{thm:big2} (Sweedler theory for (co)modules)
Let $\dc{D}$ be a braided monoidal closed double category, which is fibrant and locally presentable.
\begin{enumerate}
 \item The category of modules $\Mod(\dc{D})$ is tensored and cotensored enriched in the category of comodules $\Comod(\dc{D})$.
 \item The fibration $\Mod(\dc{D})\to\Mnd(\dc{D})$ is enriched in the monoidal opfibration $\Comod(\dc{D})\to\Cmd(\dc{D})$.
\end{enumerate}
\end{thm}

\begin{proof}
$(1)$ We apply \cref{thm:cotensorenrich} to the functor $H\colon\Comod(\dc{D})^\op\times\Mod(\dc{D})\to\Mod(\dc{D})$ as in 
\cref{eq:ModHdouble}, induced by the monoidal closed structure of $\dc{D}$.

First of all, $\Comod(\dc{D})$ is braided monoidal when $\dc{D}$ is by \cref{prop:ModComodmonoidal}. Moreover, $H$ is an action by 
\cref{prop:ModHaction}. For the remaining condition, it is indeed the case that the functor $H^\op(\mi,N)$ for any $B_Y$-module $N\colon T\bular Y$ 
has a right adjoint, using \cref{thm:totaladjointthm} as follows. Consider the square of categories and functors
\begin{equation}\label{eq:diag}
\begin{tikzcd}[column sep=.6in]
\Comod(\dc{D})\ar[r,"{H^\op(\mi,N)}"]\ar[d] & \Mod(\dc{D})^\op\ar[d] \\
\Cmd(\dc{D})\ar[r,"{H^\op(\mi,B)}"'] & \Mnd(\dc{D})^\op 
\end{tikzcd}
\end{equation}
which commutes by definition of the functors involved. This is an opfibred 1-cell between the respective opfibrations (\cref{prop:ModfibredoverMnd}) 
since $H^\op(\mi,N)$ preserves cocartesian liftings by \cref{prop:ModHcartesian}. Moreover, under the running assumptions there is an adjunction
\begin{displaymath}
\begin{tikzcd}[sep=.5in]
\Cmd(\dc{D})\ar[r,shift left=2,"{H^\op(\mi,B)}"]\ar[r,phantom,"\bot"] & \Mnd(\dc{D})^\op\ar[l,shift left=2,"{P(\mi,B)}"] 
\end{tikzcd}
\end{displaymath}
between the base categories with counit $\varepsilon$, where $P$ \cref{eq:P} is established in the proof of \cref{thm:big1}.
Finally, the composite functor between the fibres
\begin{equation}\label{eq:compfibers}
{}_{P(A,B)}\Comod(\dc{D})\xrightarrow{H^\op(\mi,N)_{P(A,B)}}{}_{H(P(A,B),B)}\Mod(\dc{D})^\op\xrightarrow{(\varepsilon_B)_!}{}_{A}\Mod(\dc{D})^\op
\end{equation}
has a right adjoint as follows. Its domain is locally presentable \cref{_AMod_CComodLocallyPresentable}, and the reindexing 
functor $\wc{\varepsilon_B}\odot\mi$ of \cref{eq:_Mod} is continuous by \cref{Mod(D)completefibration}, hence its opposite (which is the reindexing for the opposite of the fibration) is cocontinuous.
Now the functor between the fibers is also cocontinuous, because it makes the following also commute
\begin{displaymath}
\begin{tikzcd}[column sep=.8in]
{}_{P(A,B)}\Comod(\dc{D})\ar[r,"{{H^\op(\mi,N)}_{P(A,B)}}"]\ar[d] & {}_{H(P(A,B),B)}\Mod(\dc{D})^\op\ar[d] \\
\dc{D}_1^{H(X,Y)}\ar[r,"{H^\op(\mi,N)_{H(X,Y)}}"'] & \left(\dc{D}_1^{H(H(X,Y),Y)}\right)^\op
\end{tikzcd}
\end{displaymath}
where both legs are comonadic by \cref{_AModmonadic_CComodcomonadic} and the bottom functor is cocontinuous by \cref{prop:continuous fibred 1-cells}.
In detail, its corresponding total functor at the top of
\begin{displaymath}
\begin{tikzcd}[column sep=.6in]
\dc{D}_1\ar[r,"{H^\op(\mi,N)}"]\ar[d,"{\mathfrak t}"'] & \dc{D}_1^\op\ar[d,"{\mathfrak t}^\op"] \\
\dc{D}_0\ar[r,"{H^\op(\mi,Y)}"'] & \dc{D}_0^\op
\end{tikzcd}
\end{displaymath}
makes the diagram commute by construction, and is cocartesian since the two horizontal functors have a right adjoint $(H(\mi,B),H(\mi,Y))$ in $\Cat^2$ 
($\dc{D}_0$, $\dc{D}_1$ are braided monoidal closed categories). Hence it forms an opfibred 1-cell between the cocomplete opfibration $\mathfrak t$ and the opposite of the complete fibration ${\mathfrak t}^\op$
(see \cref{prop:lpparcomp}). As adjoints, both horizontal functors are cocontinuous thus it follows that the corresponding fiberwise of the top one is also cocontinuous, and we obtain an adjoint for
\cref{eq:diag} as required.

Therefore $\Mod(\dc{D})$ is enriched in $\Comod(\dc{D})$, with enriched hom functor an adjoint 
\begin{equation}\label{eq:Q}
Q\colon\Mod(\dc{D})^\op\times\Mod(\dc{D})\to\Comod(\dc{D})
\end{equation}
of the restricted $H$ on (co)modules, where $Q(M_A,N_B)$ for an $A$-module $M\colon U\bular X$ and a $B$-module $N\colon T\bular Y$ is a 
$P(A,B)$-comodule, thus with target $H(X,Y)$. Again, this theorem does not give any information about the source of the comodule. 

Next, under these assumptions $\Comod(\dc{D})$ is a monoidal closed category by \cref{prop:Comodclosed}, therefore the enrichment admits cotensors by \cref{thm:cotensorenrich}, namely the $H(C,B)$-modules 
$H(K,N)$ for any $C$-comodule $K$ and $B$-module $N$. Finally, the enrichment also admits as tensors certain $(C\triangleright B)$-modules $K\oslash N$, since $H(K,\mi)$ 
has a left adjoint $K\oslash\mi$ as follows. Using the dual of \cref{thm:totaladjointthm}, 
observe that 
the commutative diagram
\begin{displaymath}
\begin{tikzcd}[column sep=.6in]
\Mod(\dc{D})\ar[r,"{H(K,\mi)}"]\ar[d] & \Mod(\dc{D})\ar[d] \\
\Mnd(\dc{D})\ar[r,"{H(C,\mi)}"] & \Mnd(\dc{D})
 \end{tikzcd}
\end{displaymath}
is a fibred 1-cell since $H(K,\mi)$ preserves cartesian liftings by \cref{prop:ModHcartesian}, and moreover there is an adjunction 
$C\triangleright\mi\dashv 
H(C,\mi)$ between the bases with unit $\eta\colon A\to H(C,C\triangleright A)$ where $\triangleright$ is established in the proof of \cref{thm:big1}. Furthermore, the 
composite functor 
between fibers
\begin{displaymath}
_{C\triangleright A}\Mod(\dc{D})\xrightarrow{H(K,\mi)_{C\triangleright 
A}}{}_{H(C,C\triangleright A)}\Mod(\dc{D})\xrightarrow{(\eta_A)^*}{}_{A}\Mod(\dc{D})
\end{displaymath}
has a left adjoint, because its domain is locally presentable by \cref{_AMod_CComodLocallyPresentable}, the reindexing functor is continuous like above,
and the functor between the fibers is continuous with an analogous to previous arguments as follows. By definition of the functors involved, it makes the diagram
\begin{displaymath}
\begin{tikzcd}[column sep=.8in]
{}_{C\triangleright A}\Mod(\dc{D})\ar[r,"{H(K,\mi)_{C\triangleright A}}"]\ar[d] & {}_{H(C,C\triangleright A)}\Mod(\dc{D})\ar[d] \\
\dc{D}_1^{Z\ot X}\ar[r,"{H(K,\mi)_{Z\ot X}}"'] & \dc{D}_1^{H(Z,Z\ot X)}
\end{tikzcd}
\end{displaymath}
commute, where both legs are monadic by \cref{_AModmonadic_CComodcomonadic} and the bottom functor is continuous by \cref{prop:continuous fibred 1-cells}.
Indeed, its total functor at the top of
\begin{displaymath}
\begin{tikzcd}[column sep=.6in]
\dc{D}_1\ar[r,"{H(K,\mi)}"]\ar[d,"{\mathfrak t}"'] & \dc{D}_1\ar[d,"{\mathfrak t}"] \\
\dc{D}_0\ar[r,"{H(Z,\mi)}"'] & \dc{D}_0
\end{tikzcd}
\end{displaymath}
is continuous between fibred complete fibrations as part of a right adjoint in $\Cat^2$ ($\dc{D}_1$ and $\dc{D}$ are monoidal closed), 
and \cref{prop:continuous fibred 1-cells} gives the result for the functor between the fibers.
Consequently, each $H(K,\mi)$ has a left 
adjoint $K\oslash\mi$ between the total categories, giving rise to functor of two variables $\oslash\colon\Comod(\dc{D})\times\Mod(\dc{D})\to\Mod(\dc{D})$
that gives tensors for the enrichment.

$(2)$ This follows by \cref{thm:enrichedfib}. First of all, the opfibration $\Comod(\dc{D})\to\Cmd(\dc{D})$ that will serve as the base of the 
enrichment is indeed monoidal by \cref{prop:ModMndmonoidal}. It is also the case that $\Comod(\dc{D})\to\Cmd(\dc{D})$ acts on the opfibration 
$\Mod(\dc{D})^\op\to\Mnd(\dc{D})^\op$, according to \cref{Trepresentation}, as follows. For the opposite of the fibred 1-cell 
\cref{eq:HHmodfibred} 
\begin{displaymath}
 \begin{tikzcd}
\Comod(\dc{D})\times\Mod(\dc{D})^\op\ar[r,"H^\op"]\ar[d] & \Mod(\dc{D})^\op\ar[d] \\
\Cmd(\dc{D})\times\Mnd(\dc{D})^\op\ar[r,"H^\op"] & \Mnd(\dc{D})^\op
\end{tikzcd} 
\end{displaymath}
we have that both functors labelled $H^\op$ are actions by \cref{rmk:fixedaction,prop:ModHaction} respectively. Moreover, these actions are compatible: indeed, by construction of the constraints, we have that the module maps $\psi_1$ and $\mu_1$ of \cref{eq:chimod} have as underlying monad maps the ones appearing in 
\cref{eq:actionstructuredouble}.

Finally, $(H^\op,H^\op)$ has an ordinary right parameterized adjoint as defined in \cref{eq:parameterizedCat2}, namely $(Q,P)$ in
\begin{displaymath}
 \begin{tikzcd}
\Mod(\dc{D})^\op\times\Mod(\dc{D})\ar[r,"Q"]\ar[d] & \Comod(\dc{D})\ar[d] \\
\Mnd(\dc{D})^\op\times\Mnd(\dc{D})\ar[r,"P"] & \Cmd(\dc{D})
\end{tikzcd} 
\end{displaymath}
since $H^\op(\mi,B)\dashv P(\mi,B)$ with $P$ as in \cref{eq:P}, $H^\op(\mi,N)\dashv Q(\mi,N)$ with $Q$ as in \cref{eq:Q}, and the (co)units
are above each other (since the existence of Q was established using \cref{thm:totaladjointthm} which ensures that).
Thus the fibration $\Mod(\dc{D})\to\Mnd(\dc{D})$ is enriched in the monoidal opfibration $\Comod(\dc{D})\to\Cmd(\dc{D})$.
\end{proof}

Analogously to \cref{cor:VCatenrVCocat}, the above theorem directly applies to the setting of enriched matrices, giving the desired enrichment of enriched modules in enriched comodules
in their most general form.

\begin{cor}\label{cor:bigthm2}
Suppose that $\ca{V}$ is a symmetric monoidal closed and locally presentable category. Then the category $\Mod(\VMMat)$ of two-indexed $\ca{V}$-modules is tensored
and cotensored enriched in the category $\Comod(\VMMat)$ of two-indexed $\ca{V}$-comodules. 
\end{cor}

\begin{proof}
Recall the description of the involved categories from \cref{ex:twoVMod,ex:twoVComod}. All clauses of \cref{thm:big2} are satisfied: under these assumptions, $\VMMat$ is a braided (in fact symmetric) monoidal double category
by \cref{ex:VMMatmonoidal}, fibrant by \cref{ex:VMMatfibrant}, monoidal closed as a double category by \cref{ex:VMMatclosed} and also locally presentable by \cref{prop:VMatlp}.
\end{proof}

\begin{rmk}\label{rmk:thatrmk}
Analogously to \cref{rmk:thisrmk} and using \cref{rmk:modoneobject}, the one-object case of the above theorem reduces to the enrichments that arise in the context of a symmetric monoidal closed (ordinary) category, i.e. 
\cref{thm:ModenrichedComod,thm:ModenrichedComodfib}. In more detail, \cref{thm:big2} states in particular that for a monoidal category -- 
namely a (one-object, one-vertical arrow) double category -- which is symmetric monoidal closed and locally presentable, the global category of modules is tensored and cotensored enriched in the global category
of comodules, as well as the fibration of modules over monoids is enriched in the opfibration of comodules over comonoids.
\end{rmk}

\begin{rmk}\label{rmk:VMod}
We can also view \cref{cor:bigthm2} as a many-object generalization of the enrichment of modules in comodules in a specific $\ca{V}$, using precisely $\dc{D}=\VMMat$. In fact, if one wanted
to actually establish the enrichment of the category of $\ca{V}$-modules $\VMod={}^1\Mod(\VMMat)$ in the category of $\ca{V}$-comodules $\VComod={}^1\Comod(\VMMat)$ as those
were described in \cref{ex:twoVMod,ex:twoVComod}, that would require a sub-case of \cref{thm:big2} that involves modules and comodules with fixed domains\footnote{The specific enrichment of enriched modules in enriched comodules is obtained in \cite[Thm.~7.7.7]{PhDChristina} in a more technical and less general way.}. 

In a bit more detail, we would once again apply \cref{thm:cotensorenrich}, this time to the composite functor
\begin{displaymath}
\tilde{H}\colon{}^I\Comod(\dc{D})^\op\times{}^I\Mod(\dc{D})\to{}^{H(I,I)}\Mod(\dc{D})\xrightarrow{\simeq}{}^I\Mod(\dc{D})
\end{displaymath}
where the first part is the functor $H$ between the respective fibers and the second one is the (adjoint) equivalence induced, according to \cref{lem:lemmaabove}, by the standard isomorphism $I\cong H(I,I)$ in $\dc{D}_0$. We know that ${}^I\Comod(\dc{D})$ is braided monoidal by \cref{prop:IModmonoidal}, and we can adjust \cref{prop:ModHaction} to establish that $\tilde{H}$ is an action: the structure isomorphisms $\psi_1,\mu_1$ would have a pasted 2-isomorphism on their left which would `normalize' the domains of all horizontal 1-cells to $I$.
Finally, knowing that the adjoint equivalence of course has an adjoint, it would be left to show that the restricted $H^\op(\mi,N)$ for any $B_Y$-module $N\colon I\bular Y$ has a right adjoint: mimicking the proof of \cref{thm:big2}, one considers a square
\begin{displaymath}
\begin{tikzcd}[column sep=.6in]
{}^I\Comod(\dc{D})\ar[r,"{H^\op(\mi,N)}"]\ar[d] & {}^{H(I,I)}\Mod(\dc{D})^\op\ar[d] \\
\Cmd(\dc{D})\ar[r,"{H^\op(\mi,B)}"'] & \Mnd(\dc{D})^\op 
\end{tikzcd}
\end{displaymath}
where both legs are opfibrations by \cref{prop:ZModfibredoverMnd}, the top functor preserves cocartesian liftings using an argument
as for \cref{eq:cartesianfirstvar} which holds for arbitrary domain (co)modules, the bottom functor has an adjoint $P(\mi,B)$ and the composite functor between
the fibers is precisely \cref{eq:compfibers} but now emanating from $^{\phantom{ABC}I}_{P(A,B)}\Comod(\dc{D})$ and landing on $^{H(I,I)}_{\phantom{AB}A}\Mod(\dc{D})$. The domain is locally
presentable by \cref{_AMod_CComodLocallyPresentable} and the restricted functors are still cocontinuous (since the fixed domains are of the `right shape'). \cref{thm:totaladjointthm} then applies and produces the desired adjoint. 

Hence the result follows and each $^I\Mod(\dc{D})$ is enriched in $^I\Comod(\dc{D})$, giving the enrichment of $\VMod$ in $\VComod$ as a particular case when $\dc{D}=\VMMat$.
For the existence of tensors and cotensors, we can adjust \cref{prop:Comodclosed} in order to show that ${}^I\Comod(\dc{D})$ is monoidal closed, and also again work as in the proof of \cref{thm:big2} for the adjoint of $\tilde{H}(K,\mi)$. Similarly for the enrichment of the fibration ${}^I\Mod(\dc{D})\to\Mnd(\dc{D})$ in the monoidal opfibration ${}^I\Comod(\dc{D})\to\Cmd(\dc{D})$. We don't go into these details so as to not further prolong this already extended work.

Notice that the above arguments indicate that we can obtain enrichments of any $^X\Mod(\dc{D})$ in the monoidal $^I\Comod(\dc{D})$ using an appropriate functor
\begin{displaymath}
{}^I\Comod(\dc{D})^\op\times{}^X\Mod(\dc{D})\to{}^{H(I,X)}\Mod(\dc{D})\xrightarrow{\simeq}{}^X\Mod(\dc{D})
\end{displaymath} 
since $^X\Mod(\dc{D})$ being monoidal is not needed in the above arguments --  but of course $^I\Comod(\dc{D})$ being monoidal as the enriching base
is needed (\cref{prop:IModmonoidal}). 
\end{rmk}

\bibliographystyle{alpha}
\bibliography{References}

\end{document}